\documentclass[a4paper,12pt]{article}

\usepackage{amsmath,amssymb,amsthm}
\usepackage{amsfonts}
\usepackage{multicol}
\usepackage{color}

\definecolor{mygreen}{RGB}{0,160,50}

\newtheorem{theorem}{Theorem}[section]
\newtheorem{lemma}[theorem]{Lemma}
\newtheorem{proposition}[theorem]{Proposition}

\newtheorem{e-definition}[theorem]{Definition\rm}
\newtheorem{remark}[theorem]{Remark}

\newcommand{\charI}{1 \hspace*{-1mm} {\rm l}}

\def\ignore#1{}
\def\eq{\begin{equation}}
\def\en{\end{equation}}
\def\eqa{\begin{eqnarray}}
\def\ena{\end{eqnarray}}
\def\eqs{\begin{eqnarray*}}
\def\ens{\end{eqnarray*}}

\def\bP{{\mathbb P}}
\def\bE{{\mathbb E}}
\def\re{{\mathbb R}}
\def\bPstar{\bP^U_{(\bs_{l-1},s_*),X}}
\def\non{\nonumber}
\def\rstar{\rhr^U_{(\bs_{l-1},s_*),X}}
\def\hrstar{\hrhr^U_{(\bs_{l-1},s_*),j,X}}
\def\hrostar{\hrhr^U_{(\bs_{l-1},s_*),j,X_0}}
\def\hrhr{{\hat\rhr}}

\def\tg{{\tilde g}}

\def\s{\sigma}
\def\l{\lambda}
\def\d{\delta}
\def\Ref#1{(\ref{#1})}

\def\Eq{\ =\ }
\def\Le{\ \le\ }
\def\Z{{\mathbb{Z}}}
\def\h{\eta}
\def\t{\tau}
\def\n{\nu}
\def\e{\varepsilon}
\def\f{\varphi}
\def\a{\alpha}
\def\p{\pi}

\def\nin{\noindent}

\def\law{{\mathcal L}}
\def\JJ{{\mathcal J}}

\def\AA{{\mathcal A}}
\def\half{{\tfrac12}}
\def\quarter{{\tfrac14}}

\def\pr{{\mathbb P}}
\def\ex{{\mathbb E}}
\def\Def{{\ \,:=\ \,}}
\def\hg{{\hat g}}

\def\D{\Delta}
\def\uii{^{(i)}}
\def\uij{^{(i,j)}}
\def\Blb{\left\{}
\def\Brb{\right\}}
\def\giv{\,|\,}

\def\uj{^{(j)}}

\def\hph{{\hat\varphi}}
\def\dtv{d_{{\rm TV}}}

\def\m{\mu}

\def\L{\Lambda}
\def\S{\Sigma}
\def\Blm{\left|}
\def\Brm{\right|}
\def\Po{{\rm Po\,}}
\def\tf{\tilde\f}
\def\tg{\tilde g}
\def\tff{\tilde f}
\def\slo{\sum_{l\ge0}}
\def\sldo{\sum_{l'\ge0}}
\def\bs{\mathbf{s}}
\def\ud{^{\d}}
\def\sJJ{\sum_{J\in\JJ}}

\def\tB{{\widetilde B}}
\def\bone{\charI}
\def\BHJ{Barbour, Holst \& Janson}

\def\us{^{(s)}}
\def\usi{^{(s-1)}}
\def\ej{e^{(j)}}
\def\eii{e^{(i)}}
\def\sjd{\sum_{j=1}^d}
\def\sid{\sum_{i=1}^d}
\def\Blb{\left\{}
\def\Brb{\right\}}
\def\Bl{\left(}
\def\Br{\right)}
\def\tr{{\rm Tr}\,}
\def\th{\theta}

\def\sixth{\tfrac16}
\def\nud{_n^\d}
\def\nuh{_n^\h}
\def\nudd{_n^{\d'}}
\def\nudci{_{n,1}^{\d}}
\def\nudct{_{n,2}^{\d}}

\def\XX{\mathcal X}
\def\dnud{\XX\nud}
\def\leqn{\lefteqn}
\def\tm{{\tilde m}}

\def\g{\gamma}
\def\JmaxS{J^\S_{{\rm max}}}
\def\Jmax{J_{{\rm max}}}
\def\Lbar{\overline{\L}}
\def\lmax{\l_{{\rm max}}}
\def\lmin{\l_{{\rm min}}}
\def\ttt{{\tilde \tau}}
\def\tAA{{\widetilde\AA}}
\def\uxit{_{X_1,X_2}}
\def\grrr{\beta}

\def\rhr{R}
\def\trhr{{\widetilde \rhr}}
\def\Estar{\bE_*^U}
\def\Pstar{\bP_*^U}
\def\KK{{\cal K}}

\def\k{\kappa}
\def\hht{\hat\tau}

\def\ut{^{(2)}}
\def\ui{^{(1)}}

\def\gbar{{\bar\g}}
\def\gbars{{\bar\g}(\s^2)}
\def\rmax{r_{\rm max}}

\ignore{
\addtolength{\textwidth}{1.5cm}
\addtolength{\hoffset}{-1cm}
\addtolength{\textheight}{1.5cm}
\addtolength{\voffset}{-1cm}
}
\def\nti{n\to\infty}

\def\MPP{Markov jump process}
\def\eur{e^{(r)}}
\def\ptmass{\varepsilon_1}
\def\lla{\alpha}

\def\Rh{\rho}

\def\FF{{\cal F}}
\def\ccc{\g}
\def\nS#1{|#1|_{\S}}
\def\nl#1{|#1|_1}

\def\ccc{\xi}
\def\td{{\tilde \d}}

\def\Sp{{\rm Sp}}
\def\lbar{{\bar\l}}

\def\cov{{\rm Cov\,}}
\def\corr{{\rm Corr\,}}

\def\smh{\psi}

\def\AAA{{\AA}}
\def\ABA{{\widetilde\AA}}
\def\hPi{{\widehat\Pi}}
\def\tPi{{\widetilde\Pi}}
\def\hX{{\widehat X}}
\def\tX{{\widetilde X}}
\def\hc{{\hat c}}
\def\KKA{\KK}

\def\adbg#1{#1}
\def\adbn#1{#1}

\def\adbr{}
\def\adbb{}

\begin{document}

\title{\hbox{Multivariate approximation in total variation, I:} equilibrium distributions of Markov jump processes}

\author{
\renewcommand{\thefootnote}{\arabic{footnote}}
A. D. Barbour\footnotemark[1],
\ M. J. Luczak\footnotemark[2]
\ 
 \& A. Xia\footnotemark[3]
\\
Universit\"at Z\"urich, Queen Mary University of London \\ \& University of Melbourne
}

\footnotetext[1]{Institut f\"ur Mathematik, Universit\"at Z\"urich, Winterthurertrasse 190, CH-8057 Z\"urich;
e-mail {\tt a.d.barbour@math.uzh.ch}.  Work begun while ADB was Saw Swee Hock 
Professor of Statistics at the National
University of Singapore, carried out in part at the University of Melbourne and at Monash University, and
supported in part by Australian Research Council Grants Nos DP120102728, DP120102398, DP150101459 
and DP150103588.}
\footnotetext[2]{School of Mathematical Sciences, Queen Mary University of London, Mile End Road, London E1 4NS, UK; 
e-mail {\tt m.luczak@qmul.ac.uk}.
Work carried out in part at the University of Melbourne, and supported by an EPSRC Leadership Fellowship, 
grant reference EP/J004022/2, 
and in part by Australian Research Council Grants Nos DP120102398 and DP150101459.}
\footnotetext[3]{School of Mathematics and Statistics, University of Melbourne, Parkville, VIC 3010, Australia;
e-mail {\tt a.xia@ms.unimelb.edu.au}.
Work supported in part by Australian Research Council Grants Nos DP120102398 and DP150101459.}
\maketitle

\begin{abstract}
For integer valued random variables, the translated Poisson distributions
form a flexible family for approximation in total variation, in much the same
way that the normal family is used for approximation \adbb{in} Kolmogorov
distance.  Using the Stein--Chen method, approximation can often be achieved
with error bounds of the same order as \adbb{those} for the CLT. In this paper, an analogous 
theory, again based on Stein's method, is developed in the multivariate context.
The approximating family consists of the equilibrium distributions of a collection
of Markov jump processes, whose analogues in one dimension are the 
immigration--death processes with Poisson distributions as equilibria.
The method is illustrated by providing total variation error bounds for the
approximation of the equilibrium distribution of one Markov jump process
by that of another.  
In a companion paper, it is shown how to use the method
for discrete normal approximation in~$\Z^d$.
\end{abstract}
 
{\it Running head:}  Multivariate approximation

\section{Introduction}\label{Introduction}
\setcounter{equation}{0}

The Stein--Chen method (Chen, 1975) enables the distribution of a sum~$W$ of indicator 
random variables to be approximated by a Poisson distribution in a wide variety of
circumstances.  In addition, it provides an estimate of the accuracy of the
approximation, expressed in terms of the {\it total variation distance\/}.  
Such an approximation is very valuable, since it allows the approximation of
the probability $\pr[W\in A]$ of an arbitrary subset~$A$ of~$\Z_+$ by a Poisson probability,
and not just of sets~$A$ with `nice' properties.  By contrast, the 
distance classically used for quantifying normal approximation is
the Kolmogorov distance, as in the Berry--Esseen theorem, and this
measures the largest difference between the probabilities of half lines.
Of course, this can easily be extended to (the unions of small numbers of) intervals,
but gives no information at all, for instance, about the probability that~$W$ is even.

The Poisson family of distributions is, however, too restrictive to be used as widely as
the normal distribution for approximation, because mean and variance have to be equal.  
Starting from the seminal paper of Presman~(1983), more
general approximations in total variation have been derived, using more
flexible families. In particular, for the translated Poisson family,
the Stein--Chen method can be adapted in a natural way (R\"ollin 2005, 2007), allowing for 
the possibility of treating sums of dependent indicator random variables.  What is more,
the order of the error in total variation approximation obtained in this way, using the translated 
Poisson family (Barbour \& Xia, 1999) or the discretized normal family (Fang, 2014),
need be no worse than that of the error in the normal approximation, measured using 
Kolmogorov distance. This represents a substantial gain in the scope of the approximation,
at relatively small cost.

In this paper, we aim for analogous results in higher dimensions, an undertaking of considerably
greater difficulty.  The first step is to choose a suitable family
of reference distributions.  For the Poisson distribution~$\Po(\l)$,
there is a Markov jump process, the immigration--death
process with constant immigration rate~$\l$ and unit {\it per capita\/}
death rate, whose equilibrium distribution is exactly~$\Po(\l)$,
and whose generator can be used as the corresponding Stein operator (Barbour, 1988).
Proceeding by analogy, we consider the equilibrium distributions
of more general Markov jump processes as possible reference distributions.
As in the Poisson case, 
their generators automatically yield corresponding Stein equations (\BHJ, Section~10.1). 
In addition, they come with a probabilistic representation of the solutions to the Stein equation that makes
it possible to estimate the quantities needed in exploiting the method.  
Although there is often no readily available exact 
representation of the equilibrium distributions of Markov jump processes,
they are shown in Theorem~2.3 of Part~II, under a weak irreducibility condition,
to be close in total variation to discrete multivariate normal 
distributions, provided that their spread is large.  In practice, this allows
the discrete normal family to be used instead for approximation, without any material loss
of accuracy. 

We begin with a sequence $(X_n,\,n\ge1)$ of density dependent Markov jump processes on~$\Z^d$, 
where~$X_n$ has transition rates
\eq\label{ADB-transition-rates}
   X \to X + J \quad \mbox{at rate}\quad ng^J(n^{-1}X),\qquad X\in\Z^d,\ J\in\JJ,
\en
$\JJ$ is a finite subset of~$\Z^d$, and the functions~$g^J$ are twice
continuously differentiable on~$\re^d$.  For Poisson approximation in one
dimension, we take $\JJ := \{-1,1\}$ with $g^{-1}(x) = x$ and $g^1(x) = \m$ for $x\in\re$,
giving a family of immigration--death processes~$X_n$ with equilibrium distributions~$\Po(n\m)$;
$n$ plays the part of the number of summands in the CLT.
In higher dimensions, the family is chosen to allow greater flexibility.
We initially suppose only that the equations
\eq\label{ADB-deterministic-eqns}
   \frac{d\xi}{dt} \Eq F(\xi) \Def \sJJ J g^J(\xi)
\en
have an equilibrium point~$c$, so that $F(c)=0$; that the matrix
\eq\label{ADB-A-def}
    A \Def DF(c)
\en
has eigenvalues whose real parts are all negative, making~$c$ a strongly stable
equilibrium of~\Ref{ADB-deterministic-eqns}; and that the symmetric matrix
\eq\label{ADB-sigma2-def}
   \s^2 \Def \s^2(c),\quad \mbox{where}\quad \s^2(x) \Def \sJJ JJ^T g^J(x),
\en
is positive definite.

The process~$X_n$ has generator given by
\eq\label{ADB-MPP-generator}
  (\AAA_n h)(X) \Def \sJJ ng^J(n^{-1}X) (h(X+J) - h(X))
\en
for bounded $h\colon \Z^d \to \re$.
To approximate the distribution of a random vector~$W \in \Z^d$
in total variation by the equilibrium distribution~$\Pi_n$ of~$X_n$, should it exist,
a key step in using Stein's method is to show that 
the expectation $\ex \{\AAA_n h(W)\}$  is small for a large class of bounded functions~$h$.  In our
theorems, we use the functions~$h = h_f$ that are determined by solving the Stein equation
\eq\label{ADB-MPP-Stein-eqn}
     (\AAA_n h)(X) \Eq f(X)
\en
for~$h$, given any bounded $f\colon \Z^d\to\re$.
However, for ease of use, we replace the operator~$\AAA_n$ as Stein operator by the simpler 
operator
\eq\label{ADB-approx-gen}
  \ABA_n h(w) \Def  \frac n2 \tr(\s^2\D^2h(w)) + \D h^T(w)A(w-nc), \quad w\in\Z^d,
\en 
where $c \in \re^d$, $A$ and~$\s^2$ are as in \Ref{ADB-A-def} and~\Ref{ADB-sigma2-def},
respectively; \adbb{here,
\eq\label{ADB-h-diffs-def}
   \D_j h(w) \Def h(w+\ej) - h(w);\quad \D^2_{jk} h(w) \Def \D_j(\D_k h)(w),
\en
for $1 \le j,k \le d$, where~$\ej$ denotes the $j$-th coordinate vector.}
It is shown in Theorem~\ref{ADB-generator-match} that $\ABA_n$ is
close enough to the original operator~$\AAA_n$ for our purposes.

We also define~$\S$ to be the positive definite symmetric solution of the continuous 
Lyapounov equation 
\eq\label{ADB-Sigma-eqn}
    A\S + \S A^T + \s^2 \Eq 0;
\en
see, for example, Khalil~(2002, Theorem 4.6, p.136). Now~$n\S$ turns out to be asymptotically equivalent
to the covariance matrix of our approximating distribution.  For a given random vector~$W$ whose
distribution we wish to approximate, it is thus clearly a good idea to choose $n$, $A$ and~$\s^2$ in
such a way that, solving~\Ref{ADB-Sigma-eqn},  $n\S \approx \cov W$.  There are typically
many choices of $A$ and~$\s^2$ that yield the same~$\S$ as solution of~\Ref{ADB-Sigma-eqn}, and
which one is best to use in~\Ref{ADB-approx-gen} is usually dictated
by the specific context.  Having chosen $A$ and~$\s^2$, it is shown in Theorem~\ref{ADB-MPP-A-too} that there
indeed exists a \MPP~$X_n$ as in~\Ref{ADB-transition-rates} that yields the corresponding
matrices in \Ref{ADB-A-def} and~\Ref{ADB-sigma2-def}.

Even under the condition that all the eigenvalues of~$A$ in~\Ref{ADB-A-def} have negative real parts,
the process~$X_n$ may not have an equilibrium.  However, it is shown in Barbour \& Pollett~(2012,
Section~4) that it has a quasi-equilibrium close to~$nc$, and that this is asymptotically extremely 
close to the equilibrium distribution~$\Pi\nud$ of its restriction to a $n\d$-ball around~$nc$,
whatever the value of $\d > 0$. For technical reasons, we use balls
in~$\re^d$ derived from the norm $\nS{\cdot}$ defined by
\eq\label{ADB-sigma-norm}
   \nS{Y}^2 \Def Y^T \S^{-1} Y,
\en
where~$\S$ is as defined above; we let $B_{\d,\S}(c) := \{\xi\in \re^d \colon \nS{\xi-c} \le \d\} $.
Defining 
\eq\label{Aug-cal-X-def}
   \dnud(J) \Def \bigl\{X \in \Z^d\colon\, \{X,X+J\} \subset B_{n\d,\S}(nc) \bigr\},
\en
we replace~$X_n$ with the process~$X\nud$ having transition rates
\eqa
X \to X + J \ \ \mbox{at rate}\ \
ng^J_\d(n^{-1}X) \Def \left\{ \begin{array}{ll}
               ng^J(n^{-1}X), &\ \mbox{if}\ X \in \dnud(J); \\
               0,    &\ \mbox{otherwise}, 
                 \end{array}\right.\label{ADB-transition-rates-delta}
\ena
for $X\in\Z^d$ and $J\in\JJ$, with~$\d$ to be chosen suitably small and positive;
broadly speaking, we choose~$\d$ so that~$c$ is a strongly attractive equilibrium of the 
equations~\Ref{ADB-deterministic-eqns} throughout $B_{\d,\S}(c)$.  
Then, if 
\eq\label{Aug-tilde-B-def}
         X\nud(0) \in \tB_{n,\d}(c) \Def \Z^d \cap B_{n\d,\S}(nc),
\en
it follows that~$X\nud$
is a Markov process on the finite state space~$\tB_{n,\d}(c)$, and so has an equilibrium distribution;
furthermore, if all states in~$\tB_{n,\d}(c)$ communicate, this equilibrium 
distribution~$\Pi\nud$ is unique. 
Assumptions G3 and~G4 below guarantee that this is the case: see Lemma~\ref{Dec-irreducible}.

Now, if $X\nud \sim \Pi\nud$, it follows by Dynkin's formula and because each set~$\dnud(J)$ is
bounded that $\ex\{\AA\nud h(X\nud)\} = 0$ for all functions $h\colon \Z^d \to \re$, where
\eq\label{ADB-generator-eqn}
    \AA\nud h(X) \Def  n\sJJ g^J_\d(n^{-1}X) \{h(X+J) - h(X) \},\qquad X \in \Z^d .
\en
The essence of Stein's method for total variation approximation is to find a 
function~$h_B = h_{B,n}\ud$ that solves the equation
\eq\label{ADB-Stein-eqn}
   \AA\nud h_B(X) \Eq \bone_B(X) - \Pi\nud\{B\},\quad X \in  \tB_{n,\d}(c),
\en
for each $B \subset \tB_{n,\d}(c)$.  Then, if~$W$ is any random element
of~$\Z^d$ and $B \subset \tB_{n,\d}(c)$, it follows that
\eqs
	     \pr[W \in B] - \Pi\nud\{B\} 
        &=& \ex\{(\bone_B(W) - \Pi\nud\{B\})I[W \in \tB_{n,\d'}(c)]\}\non\\
        &&\qquad\mbox{}   - \Pi\nud\{B\}\pr[W \notin \tB_{n,\d'}(c)], 
\ens
for any $\d' \le \d$, so that 
\eq\label{ADB-Stein-este}
  \dtv(\law(W),\Pi\nud) \Le \sup_{B \subset \tB_{n,\d}(c)}  
            |\ex\{\AA\nud h_B(W) I[W \in \tB_{n,\d'}(c)]\}| + \pr[W \notin \tB_{n,\d'}(c)].
\en
Showing that $\law(W)$ is close to~$\Pi\nud$ in total variation thus reduces to showing that
the right hand side of~\Ref{ADB-Stein-este} is small.
Bounding the probability $\pr[W \notin \tB_{n,\d'}(c)]$ typically involves direct
estimates, such as Chebyshev's inequality.
Thus the main effort goes into bounding $|\ex\{\AA\nud h_B(W)\}|$.  

In order to extract the essential parts of $\ex\{\AA\nud h_B(W)\}$, we expand the expression for
$\AA\nud h_B(X)$, using
Newton's expansion.  To control the remainders in the expansion, we need to be able to
control the magnitudes of the first and second differences \adbb{$\D_j h_B(X)$ and $\D^2_{jk} h_B(X)$}
\ignore{
\eq\label{ADB-h-diffs-def}
   \D_j h_B(X) \Def h_B(X+\ej) - h(X);\quad
   \D_{jk} h_B(X) \Def \D_j(\D_k h_B)(X),   
\en
}
for $1 \le j,k \le d$. We obtain bounds for these, 
given in Theorem~\ref{ADB-h-bounds}, within a ball $\nS{X - nc} \le n\d/4$, for~$\d$
small enough.  They are derived using the explicit representation
\eq\label{ADB-Poisson-eq}
   h_B(X) \Def h_{B,n}\ud(X) \Eq - \int_0^\infty (\pr[X\nud(t) \in B \giv X\nud(0) = X] - \Pi\nud\{B\})\,dt,
\en
(see Kemeny \& Snell (1960, Theorem~5.13(d); 1961, Equation~(9))), and depend on
careful analysis of the Markov process~$X\nud$.  This is carried out in Sections \ref{Xn-close-to-nc}
and~\ref{ADB-special}.
For the remainders in the expansion of $\ex\{\AA\nud h_B(W)\}$ to be small,
we also need to know that $\dtv(\law(W),\law(W+\ej))$ is small for each $1\le j\le d$,
and that $\ex\nS{W-nc}^2 \le vn$ for some constant~$v$.  
This is true if $W \sim \Pi\nud$, as is shown in
Proposition~\ref{ADB-lema2}, but needs to be proved separately for any~$W$ that is to
be approximated by~$\Pi\nud$.

As a result of these considerations, provided that $\dtv(\law(W),\law(W+\ej))$ 
is small for each $1\le j\le d$ and that $\ex\nS{W-nc}^2 \le vn$, we shall have shown, for
suitable $\d > 0$, that
$\ex\{\AA\nud h_B(W)I\nud(W)\}$ is close to $\ex\{\ABA_n h_B(W)I\nud(W)\}$, 
\adbb{where} $I\nud(X) := I[\nS{X-nc} \le n\d/3]$ and~$\ABA_n$ is as
in~\Ref{ADB-approx-gen}.
Hence, for any integer valued random vector~$W$ such that $\ex\{\ABA_n h_B(W)I\nud(W)\}$ is
uniformly small for all~$B \subset \tB_{n,\d}(c)$, $\dtv(\law(W),\law(W+\ej))$ is small for each 
$1\le j\le d$, \adbb{$\pr[W \notin \tB_{n,\d'}(c)]$ is small,} and $\ex\nS{W-nc}^2 \le vn$, it follows from~\Ref{ADB-Stein-este} that 
$\dtv(\law(W),\Pi\nud)$ is small. The precise statement of this conclusion, giving a set of
quantities that bound $\dtv(\law(W),\Pi\nud)$ for an arbitrary integer valued random $d$-vector $W$, 
is presented in Theorem~\ref{ADB-first-approx-thm}.  An application is given in Section~\ref{Sect-MPP}.

\section{The analysis of~$X\nud$: general processes}\label{Xn-close-to-nc}
\setcounter{equation}{0}

\subsection{Main assumptions}\label{ADB-assumptions}
The main arguments of the paper are based on the analysis of a sequence of Markov jump processes~$X_n$,
whose transition rates are given in~\Ref{ADB-transition-rates}.  For some $\d_0 > 0$, 
we make the following assumptions.

\medskip
\begin{description}
\item[Assumption G0.] 
The equations~\Ref{ADB-deterministic-eqns} have an equilibrium~$c$; thus \hbox{$F(c)=0$.} 

\item[Assumption G1.] 
All eigenvalues of the matrix $A := DF(c)$ have negative real parts.

\item[Assumption G2.]  
For each $J \in \JJ$, the function $g^J$ is of class~$C^2$ in the 
  \adbb{Euclidean ball} $B_{\d_0}(c) := \{x\colon |x-c| \le \d_0\}$.

\item[Assumption G3.] 
There exists $\e_0 > 0$ such that
\[
    \inf_{x \in B_{\d_0}(c)} g^J(x) \ \ge\ \e_0 g^J(c)\ =:\ \m_0^J \ >\ 0,\quad J\in\JJ. 
\]

\item[Assumption G4.] 
For each unit vector~$\ej \in \re^d$, $1\le j\le d$, there exists a finite sequence of elements
$J_{1}\uj,\ldots,J_{r(j)}\uj$ of~$\JJ$ such that
\[
    \ej \Eq \sum_{l=1}^{r(j)} J_l\uj.
\]
\end{description}

For $d$-vectors, we use $|\cdot|$ to denote the Euclidean norm, $|\cdot|_1$ to denote
the $\ell_1$-norm, and $|X|_\S$ to denote $|\S^{-1/2}X|$.  For
a $d\times d$ matrix $B$, 
we define the spectral norm
\[
     \|B\| \Def \sup_{y\in\re^d\colon|y| = 1} |By|,
\]
and use $\|B\|_1$ to denote $\sid\sjd |B_{ij}|$.  Note that, for any $d$-vector $b$ and $d\times d$
matrix~$B$, the inequalities 
\[
    d^{-1}|b|_1 \Le \sqrt{d^{-1}b^Tb} \quad \mbox{and}\quad d^{-2}\|B\|_1 \Le 
    \sqrt{d^{-2}\tr(B^TB)} \Le \sqrt{d^{-1}\|B\|^2}
\]
yield
\eq\label{Aug-norm-comparisons}
  |b|_1 \Le d^{1/2}|b| \quad \mbox{and} \quad \|B\|_1 \Le d^{3/2} \|B\|.
\en 
\adbr{For a $d\times d$ positive definite symmetric matrix~$M$, 
 we write~$\lbar(M)$ for $d^{-1}\tr(M)$, $\lmin(M)$ and~$\lmax(M)$ for its smallest and
largest eigenvalues, respectively,  and $\Rh(M) := \lmax(M)/\lmin(M)$ for its condition number;
we use $\Sp'(M)$ to denote the triple $(\lbar(M),\lmin(M),\lmax(M))$.}  

For a real function $h\colon \Z^d \to \re$, we define 
\[
  \|\D h(X)\|_\infty \Def \max_{1\le i\le d}|\D_i h(X)|;\quad
           \|\D^2 h(X)\|_\infty \Def \max_{1\le i,j\le d}|\D_{ij} h(X)|.
\]
For any $a>0$, we then set
\eqa
   \|h\|_{a,\infty} &:=&  \max\{|h(X)|\colon X \in \Z^d, |X-nc| \le a\};\non\\
   \|\D h\|_{a,\infty} &:=&  \max\{\|\D h(X)\|_\infty\colon X \in \Z^d, |X-nc| \le a\};\label{ADB-norm-E-defs}\\
   \|\D^2 h\|_{a,\infty} &:=& \max\{\|\D^2 h(X)\|_\infty\colon X \in \Z^d, |X-nc| \le a\},
       \non
\ena
for~$c$ as above. For $g\colon \re^d \to \re$ twice differentiable, we set
\[
    \|D^2g(x)\| \Def \limsup_{t\to0}\sup_{y\colon|y|=1} t^{-1}|Dg(x+ty) - Dg(x)|,
\]
\adbr{where~$D$ denotes the differential operator.}  We then define the quantities
\eq\label{ADB-L-defs}
  L_0 \Def \max_{J\in\JJ}\frac{|g^J|_{\d_0}}{g^J(c)};\quad
  L_1 \Def \max_{J\in\JJ}\frac{|D g^J|_{\d_0}}{g^J(c)};\quad
  L_2 \Def \max_{J\in\JJ}\frac{\|D^2 g^J\|_{\d_0}}{g^J(c)},
\en
finite in view of Assumptions~G2 and~G3, where $\|H\|_{\d} := \sup_{x\in B_\d(c)} \|H(x)\|$,
for any vector- or matrix-valued function~$H$ and for any choice of norm~$\|\cdot\|$.

We also define 
\eq\label{ADB-Lambda-def}
  \begin{array}{ll}
    \L \Def \sJJ g^J(c) |J|^2 \Eq \tr(\s^2); \qquad &\g \Def \sJJ g^J(c) |J|^3;\\[1ex]
     \Jmax \Def \max_{J\in\JJ} |J|; \qquad &\JmaxS \Def \max_{J\in\JJ} |\S^{-1/2}J|;\\[1ex]
     \s^2_\S \Def \S^{-1/2}\s^2\S^{-1/2};\qquad &\lla_1 \Def \half\lmin(\s^2_\S);\\[1ex]
      \Lbar \Def \lbar(\s^2) \Eq d^{-1}\L;\qquad\ \gbar \Def d^{-3/2}\g;\qquad &\m_* \Def \min_{J\in\JJ}\m_0^J,
  \end{array} 
\en
where $\s^2$ is defined in~\Ref{ADB-sigma2-def}, and~$\S$ in~\Ref{ADB-Sigma-eqn}.
In the sections that follow, we establish many bounds that depend on these basic
parameters.  They are mainly expressed as continuous functions of the elements
of the set
\eq\label{ADB-key-quantities}
  \KK \Def \{L_0,\,L_1,\,L_2,\, \e_0,\, \Sp'(\s^2/\Lbar),\,  \Sp'(\S),\, d^{-1}\Jmax,\,  \|A\|/\Lbar,\, 
     \adbr{\Lbar/\m_*,}\,\d_0\},
\en
and, with slight abuse of notation, are said to belong to the set~$\KK$.
If they are also continuous functions of another parameter, such as~$\d$, they are said to belong
to~$\KK(\d)$.   The $\Lbar$-factors ensure that the quantities
remain invariant if all the transition rates~$g^J$ are multiplied by the same constant.
In particular, constants
of the form $\k_i$ and~$K_i$ belong to~$\KK$, and the implied constants in any order
expressions also belong to~$\KK$.  

The $d$-dependence in $\lbar(\s^2)$ and~$d^{-1}\Jmax$ is put in to ensure that
the quantities do not automatically have to grow with the dimension~$d$. 
It is chosen in this way for the latter in view of
Lemma~\ref{ADB-innovations-MPP}, and for the former by comparison with $\s^2 = I$.
In order to avoid many provisos in the bounds, we shall assume throughout that $d \le n^{1/4}$,
which is ultimately no restriction, since our bounds are typically of no use 
unless~$d$ is rather smaller than~$n^{1/7}$.

We note two immediate consequences of Assumptions G3 and~G4.

\begin{lemma}\label{Dec-irreducible}
Assumptions G3 and~G4 imply that $\s^2$ is positive definite, and that, 
for any $\d > 0$, there exists $n_{\ref{Dec-irreducible}}(\d) < \infty$
such that the process~$X_n\ud$ is irreducible on~$\tB_{n,\d}(c)$, defined in~\Ref{Aug-tilde-B-def},
as long as $n \ge n_{\ref{Dec-irreducible}}(\d)$. 
\end{lemma}

\nin{\it Proof.}\
For the first statement, if $x^T\s^2x = 0$, then $x^TJ = 0$ for all $J\in\JJ$,
because of Assumption~G3.  This, 
from Assumption~G4, implies that $x^T\ej = 0$ for all $1\le j\le d$, so that $x=0$.

For the second statement,
setting $\rmax := \max_{1\le j\le d}r(j)$, it is immediate that, under the transitions for the
Markov process~$X_n\ud$, the states $X$ and~$X\pm\ej$ communicate,
for all $1\le j\le d$, as long as $\nS{X-nc} < n\d - \rmax\JmaxS$. Hence, starting from an~$X$ with
$\nS{X-nc} \le \max_{1\le j\le d} \nS{\ej}$, it follows that all states~$X$ with 
$\nS{X-nc} < n\d - \rmax\JmaxS$ intercommunicate. 

For the remainder, we note that,
because the set~$\JJ$ is finite, the infimum $\inf_{u\in\re^d\colon \nS{u}=1} \min_{J\in\JJ} u^T \S^{-1} J$
is attained at some~$u_*$.  Then $\min_{J\in\JJ} u_*^T \S^{-1} J \ge 0$ together with $F(c)=0$ would
imply that $u_*^T \S^{-1} J = 0$ for all $J\in\JJ$; and this is impossible, as argued above.  
Hence there exists $k_*>0$ such that, for all~$u$ with $\nS{u}=1$,
$\min_{J\in\JJ}u^T \S^{-1} J < -k_*$; without loss of generality, we can also take $k_* \le 1$.

Taking any $X$ with $\nS{X-nc} \le n\d$, write $X - nc = xu$, for $u \in \re^d$ with $\nS{u} = 1$
and $x \ge 0$.
Then, noting that $\sqrt{1-y} \le 1 - y/2$ in $0\le y\le 1$, we have
\eqs
   \min_{J\in\JJ} \nS{X+J - nc} 
           &=&  \min_{J\in\JJ} \Bigl\{\nS{X-nc}^2 + 2(X-nc)^T \S^{-1} J + \nS{J}^2 \Bigr\}^{1/2} \\
           &\le&   x\Bigl\{1 - 2x^{-1}k_* + x^{-2}\{\JmaxS\}^2 \Bigr\}^{1/2} \\
           &\le& x - k_*/2,
\ens
provided that $x \ge \max\{k_*,\{\JmaxS\}^2/k_*\}$.  Thus each state with $\nS{X-nc} \le n\d$ communicates with
some state~$X'$ for which $\nS{X'-nc} \le \nS{X-nc} - k_*/2$, and hence, repeating this step, with one
such that $\nS{X-nc} < n\d - \rmax\JmaxS$.  Combining these results, we see that~$X_n\ud$ is irreducible,
provided that 
\[
   \phantom{XXX}
      n\ \ge\ n_{\ref{Dec-irreducible}}(\d) \Def \d^{-1}\bigl\{(\rmax+1)\JmaxS + \max\{k_*\adbb{,}\{\JmaxS\}^2/k_*\}\bigr\}.
           \phantom{XXX}\qed
\]

If Assumption~G4 is {\it not\/} satisfied, then the lattice generated by the jumps in~$\JJ$ is a proper
sub-lattice of~$\Z^d$. 

\subsection{$X\nud$ stays close to~$nc$}\label{ML-close-to-nc}
In this section, we show that, whatever its initial value~$X\nud(0)$, the process~$X\nud$
rapidly gets close to~$nc$.  Thereafter,
it remains close to~$nc$ with high probability for a very long time.
To formulate our results, we define the hitting times
\eqa
    \t\nud(\h) \Def \inf\{u\ge0\colon\, \nS{X\nud(u) - nc} \ge n\h\}; \non\\
     \ttt\nud(\h) \Def  \inf\{u\ge0\colon\, \nS{X\nud(u) - nc} \le n\h\},           
                   \label{ADB-tau-delta-def}
\ena
for any $0 < \h \le \d \le \d_0$.

We begin by establishing some Lyapunov--Foster--Tweedie drift conditions, showing that~$X\nud$ 
has a strong tendency to drift towards~$nc$ in the $\nS{\cdot}$ norm.

\begin{lemma}\label{ML-drift-lemma-1}
Let~$X_n$ be a sequence of Markov jump processes,
whose transition rates are given in~\Ref{ADB-transition-rates},
and such that Assumptions G0--G4 are satisfied.  Define
\eqs
 h_0(X) &:=& (X-nc)^T\S^{-1}(X-nc) \Eq \nS{X-nc}^2; \\
 h_\th(X) &:=& \exp\{n^{-1}\th h_0(X)\},\ \th > 0.
\ens
Then there exist positive constants $K_{\ref{ML-drift-lemma-1}}, \d_{\ref{ML-drift-lemma-1}}$ and~$\th_1$ in~$\KK$ 
and $\d'_{\ref{ML-drift-lemma-1}}(d) \in \KK(d)$
such that, for any $\d \le \min\{\d_{\ref{ML-drift-lemma-1}},\d'_{\ref{ML-drift-lemma-1}}(d)\}$
and any $X \in \tB_{n,\d}(c)$ with $\nS{X-nc} \ge  K_{\ref{ML-drift-lemma-1}}\sqrt{nd}$, we have
$$
   \AA\nud h_0(X) \Le -\lla_1 h_0(X) ;\quad 
    \AA\nud h_\th(X) \Le  -\half n^{-1} \lla_1 \th h_0(X) h_\th(X),\ 0 < \th \le \th_1;
$$
for the latter inequality, we also require that $n\ge n_{\ref{ML-drift-lemma-1}} \in \KK$.
The quantities $K_{\ref{ML-drift-lemma-1}}, \d_{\ref{ML-drift-lemma-1}}$, 
$\d'_{\ref{ML-drift-lemma-1}}(d)$ and~$\th_1$  are given in \Ref{ADB-K1-def}, \Ref{ADB-delta1-def} 
and~\Ref{ADB-theta1-def}.
\end{lemma}

\begin{proof}
 It is immediate that, for the above choice of~$h_0$,
\[
   h_0(X+J) - h_0(X) \Eq J^T\S^{-1}(X-nc) + (X-nc)^T\S^{-1}J + J^T\S^{-1}J.
\]
Multiplying by $ng_\d^J(x)$, where~$x := n^{-1}X$, and adding over~$J$, we have
\eq\label{ML-A-Xsquared}
   \AA\nud h_0(X) \Eq n\{F(x)^T\S^{-1}(X-nc) + (X-nc)^T\S^{-1}F(x) + \tr(\S^{-1}\s^2(x))\},
\en
as long as $\nS{X-nc} < n\d - \JmaxS$, where~$F$ is as defined in~\Ref{ADB-deterministic-eqns},
and~$\s^2$ as in~\Ref{ADB-sigma2-def}. For $\nS{X-nc} \ge n\d - \JmaxS$, the 
truncation~\Ref{ADB-transition-rates-delta} may change this expression: see below. 
Now, using~\Ref{ADB-L-defs}, for $x,y \in B_{\d_0}(c)$, we have
\eqa
   |F(x) - F(y) - A(x-y)| &\le& \frac12\sJJ |J| g^J(c) L_2 |x-y|\{|x-y| + 2|y-c|\} \non\\
   &\le&  \L L_2|x-y|\{|x-y| + |y-c|\}. \label{ML-F-approx}
\ena
Substituting \Ref{ML-F-approx}, with $y=c$, into~\Ref{ML-A-Xsquared}, and 
using~\Ref{ADB-Sigma-eqn}, we have
\eqa
   \AA\nud h_0(X) &\le& -(X-nc)^T \S^{-1}\s^2\S^{-1}(X-nc) + n\tr(\S^{-1}\s^2(x)) \non\\
     &&\mbox{}\hskip0.3in   + 2\L L_2 n^{-1}\|\S^{-1/2}\|\,|X-nc|^2\, \nS{X-nc}.  
    \label{ML-first-drift-bound}
\ena
Using the inequalities
\eqa
    &&(X-nc)^T \S^{-1}\s^2\S^{-1}(X-nc) \ \ge\ \lmin(\s^2_\S)\, \nS{X-nc};\non\\
   &&\lmin(\S)\, \adbg{\nS{X-nc}^2} \Le |X-nc|^2 \Le \lmax(\S)\, \nS{X-nc}^2, 
  \label{ML-quad-lower-bound}
\ena
it first follows that $(X-nc)^T \S^{-1}\s^2\S^{-1}(X-nc) \ge 2\lla_1 \nS{X-nc}^2$.
Then
\eq\label{ADB-trace-inequality}
   n\tr(\S^{-1}\s^2(x)) \Le nL_0 \tr(\s^2_\S) \Le \half \lla_1 \nS{X-nc}^2
\en
if $\nS{X-nc} \ge K_{\ref{ML-drift-lemma-1}} \sqrt{nd}$, where
\eq\label{ADB-K1-def}
   K_{\ref{ML-drift-lemma-1}}^2 \adbr{\Def \frac{2L_0}{d\lla_1}\,\tr(\s^2_\S)} \Le 4 L_0 \Rh(\s^2)\Rh(\S),
\en 
since $(1/\adbb{2}d\lla_1)\tr(\s^2_\S) \le \Rh(\s^2_\S) \le \Rh(\s^2)\Rh(\S)$.
Finally,
\eq\label{ADB-finally}
   2\L L_2 n^{-1}\|\S^{-1/2}\|\,|X-nc|^2\, \nS{X-nc} \Le \half \lla_1 \nS{X-nc}^2
\en
if $\nS{X-nc} \le n\, \min\{\d_{\ref{ML-drift-lemma-1}},\d'_{\ref{ML-drift-lemma-1}}(d)\}$, where
\eq\label{ADB-delta1-def}
   \d_{\ref{ML-drift-lemma-1}} \Def 
            \frac{\d_0}{\sqrt{\lmax(\S)}};\quad \d'_{\ref{ML-drift-lemma-1}}(d) \Def \frac{1}{d}\,
              \frac{\lla_1\sqrt{\lmin(\S)}}{4\Lbar L_2 \lmax(\S)}.
\en 
This proves the first part of the lemma for all~$X$ such that 
$\nS{X-nc} < n\d - \JmaxS$.

If $n\d - \JmaxS \le \nS{X-nc} \le n\d$, we may have $g^J(n^{-1}X) > g_\d^J(n^{-1}X) = 0$
for some~$J$.  However, from the definition of~$h_0$, these~$J$ represent transitions for
which $h_0(X+J) - h_0(X) > 0$, and replacing $g^J(n^{-1}X)$ by zero makes the value of $\AA\nud h_0(X)$
even smaller than that given in~\Ref{ML-A-Xsquared}, and hence preserves the 
inequality~\Ref{ML-first-drift-bound}.

For the second part, taking $\d \le \d_{\ref{ML-drift-lemma-1}}$, we note that 
$e^x - 1 \le x + x^2$ in $x \le 1$.
Now, for $\JmaxS \le n\d_{\ref{ML-drift-lemma-1}}$ and $\nS{X-nc} \le n\d_{\ref{ML-drift-lemma-1}}$, we have
\eqs
   \frac\th n|h_0(X+J) - h_0(X)| &\le&  \frac\th n\{2\JmaxS\,\nS{X-nc} + (\JmaxS)^2\} 
   \Le   3\th\JmaxS\d_{\ref{ML-drift-lemma-1}},
\ens
and $\JmaxS \le n\d_{\ref{ML-drift-lemma-1}}$ if 
$n \ge (d^{-1}\JmaxS/\d_{\ref{ML-drift-lemma-1}})^{4/3} =: n_{\ref{ML-drift-lemma-1}}$, 
because $n\ge d^4$. Hence
it follows that $n^{-1}\th|h_0(X+J) - h_0(X)| \le 1$ for all $X \in \tB_{n,\d}(c)$, if~$\th\le \th_1$,
$n \ge n_{\ref{ML-drift-lemma-1}}$ and 
\eq\label{Oct-theta-1.0}
    \th_1 \JmaxS \d_{\ref{ML-drift-lemma-1}} \Le 1/3;
\en
note that then $d\th_1 \in \KK$.
Then, for~$X$ such that $\nS{X-nc} < n\d - \JmaxS$, and with $x := n^{-1}X$,
\[
   \AA\nud h_\th(X) \Eq nh_\th(X) \sJJ g^J(x)\bigl\{e^{n^{-1}\th(h_0(X+J) - h_0(X))} - 1 \bigr\}.
\]
Hence, if $\nS{X-nc} < n\d - \JmaxS$, we have
\eqs
  \leqn{ n\sJJ g^J(x)\bigl\{e^{n^{-1}\th(h_0(X+J) - h_0(X))} - 1 \bigr\} } \\
   &&\Le n^{-1}\th \AA\nud h_0(X) +  n \sJJ g^J(x) n^{-2}\th^2 |h_0(X+J) - h_0(X)|^2.
\ens
Since
\eqs
   |h_0(X+J) - h_0(X)|^2 &\le& \{2\nS{X-nc}\nS{J} + \nS{J}^2\}^2 \\
    &\le& \nS{J}^2(8\nS{X-nc}^2 + 2(\JmaxS)^2),
\ens
it follows in turn that, if $\d \le \min\{\d_{\ref{ML-drift-lemma-1}},\d'_{\ref{ML-drift-lemma-1}}(d)\}$, then  
\eqs
\leqn{ n\sJJ g^J(x)\bigl\{e^{n^{-1}\th(h_0(X+J) - h_0(X))} - 1 \bigr\} } \\
   &&\Le -n^{-1}\th \lla_1 h_0(X) + 2n^{-1}\th^2 L_0\tr(\s^2_\S)\{4h_0(X)  + (\JmaxS)^2\},
\ens 
if $\th \le \th_1$.  But now, if~$\th_1$ is also chosen so that 
\eq\label{Oct-theta-1.1}
    8d\th_1 L_0\lmax(\s^2_\S) \Le \quarter\lla_1 \Eq \tfrac18 \lmin(\s^2_\S), 
\en
we have $8\th^2 L_0\tr(\s^2_\S)h_0(X) \le \quarter\th\lla_1 h_0(X)$, and if
\eq\label{Oct-theta-1.2}
    2 d\th_1 L_0 \lmax(\s^2_\S) \, (\JmaxS)^2  
        \Le \adbr{\quarter \lla_1\,d K_{\ref{ML-drift-lemma-1}}^2 ,}
\en
and $\nS{X-nc} \ge K_{\ref{ML-drift-lemma-1}}\sqrt{nd}$, we have
$2\th^2 L_0\tr(\s^2_\S)(\JmaxS)^2 \le \quarter \th\lla_1 h_0(X)$ also, so that then
\eq\label{ML-second-MG}
   n\sJJ g^J(x)\bigl\{e^{n^{-1}\th(h_0(X+J) - h_0(X))} - 1 \bigr\}
         \Le -\half n^{-1} \lla_1 \th h_0(X).
\en 
Note that \Ref{Oct-theta-1.0}, \Ref{Oct-theta-1.1} and~\Ref{Oct-theta-1.2} are satisfied by choosing
\eq\label{ADB-theta1-def}
    d\th_1 \Eq \min\{1/(\adbb{3}d^{-1}\JmaxS \d_{\ref{ML-drift-lemma-1}}), \adbr{1/(64 L_0 \Rh(\s^2)\Rh(\S))},
         1/4(d^{-1}\JmaxS)^2 \} \ \in\ \KK,
\en
since we assume that $n \ge d^4$.
As for the first part, if 
$n\d - \JmaxS \le \nS{X-nc} \le n\d$, the inequality~\Ref{ML-second-MG} is 
still true, completing the
proof of the second statement of the lemma.
\end{proof}

\begin{remark}\label{ADB-delta1=delta0}
{\rm
 If the functions~$g^J$ are {\it linear\/} within $B_{\d_0,\S}$, then $L_2=0$, and we 
can take $\min\{\d_{\ref{ML-drift-lemma-1}},\d'_{\ref{ML-drift-lemma-1}}(d)\} = 
\d_{\ref{ML-drift-lemma-1}} = \d_0/\sqrt{\lmax(\S)}$.
}
\end{remark}

The first of the drift inequalities in Lemma~\ref{ML-drift-lemma-1} is now used to show 
that~$X\nud$ quickly reaches even small balls around~$nc$, if 
$\d \le \min\{\d_{\ref{ML-drift-lemma-1}},\d'_{\ref{ML-drift-lemma-1}}(d)\}$.

\begin{lemma}\label{ML-time-to-root-n}
Let~$X_n$ be a sequence of Markov jump processes,
whose transition rates are given in~\Ref{ADB-transition-rates},
and such that Assumptions G0--G4 are satisfied.
Let~$\lla_1$ be as in~\Ref{ADB-Lambda-def} and $K_{\ref{ML-drift-lemma-1}}$, 
$\d_{\ref{ML-drift-lemma-1}}$ and $\d'_{\ref{ML-drift-lemma-1}}(d)$ as in Lemma~\ref{ML-drift-lemma-1}.
Then, if $\d \le \min\{\d_{\ref{ML-drift-lemma-1}},\d'_{\ref{ML-drift-lemma-1}}(d)\}$
and $\h > \max\{K_{\ref{ML-drift-lemma-1}}\sqrt{d/n},2n^{-1}\JmaxS\}$, we have
\[
   \pr[\ttt\nud(\h) > t  \giv X\nud(0) = X_0] \Le  4(n\h)^{-2}\nS{X_0-nc}^2\, e^{-\lla_1t}.
\]
\end{lemma}

\begin{proof}
As before,
let $h_0(X) := \nS{X-nc}^2$, and define $M_0(t) := h_0(X\nud(t))e^{\lla_1 t}$.  Then
it follows from the first part of Lemma~\ref{ML-drift-lemma-1}, by a standard argument, that 
$M_0(t\wedge\ttt\nud(K_{\ref{ML-drift-lemma-1}}\sqrt{d/n}))$, $t\ge0$, is a non-negative supermartingale
\adbr{with respect to the filtration~$\adbb{\FF^{X\nud} :=} (\FF_t^{X\nud},\,t\ge0)$ generated by~$X\nud$.}  
This implies that
\eqs
  \leqn{(n\h - \JmaxS)^2 \ex\bigl\{e^{\lla_1 \ttt\nud(\h)} \bone\{\ttt\nud(\h) \le t\} 
                  \giv X\nud(0) = X_0\bigr\}} \\
           &&  \Le \ex\{M_0(t\wedge\ttt\nud(\h)) \giv X\nud(0) = X_0\} \Le h_0(X_0),
\ens
since $h_0(X\nud(\ttt\nud(\h))) \ge (n\h - \JmaxS)^2$, because the jumps of~$X\nud$
are bounded in $\S$-norm by~$\JmaxS$.
Letting $t\to\infty$, we have
\[
    \ex\bigl\{e^{\lla_1 \ttt\nud(\h)}\giv X\nud(0) = X_0\bigr\} 
               \Le \Blb\frac{\nS{X_0-nc}}{n\h - \JmaxS}\Brb^2.
\]
The lemma now follows immediately.
\end{proof}

The second drift inequality in Lemma~\ref{ML-drift-lemma-1} implies that the process~$X\nud$ takes
a long time to get far away from neighbourhoods of~$nc$.  For use in what follows, we define
\eq\label{Oct-psi-def}
   \smh(n) \Def 4\sqrt{\frac{\log n}{(d\th_1)n^{3/4}}} \quad\mbox{and}\quad
              \smh^{-1}(\h) \Def \min\{n\ge4\colon\,\smh(n) \le \h\}.
\en

\begin{lemma}\label{ADB-lema3}
Let~$X_n$ be a sequence of Markov jump processes,
whose transition rates are given in~\Ref{ADB-transition-rates},
and such that Assumptions G0--G4 are satisfied.
Then there exists $K_{\ref{ADB-lema3}} \in \KK$ such that, for all 
$\h \le \d \le \min\{\d_{\ref{ML-drift-lemma-1}},\d'_{\ref{ML-drift-lemma-1}}(d)\}$ and for~$\th_1$ as in
Lemma~\ref{ML-drift-lemma-1}, we have
\[
   \pr[\t\nud(\h) \le t \giv X\nud(0) = X_0] 
           \Le (n K_{\ref{ADB-lema3}} \L t + \exp\{n^{-1}\th_1\nS{X_0-nc}^2\})e^{-n\th_1\h^2},
\]
if $n \ge n_{\ref{ML-drift-lemma-1}}$.
In particular, for any $\d \le \min\{\d_{\ref{ML-drift-lemma-1}},\d'_{\ref{ML-drift-lemma-1}}(d)\}$,
for any $\h \le \d$, and for any~$T>0$, there exists $n_{\ref{ADB-lema3}}(T) \in \KK(\Lbar T)$
such that, for all $\nS{X_0-nc} \le n\h/2$ and $t \le T$, we have
\[
   \pr[\t\nud(3\h/4) \le t \giv X\nud(0) = X_0] \Le 2n^{-4},
\]
as long as~$n \ge \max\{n_{\ref{ADB-lema3}}(T),\smh^{-1}(\h)\}$.
The quantities $K_{\ref{ADB-lema3}}$ and~$n_{\ref{ADB-lema3}}(T)$ are defined in \Ref{Dec-K2.4-def} 
and~\Ref{Dec-n2.4(T)-def}, respectively.  
\end{lemma}

\nin {\it Proof.}\     
 It follows from the second part of Lemma~\ref{ML-drift-lemma-1} that, for $0 \le \th \le \th_1$,
\[
 M_\th(t) \Def h_\th(X\nud(t)) - H_\th \int_0^t \bone\{\nS{X\nud(s)-nc} \le K_{\ref{ML-drift-lemma-1}}\sqrt{nd}\}\,ds
\]
is an \adbb{$\FF^{X\nud}$-}supermartingale, where
\[
    H_\th \Def \max_{X\in\Z^d\colon \nS{X-nc} \le K_{\ref{ML-drift-lemma-1}}\sqrt{nd}} \AA\nud h_\th(X).
\]
Clearly, recalling $n \ge d^4$, $H_\th$ is bounded by
\eq\label{ADB-H-th-bound}
    n\sJJ \|g^J\|_{\d_0}\exp\{n^{-1}\th[K_{\ref{ML-drift-lemma-1}}\sqrt{nd} + \JmaxS]^2\} \Le n\L K_{\ref{ADB-lema3}},
\en
for 
\eq\label{Dec-K2.4-def}
   K_{\ref{ADB-lema3}} \Def L_0 \exp\{\th_1[K_{\ref{ML-drift-lemma-1}} + d^{-1}\JmaxS]^2\} \in \KK.
\en  
By the optional stopping theorem,
applied to $M_\th(\min\{t,\t\nud(\h)\})$, it thus follows that
\[
    e^{n\th\h^2}\pr[\t\nud(\h) \le t \giv X\nud(0) = X_0] - n\L K_{\ref{ADB-lema3}} t \Le \exp\{n^{-1}\th\nS{X_0-nc}^2\},
\]
proving the first claim.  The second follows for $n \ge \max\{n_{\ref{ADB-lema3}}(T),\smh^{-1}(\h)\}$, where
\eq\label{Dec-n2.4(T)-def}
   n_{\ref{ADB-lema3}}(T) \Def \max\{K_{\ref{ADB-lema3}}\Lbar T, n_{\ref{ML-drift-lemma-1}}\}, 
\en
since, for such choices of~$n$, 
\[
   nK_{\ref{ADB-lema3}}\L T \Le n^{9/4} \Le n^4 \Le e^{n\th_1\h^2/4}, 
                \quad\mbox{and thus}\quad e^{-5n\th_1\h^2/16} \Le n^{-4}.   \quad\qed
\] 

\section{The analysis of $X\nud$: elementary processes}\label{ADB-special}
\setcounter{equation}{0}
In this section, we conduct a more detailed analysis of the Markov jump processes~$X\nud$. 
\adbr{The results that follow are used to bound the solution
to the Stein equation~\Ref{ADB-Stein-eqn} and its differences, using the representation 
given in~\Ref{ADB-Poisson-eq}; this is an essential step in proving our
approximation theorem.  In order to find  \MPP es that yield a given pair $A,\s^2$, we only
need to consider ones 
whose transition rates satisfy more restrictive conditions than Assumptions G0--G4; we refer to them
as {\it elementary\/} (sequences of) processes. Since this simplifies some of the coming arguments, we 
conduct them within the
context of elementary processes, though analogous results hold under the previous assumptions: 
see Remark~\ref{ADB-J-jumps-in-Th3.1}. 
We retain Assumptions G0 and~G1, replacing the remainder with the Assumptions S2--S4 below.}

\medskip
\begin{description}
\item[Assumption S2.]  
The set $\JJ$ contains the vectors $\{\pm \ej,\, 1\le j\le d\}$.

\item[Assumption S3.] 
The transition rates~$g^J(x)$ are constant in $B_{\d_0}(c)$, for all $J \in \JJ \setminus 
       \{\ej,\,1\le j\le d\}$.

\item[Assumption S4.] 
For $1\le j\le d$,
$g^{\ej}(x)$ is linear and satisfies $g^{\ej}(x) \ge \half g^{\ej}(c)$ in $x \in B_{\d_0}(c)$.
\end{description}

\adbr{
Defining $I\uj := \{i\colon 1\le i\le d, A_{ij} \neq 0\}$, $1\le j\le d$, we write 
\eq\label{ADB-new-gdef}
 g\uj \Def g^{-\ej}(c),\quad G\uj \Def \sum_{i \in I\uj} g^{\eii}(c), 
            \quad g_* \Def \min_{1\le j\le d}(g\uj\wedge G\uj),
\en
observing that $G\uj \le \L$, $1\le j\le d$.}
We retain the definitions~\Ref{ADB-L-defs},
noting that, \adbr{for elementary processes,} $L_2 = 0$ and that $L_0 \le 3/2$, and that~$\e_0$ as 
defined in Assumption~G3 can be taken to be~$1/2$. As observed in Remark~\ref{ADB-delta1=delta0}, 
since $L_2=0$, we have
\[
    \min\{\d_{\ref{ML-drift-lemma-1}},\d'_{\ref{ML-drift-lemma-1}}(d)\}
            \Eq \d_{\ref{ML-drift-lemma-1}} \Eq \d_0/\sqrt{\lmax(\S)}
\]
for the upper bound on~$\d$ in Lemma~\ref{ML-drift-lemma-1}. \adbr{We also define
\eq\label{Oct-n2.7-def}
     n_{\Ref{Oct-n2.7-def}} \Def \max\Blb (5(d^{-1}\JmaxS)\max\{1,\sqrt{d\th_1}\})^{8/3}, 
         n_{\ref{ADB-lema3}}(1/g_*) \Brb \ \in\ \KKA. 
\en
After} some work, it follows from the definitions of $\smh$ and~$n_{\Ref{Oct-n2.7-def}}$,
and because $d^4 \le n$, that $n \ge \max\{n_{\Ref{Oct-n2.7-def}},\smh^{-1}(\d)\}$
implies that
\eq\label{Dec-n2.7-implies}
  \d\ \ge\ 20n^{-3/4}(d^{-1}\JmaxS) \ \ge\ 20\JmaxS/n;
\en
these inequalities are used later.

\subsection{Any $c$, $A$ and~$\s^2$ can be associated with an elementary process}\label{ADB-MPP-Tropp}
\adbr{
In this section, we relate the generator~$\ABA_n$, defined using an arbitrary choice of $c$, $A$ and~$\s^2$,
to the generator~$\AA\nud$ of an elementary process.  
The main difficulty is to match~$\s^2$, overcome by using Tropp~(2015, Theorem~1.1).
}

\begin{lemma}\label{ADB-innovations-MPP}
Let $\s^2$ be any $d\times d$ covariance matrix with positive eigenvalues $\l_1\ge\l_2\ge\ldots\ge\l_d>0$.
Then~$\s^2$ can be represented in the form
\[
    \s^2 \Eq \sJJ \tg(J) JJ^T,
\]
for a finite set~$\JJ \in \Z^d$ such that $\eii \in \JJ$, $1\le i\le d$,
such that $J \in \JJ$ implies that $-J \in \JJ$, with $\tg(-J) = \tg(J)$,
and such that 
$$
   \max_{J \in \JJ}\max_{1\le i\le d}|J_i| \Le 1 + \half\sqrt{2(d-1)\Rh(\s^2)}.
$$
Furthermore,
 $\tg(\eii) \ge \quarter\l_d$ for each~$1\le i\le d$. 
\end{lemma}

\begin{proof}
Write $\l_0 := \half\l_d = \half\lmin(\s^2)$, so that $\s^2 - \l_0 I$ is positive definite,
and has condition number $\Rh(\s^2 - \l_0 I) \le 2\Rh(\s^2)$.
By Theorem~1.1 of Tropp~(2015),  we can write
\[
    \s^2 - \l_0 I \Eq \sum_{J \in \JJ_1} \g(J) JJ^T,  
\]
where the set~$\JJ_1$ is finite, $\g(J) > 0$ for each~$J\in\JJ_1$, and the vectors~$J$ have integer coordinates
with $|J_i| \le 1 + \half\sqrt{(d-1)\Rh(\s^2 - \l_0 I)}$.  Note that the same covariance matrix is obtained
if $\g(J)JJ^T$ is replaced by $\half\g(J)\{JJ^T + (-J)(-J)^T\}$, which we do, expanding the set~$\JJ_1$ if necessary.
Writing
$\l_0 I = \sum_{i=1}^d \half\l_0 \{\eii (\eii)^T + (-\eii) (-\eii)^T\}$, and taking 
$\JJ = \JJ_1 \cup \{\pm\eii,\,1\le i\le d\}$, the lemma follows.
\end{proof}

Fitting~$A$ and~$c$ as well, in such a way that Assumptions G0--G1 and S2--S4 are all satisfied, is now easy. 

\begin{theorem}\label{ADB-MPP-A-too}
 For any $c\in\re^d$, $A$ whose eigenvalues all have negative real parts, and positive definite~$\s^2$,
there exists a sequence of elementary processes having $F(c)=0$,
$DF(c)=A$ and $\s^2$ given by~\Ref{ADB-sigma2-def}.  For these processes, defining $\d_0 := \lmin(\s^2)/(8\|A\|)$
and $\Lbar := \lbar(\s^2)$,
we have~$\e_0 \ge 1/2$ in Assumption~S4, and the quantities in~$\KKA$ are all bounded by 
continuous functions of~$\|A\|/\Lbar$ and the elements of 
$\Sp'(\s^2/\Lbar)$ and~$\Sp'(\S)$.
\end{theorem}

\begin{proof}
Represent~$\s^2$ as in Lemma~\ref{ADB-innovations-MPP}. For $J\in\JJ$, define
\[
   g^J(x) \Def \begin{cases}
                  \tg(J), &\mbox{if}\ J \in \JJ \setminus \{\eii,\,1\le i\le d\};\\
                  \tg(J) + (A(x-c))_i,  &\mbox{for}\  \adbb{J = \eii,}\ 1\le i\le d.
               \end{cases}
\]
With these functions~$g^J$, we have $\s^2 \Eq \sJJ g^J(c) JJ^T$ and,
writing $F(x) := \sJJ Jg^J(x)$, we also have $F(c)=0$ and $DF(c) = A$;
\adbr{define $\gbars := d^{-3/2} \sJJ g^J(c) |J|^3$.}

Now all the transition rates~$g^J(x)$ are constant in~$x$, except 
for $J = \eii$, $1\le i\le d$, when they are linear. For $g^{\eii}$, we have
\[
    \frac{g^{\eii}(x)}{g^{\eii}(c)} \Eq \adbb{\frac{\tg(\eii) + (A(x-c))_i}{\tg(\eii)}},
\]
and this is at least~$1/2$ if 
$$
   |x-c|\,\|A\| \Le \tfrac1{8}\lmin(\s^2)  \Le \adbb{\half}\tg(\eii),
$$
which is in turn true if $|x-c| \le \d_0$, so that we can take $\e_0 = 1/2$.
The same calculation shows that $L_0 \le 3/2$, and it is also immediate, from Lemma~\ref{ADB-innovations-MPP}, that 
\eqa
    L_1 &\le& 2\|A\|/\min_{1\le i\le d}\tg(\eii) \Le 4\|A\|/\lmin(\s^2); \non\\  
    \Lbar/g_* &\le& \lbar(\s^2)/\min_{1\le i\le d}\tg(\eii) \Le 4\Rh(\s^2/\lbar(\s^2)); \label{Sept-L1-este}\\
    d^{-1/2}\gbars/\Lbar &\le& \sqrt{1 + \Rh(\s^2/\lbar(\s^2))}.
\ena
Finally, again from Lemma~\ref{ADB-innovations-MPP}, 
\eq\label{Sept-Jmax-este}
   d^{-1}\Jmax \Le d^{-1}\bigl\{d\bigl(1 + \half\sqrt{2(d-1)\Rh(\s^2)}\bigr)^2\bigr\}^{1/2} 
           \Le \adbb{1 + \sqrt{\half\Rh(\s^2/\lbar(\s^2))}}.
\en
Hence, for this choice of~$\d_0$, the quantities in~$\KKA$ are all bounded by 
continuous functions of~$\|A\|/\Lbar$ and the elements of $\Sp'(\s^2/\Lbar)$ 
and~$\Sp'(\S)$.
\end{proof}

\subsection{The dependence of $\law(X\nud(U))$ on $X\nud(0)$}\label{ADB-initial-conditions}
We first show that the distribution 
$\law(X\nud(U) \giv X\nud(0) = X)$ does not change too much if the initial condition
is slightly altered.  The argument is based on 
that for one-dimensional processes given in Socoll \& Barbour~(2010).
We begin by bounding differences of the form
\[
    \ex \{f(X\nud(U)) \giv X\nud(0) = X - \ej\} - \ex \{f(X\nud(U)) \giv X\nud(0) = X\},
\]
and then prove a sharper bound on second differences.

\begin{theorem}\label{ADB-initial-displacement}
Let~$X_n$ be a sequence of elementary processes.
Fix any $\d < \d_{\ref{ML-drift-lemma-1}}$.  Then there are constants 
$K^j_{\ref{ADB-initial-displacement}}$, $1\le j\le d$, in~$\KKA$, such that, 
for all $n \ge \max\{n_{\Ref{Oct-n2.7-def}},\smh^{-1}(\d)\}$ as \adbb{in~\Ref{Oct-n2.7-def}},
\eqa
    \lefteqn{\sup_{f\colon \|f\|_\infty = 1}
        |\ex \{f(X\nud(U)) \giv \adbn{X\nud}(0) = X-\ej\} - \ex \{f(X\nud(U)) \giv X\nud(0) = X\}| } \non \\
          &&\quad\Le K^j_{\ref{ADB-initial-displacement}} n^{-1/2} \Bl \frac{\adbr{G\uj}}{g\uj}\Br^{1/4} 
        \max\Blb 1, \frac1{(g\uj\adbr{G\uj})^{1/4}\sqrt U} \Brb , 
       \phantom{XXXXXX} \label{ADB-f-diff}
\ena
uniformly for all $U > 0$ and $\nS{X - nc} \le n\d/2$.
\end{theorem}
 
\begin{proof}
For any $x \in \ell_1$ and any stochastic matrix~$P$, we have $\nl{x^T P} \le \nl{x}$.
Hence the quantity being bounded in~\Ref{ADB-f-diff} is non-increasing in~$U$. We can thus 
\adbr{take $U \le U\uj := 1/\sqrt{G\uj g\uj}$ in what follows, and use the bound obtained for $U = U\uj$
as a bound for all larger values of~$U$.  Note that $U\uj \le 1/g_*$.}

We begin by realizing the chain~$X\nud$ with $X\nud(0) = X_0$ 
in the form $X\nud(u) := X_0 - \ej N\nud(u) + W\nud(u)$, where 
the bivariate chain $(N\nud,W\nud)$ with state space $\Z_+ \times \Z^d$ starts at $(0,0)$, 
and, at times~$u$ such that $\nS{X\nud(u) - nc} \le n\d - \JmaxS$, has transition rates
given by
\eq\label{ADB-bivariate-process-2}
\begin{array}{ll}
 (l,W) \rightarrow (l+1,W) & \mbox{at rate}\ ng\uj ;\\
 (l,W) \rightarrow (l,W+J) & \mbox{at rate}\ n g^{J}\bigl( (X_0-l\ej+W)/n \bigr),
     \quad  J \neq -\ej \in \JJ;
\end{array}
\en
note that the first of these transitions {\it reduces\/} the $j$-coordinate of~$X\nud$ by~$1$.
At other values of~$X$, it may be that $g_\d^J(n^{-1}X)$ does not agree with~$g^J(n^{-1}X)$,
and so the transition rates of $(N\nud,W\nud)$ may be different from those given in~\Ref{ADB-bivariate-process-2}.
For this reason, if the time interval $[0,U]$ is of interest, we treat any paths of~$X\nud$ for which 
$\sup_{0\le u\le U}\nS{X\nud(u) - nc} > n\d - 3\JmaxS$ separately; the factor~$3$ ensures that shifting 
a path by a vector $J' + J''$, for any $J',J'' \in \JJ$,
still leaves it entirely within $\{X\colon \nS{X - nc} \le n\d - \Jmax\}$ over $[0,U]$.

Using the bivariate process, we deduce that
\eqa
  \lefteqn{d_{TV}\{\law_{X_0}(X\nud(U)),\law_{X_0-\ej}(X\nud(U))\}}\non\\
  &=& \frac12\sum_{X \in \Z^d}|\pr_{X_0}[X\nud(U)=X + X_0] 
 -\pr_{X_0-\ej}[X\nud(U)=X+X_0]| \non\\
  &=& \frac{1}{2}\sum_{X \in \Z^d}\Bigl|\sum_{l\geq 0}\pr_{X_0}[N\nud(U)=l]
    \pr_{X_0}[W\nud(U)=X+l\ej \giv N\nud(U)=l]  \non\\
  &&\ \mbox{} - \sum_{l\geq 1}\pr_{X_0}[N\nud(U)=l-1]
    \pr_{X_0-\ej}[W\nud(U)=X+l\ej \giv N\nud(U)=l-1] \Bigr| \non\\
  &\le& \!\! \frac{1}{2}\sum_{X \in \Z^d}\sum_{l \geq 0}
       |\pr_{X_0}[N\nud(U)=l]-\pr_{X_0}[N\nud(U)=l-1]|q^{U}_{l-1,X_0-\ej}(X+l\ej) \non\\
  && \ \mbox{}+\frac{1}{2}\sum_{X \in \Z^d}\sum_{l \geq 1}\pr_{X_0}[N\nud(U)=l]
      |q^{U}_{l,X_0}(X+l\ej)-q^{U}_{l-1,X_0-\ej}(X+l\ej)|, \non \\ \label{ADB-Poisson-part-2}
\end{eqnarray}
where 
\begin{equation}\label{ADB-dens1-S}
    q^{U}_{l,X}(W) \Def \pr[W\nud(U)=W \giv N\nud(U)=l, X\nud(0)=X]. 
\end{equation}
Now, from \BHJ~(1992, Proposition A.2.7), 
\begin{equation}\label{ADB-BHJ-S}
  \sum_{l \geq 0}|\Po(\l)\{l\}-\Po(\l)\{l-1\}| \Eq 2\max_{l \ge 0} \Po(\l)\{l\}
   \ \leq\ \frac{1}{\sqrt{\l}}. 
\end{equation} 
Hence, since $N\nud$ is a Poisson process of rate~$n \adbr{g\uj}$ until the time
\eq\label{tau-hat-def-S}
   \hht\nud \Def \t\nud(\d - 3n^{-1}\JmaxS),
\en
where~$\t\nud(\h)$ is as defined in~\Ref{ADB-tau-delta-def},
it follows that the first term in~\Ref{ADB-Poisson-part-2} is bounded by 
\eq\label{ADB-first-part-bnd-S}
    \pr_{X_0}[\hht\nud \le U] + \adbr{\half}\{n g\uj U\}^{-1/2}.
\en

Recall that $n \ge \max\{n_{\Ref{Oct-n2.7-def}},\smh^{-1}(\d)\}$, so that,
from~\Ref{Dec-n2.7-implies}, $\d - 3n^{-1}\JmaxS > 3\d/4$. 
Hence, for any $U \le \adbr{U\uj} \le 1/g_*$, we can use
Lemma~\ref{ADB-lema3} and the definition of~$\hht\nud$ to give 
\eq\label{ADB-tau-hat-prob-S}
    \pr_X[\hht\nud \le U] \Le 
    \pr_X[\t\nud(3\d/4) \le \adbr{U\uj}]
        \Le 2n^{-4},
\en
uniformly in $\nS{X - nc} \le n\d/2$.
Putting this into~\Ref{ADB-first-part-bnd-S}, \adbr{for $U \le U\uj$}, gives a contribution
to~$d_{TV}\{\law_{X_0}(X\nud(U)),\law_{X_0-\ej}(X\nud(U))\}$ from the first part of~\Ref{ADB-Poisson-part-2} of at 
most 
\eq\label{ADB-prob-este-part-1}
    \adbr{2}n^{-4} +  \half\{n g\uj U\}^{-1/2}.
\en
It thus remains only to control the differences between the conditional probabilities
$q^{U}_{l,X}(W)$ and~$q^{U}_{l-1,X-\ej}(W)$.

To make the comparison between $q^{U}_{l,X}(W)$ and~$q^{U}_{l-1,X-\ej}(W)$ for $l\ge1$, we first
condition on the whole paths of~$N\nud$ leading to the events $\{N\nud(U)=l\}$
and $\{N\nud(U)=l-1\}$, respectively, chosen to be suitably matched; we write
\eqa
  q^{U}_{l,X}(W) &=&  \frac 1{U^l} \int_{[0,U]^l}ds_1\,\ldots\,ds_{l-1}\,ds^* \non\\ 
  &&\pr_X[W\nud(U)=W \giv (N\nud)^U=\n_l(\cdot\;;s_1,\ldots,s_{l-1}, s^*)];\non\\
 q^{U}_{l-1,X-\ej}(W) &=& \frac 1 {U^l}\int_{[0,U]^{l}}ds_1...ds_{l-1}ds^*  \label{ADB-dens22-S} \\
 &&   \pr_{X-\ej}[W\nud(U)=W \giv (N\nud)^U=\n_{l-1}(\cdot\;;s_1,\ldots,s_{l-1})] ,
   \non
\ena
where
\eq\label{ADB-nu-def-S}
   \n_r(u;t_1, \ldots, t_{r})\ :=\ \sum_{i=1}^{r} \charI_{[0,u]}(t_i),
\en
and, for a function~$Y$ on~$\re_+$, $Y^u$ is used to denote $(Y(s),\,0\le s\le u)$.  
Fixing $\bs_{l-1} := (s_1,s_2,\ldots,s_{l-1})$, 
let $\bPstar$ denote the distribution of~$(W\nud)^U$, conditional on 
$(N\nud)^U = \n_l(\cdot\;;\bs_{l-1}, s^*)$ and  $X\nud(0)=X$, and let~$\bP^U_{\bs_{l-1},X}$
denote the distribution conditional on 
$(N\nud)^U = \n_{l-1}(\cdot\;;\bs_{l-1})$ and  $X\nud(0)=X$.
Write $\hrstar(u,w^u)$ to denote the Radon--Nikodym
derivative $d\bP^U_{\bs_{l-1},X-\ej}/d\bPstar$  evaluated at the path $w^u$, for any $0\le u\le U$.
Then
$$
    \bPstar[W\nud(U)=W]  \Eq \int_{\{w^U\colon w(U)=W\}} \hrstar(U,w^U)\,d\bPstar(w^U),
$$
and hence
\eqa
  \lefteqn{\bPstar[W\nud(U)=W] - \bP^U_{\bs_{l-1},X-\ej}[W\nud(U)=W] } \non\\ 
   && \Eq \int \charI_{\{W\}}(w(U))\{1 - \hrstar(U,w^U)\}\,d\bPstar(w^U).\phantom{XXXX}\label{ADB-pm-diff-S}
\ena
Thus
\eqa
  \lefteqn{\sum_{W\in\Z^d}|q^{U}_{l,X}(W) - q^{U}_{l-1,X-\ej}(W)|}\non\\
   &\le& \frac 1{U^l} \int_{[0,U]^l}ds_1\,\ldots\,ds_{l-1}\,ds^* \non\\
  && \sum_{W\in\Z^d}
      \ex^U_{(\bs_{l-1},s_*),X}\left\{|\hrstar(U,(W\nud)^U)-1|\charI_{\{W\}}(W\nud(U))\right\}\label{ADB-RN-bound-S} \\
   &\le&  \frac 2{U^l} \int_{[0,U]^l}ds_1\,\ldots\,ds_{l-1}\,ds^*\,
      \ex^U_{(\bs_{l-1},s_*),X}\left\{[1 - \hrstar(U,(W\nud)^U)]_+\right\}. 
       \non 
\ena

To evaluate the expectation, note that $\hrstar(u,(W\nud)^u)$, $u\ge0$, is a $\bPstar$-martingale
\adbb{with respect to the filtration~$\FF^{X\nud}$}, with expectation~$1$.
Now, if the path $w^U$ has~$r$ jumps of vectors $J_1,\ldots,J_r$ at times $t_1 < \cdots < t_r$, 
write
\eq\label{ADB-y-def-S}
  x_Y(v) \ :=\ n^{-1}(w(v) - \ej\n_{l-1}(v\;;s_1,\ldots,s_{l-1}) + Y),
\en
and define
\eq\label{ADB-ghat-def-S}
  \hg^{J'}(\cdot) \Def g^{J'}(\cdot),\ J'\ne -\ej;\qquad \hg^{-\ej}(\cdot) \Def 0;
   \qquad \hg(\cdot) \Def \sum_{J'\in\JJ} \hg^{J'}(\cdot).
\en
Then, for $u\le\hht\nud$, we have
\eqa
  \lefteqn{\hrstar(u,w^u)} \label{ADB-rho-hat-star}\\
  &=& 
  \begin{cases}
      \exp\left( n\int_{0}^u \{\hg(x_{X-\ej}(v)) - \hg(x_{X-\ej}(v)-\ej n^{-1})\}\,{dv}\right)\\
       \qquad\prod_{\{k\colon 0 \le t_k \le u\}}
        \left\{\hg^{J_k}(x_{X-\ej}(t_k-) - \ej n^{-1}) / \hg^{J_k}(x_{X-\ej}(t_k-)) \right\} \\ 
           \hskip4in\mbox{if}\ u < s^*;\\
     \exp\left( n\int_{0}^{s^*} \{\hg(x_{X-\ej}(v)) - \hg(x_{X-\ej}(v)-\ej n^{-1})\}\,{dv}\right)\\
       \qquad\prod_{\{k\colon 0 \le t_k \le s^*\}}
        \left\{\hg^{J_k}(x_{X-\ej}(t_k-) - \ej n^{-1}) / \hg^{J_k}(x_{X-\ej}(t_k-)) \right\}  \\
            \hskip4in\mbox{if}\ u \ge s^*; \\
  \end{cases} \non
\ena
after the `extra jump' at~$s^*$, the chains have come together. Note that $\hrstar(u,w^u)$ is
absolutely continuous except for jumps at the times~$t_k$.  Then also,
from Assumptions S3 and~S4,
\[
    \frac{\hg^{J'}(x-\ej n^{-1})}{\hg^{J'}(x)} \Eq 1, \adbr{\qquad J' \notin \{\eii,\, i \in I\uj\}},
\]
and
\eq\label{ADB-ratio-bnd-S}
    \left| \frac{\hg^{\eii}(x-\ej n^{-1})}{\hg^{\eii}(x)} - 1 \right| \Le \frac{2\|D g^{\eii}\|_{\d_0}}{n g^{\eii}(c)}
      \Le 2L_1/n,\quad \adbr{i\in I\uj},
\en
uniformly in $|x-c| \le \d_0$.  Hence, if we define the stopping time
\eq\label{ADB-stopping2}
    \hph_n \Def \inf\{u\ge0\colon\, \hrostar(u,(W\nud)^u) \ge 2\}, 
\en
the jumps of the martingale $\hrostar(u,(W\nud)^u)$, stopped at the time $\min(U,\hht_n^\d,\hph_n)$,
are of size at most $4L_1/n$.  Hence, \adbr{recalling that~$L_0\le 3/2$,} the stopped martingale has expected 
quadratic variation up to time~$u$ of at most
\eq\label{ADB-QV-S}
    \int_0^u \left(\frac{4L_1}{n}\right)^2\,n \adbr{\sum_{i\in I\uj} }\|g^{J'}\|_{\d_0} \,dv
        \Le \adbr{n^{-1} K_{\Ref{ADB-QV-S}} G\uj u,}
\en
where \adbr{$K_{\Ref{ADB-QV-S}} := 24L_1^2 \in \KKA$.} This in turn also implies that, 
for $0 < u \le U$,
\eq\label{ADB-mean-square-MG-S}
   \ex^U_{(\bs_{l-1},s_*),X_0}
     \{(\hrostar(u\wedge\hht\nud\wedge\hph_n,(W\nud)^{u\wedge\hht\nud\wedge\hph_n})
                   - 1)^2\} \Le n^{-1} \adbr{K_{\Ref{ADB-QV-S}} G\uj} u.
\en

Clearly, from~\Ref{ADB-mean-square-MG-S} and from Kolmogorov's inequality, once again taking 
$U = \adbr{U\uj}$,
\eq\label{ADB-phi-bound-S}
   \bP^U_{(\bs_{l-1},s_*),X_0}[\hph_n < \min\{U,\hht_n^\d\}] \Le n^{-1}\adbr{K_{\Ref{ADB-QV-S}} G\uj U\uj}
               \adbb{\Eq} n^{-1}\adbr{K_{\Ref{ADB-QV-S}} (G\uj/g\uj)^{1/2}}.
\en
Hence, for this choice of~$U$, from~\Ref{ADB-mean-square-MG-S} and~\Ref{ADB-phi-bound-S},
\eqa
  \lefteqn{\bE^U_{(\bs_{l-1},s_*),X_0}\left\{[1 - \hrostar(U,(W\nud)^U)]_+\right\}}\non \\ 
       && \Le \min\bigl\{1,2n^{-1/2}\adbr{K_{\Ref{ADB-QV-S}}(G\uj/g\uj)^{1/4}} 
                  + \bP^U_{(\bs_{l-1},s_*),X_0}[\hht_n^\d < U] \bigr\}. 
           \phantom{XXX} \label{ADB-abs-bnd-S}
\ena
In view of Lemma~\ref{ADB-lema3},
the expectation of the term $\bP^U_{(\bs_{l-1},s_*),X_0}[\hht_n^\d < U]$ is bounded by~$2n^{-4}$, 
uniformly in $\nS{X_0-nc} \le n\d/2$, because~$n \ge \max\{\adbg{n_{\Ref{Oct-n2.7-def}}},\smh^{-1}(\d)\}$. 
Substituting this into~\Ref{ADB-Poisson-part-2}, and using~\Ref{ADB-RN-bound-S}, it follows that
\eqa
   &&\sum_{l \geq 1}\pr[N\nud(U)=l-1]\sum_{W\in\Z^d}|q^{U}_{l,X_0}(W+l\ej) - q^{U}_{l-1,X_0-\ej}(W+l\ej)| \non\\
     &&\qquad\Le 2\{2n^{-1/2}\adbr{K_{\Ref{ADB-QV-S}}(G\uj/g\uj)^{1/4}} + 2\pr_{X_0}[\hht_n^\d < U]\}\non\\
     &&\qquad\Le 2\{2n^{-1/2}\adbr{K_{\Ref{ADB-QV-S}}(G\uj/g\uj)^{1/4}} + 4n^{-4}\}, \label{ADB-single-rho-result-S}
\ena
uniformly for~$X_0$ such that $\nS{X_0-nc} \le n\d/2$, and for $n \ge \max\{n_{\Ref{Oct-n2.7-def}},\smh^{-1}(\d)\}$. 
Thus the contribution
to~$d_{TV}\{\law_{X_0}(X\nud(U)),\law_{X_0-\ej}(X\nud(U))\}$ from the second part of~\Ref{ADB-Poisson-part-2} is at 
most 
\eq\label{Dec-second-contn-S}
    2\adbr{K_{\Ref{ADB-QV-S}}(G\uj/g\uj)^{1/4}} n^{-1/2} + 4n^{-4},
\en 
and this, with~\Ref{ADB-prob-este-part-1}, proves the theorem.
\end{proof}

\begin{remark}\label{ADB-Poisson-case}
{\rm
\adbr{As observed after~\Ref{ADB-new-gdef}, we \adbb{always} have $G\uj \le d\Lbar$;
however, \adbb{if $A = -\l I$ and~$(X_n)$ is as in Theorem~\ref{ADB-MPP-A-too}, $G\uj/g\uj = 1$ does not} grow
with~$d$.}
}
\ignore{
The bound~\Ref{ADB-QV-S} for the expected quadratic variation 
of the martingale $\rstar(u,w^u)$ allows for contributions
from all jumps~$J'$ of~$X\nud$.
However, as is clear from~\Ref{ADB-rho-hat-star}, if~$J'$ is such that 
$g^{J'}(y) = g^{J'}(y+Jn^{-1})$ for all positions~$y$, then a jump of~$J'$ in~$X\nud$ does not
change the value of $\rstar(u,w^u)$.  Since only jumps in the coordinate directions~$\ej$ have 
rates that vary with position, only these contribute to the sum in~\Ref{ADB-QV-S}.
} 
\end{remark}

Theorem~\ref{ADB-initial-displacement} bounds differences of the form
$$
    \ex\{f(X\nud(U) \giv X\nud(0) = X_0 - \ej\} - \ex\{f(X\nud(U) \giv X\nud(0) = X_0\}, 
$$
showing that they
are of order $O(n^{-1/2})$ uniformly in $U\ge0$, for~$f$ such that
$\|f\|_\infty \le 1$.  
We now show that the corresponding second differences are of order $O(n^{-1})$.

\begin{theorem}\label{ADB-initial-displacement-2}
Let~$X_n$ be a sequence of elementary processes.
Fix any $\d < \d_{\ref{ML-drift-lemma-1}}$.  
Then there are constants 
$(K^{ji}_{\ref{ADB-initial-displacement-2}},\, 1\le j,i \le d)$ in~$\KKA$
such that, for any function~$f$ with $\|f\|_\infty \le 1$,
\eqa
    \lefteqn{\bigl|\ex\{f(X\nud(U)) \giv X\nud(0) = X_0 - \ej - \eii\} - \ex\{f(X\nud(U)) \giv X\nud(0) = X_0 - \ej\}} 
          \non \\
    &&\mbox{}\qquad - \ex\{f(X\nud(U)) \giv X\nud(0) = X_0  - \eii\} +  \ex\{f(X\nud(U)) \giv X\nud(0) = X_0\} \bigr| 
          \non \\ 
    &&\qquad\qquad\quad  \Le \adbb{K^{ji}_{\ref{ADB-initial-displacement-2}}}\, n^{-1}\Bl \frac{\adbr{G_{ij}^+}}{g_{ij}^-} \Br^{1/2}
   \max\Blb 1, \frac1{U\sqrt{\adbr{G_{ij}^+} g_{ij}^+}} \Brb,\phantom{XXXXX}
          \label{ADB-second-f-diff}
\ena
uniformly for all $U > 0$, for $\nS{X_0-nc} \le n\d/4$, and for 
$n \ge \max\{n_{\Ref{Oct-n2.7-def}},\smh^{-1}(\d)\}$, where
\[
    g_{ij}^+ \Def \max\{g\uii,g\uj\}; \quad  g_{ij}^- \Def \min\{g\uii,g\uj\};\quad
                 \adbr{G_{ij}^+ \Def \max\{g_{ij}^+, (G\uii + G\uj)\}}. 
\]
\end{theorem}
 
\begin{proof}
As in the previous theorem, the supremum over~$f$ of the quantity being bounded in~\Ref{ADB-second-f-diff} 
is non-increasing in~$U$,
so that we can argue for \adbr{$U \le U\uij := (G_{ij}^+ g_{ij}^+)^{-1/2} \le 1/g_{ij}^+$}, and then use the bound for 
$U = U\uij$ for all larger values of~$U$.
We give the detailed argument for $j$ and~$i$ distinct; it is almost identical if they are the same.

Much as for~\Ref{ADB-Poisson-part-2}, we split off Poisson processes of $-\ej$ and~$-\eii$ jumps. 
We write $X\nud(u) := X_0 -\ej N\nud(u) - \eii (N')\nud(u) + W\nud(u)$, 
where the trivariate chain $(N\nud,(N')\nud,W\nud)$ with state space $\Z_+^2 \times \Z^d$ has transition rates
\eq\label{ADB-trivariate-process}
\begin{array}{ll}
 (l,l',W) \rightarrow (l+1,l',W) & \mbox{at rate}\ \, ng\uj ;\\
 (l,l',W) \rightarrow (l,l'+1,W) & \mbox{at rate}\ \,ng\uii ;\\
 (l,l',W) \rightarrow (l,l',W+J) & \mbox{at rate}\ \,n g^{J}\bigl( (X_0 - l\ej - l'\eii +W)/n \bigr),\\
     & \qquad\qquad\quad J\notin \{-\ej,-\eii\},
\end{array}
\en
up to the time~$\hht_n^\d$, and starts at $(0,0,0)$. Defining
\eqs
   q^u_{l,l',X}(W) &:=& \pr_X[W\nud(u) = W \giv N\nud(u)=l, (N')\nud(u) = l'];\\
   p_X(l,l',u) &:=& \pr_X[N\nud(u)=l,(N')\nud(u)=l'],
\ens
this allows us to deduce that
\eqa
  \lefteqn{\ex\{f(X\nud(U)) \giv X\nud(0) = X_0 - \ej - \eii \} 
                    - \ex\{f(X\nud(U)) \giv X\nud(0) = X_0 - \ej \}}\non\\[1ex]
   &&\qquad\mbox{} - \ex\{f(X\nud(U)) \giv X\nud(0) = X_0 - \eii \} 
                          +  \ex\{f(X\nud(U)) \giv X\nud(0) = X_0\} \non\\[1ex]
  &=& \sum_{X \in \Z^d}f(X)\sum_{l\geq 0}\sum_{l'\ge0} \Blb p_{X_0}(l-1,l'-1,U)
                     q^U_{l-1,l'-1,X_0-\ej-\eii }(X+l\ej+l'\eii )  \right. \non\\
  &&\quad\hskip1.1in \left.\mbox{} - p_{X_0}(l-1,l',U)q^U_{l-1,l',X_0-\ej }(X+l\ej+l'\eii ) \right. \non\\
  &&\qquad \hskip1.2in\left.\mbox{}  - p_{X_0}(l,l'-1,U)q^U_{l,l'-1,X_0-\eii }(X+l\ej+l'\eii ) \right. \non\\
  &&\quad \qquad \hskip1.3in \left.\mbox{} 
            + p_{X_0}(l,l',U) q^U_{l,l',X_0}(X+l\ej+l'\eii ) \Brb . \label{ADB-Poisson-part-3}
\ena
Write $r_{jk,X}(l,l',u) := p_{X}(l-j,l'-k,u)/p_{X}(l,l',u)$ for $j,k \in \{0,1\}$, and 
\[
    \rhr_{j,k,Y;l,l',X}^u(W) \Def q^u_{l-j,l'-k,X+Y}(W)/q^u_{l,l',X}(W).
\]
Then the right hand side of~\Ref{ADB-Poisson-part-3} can be expressed as
\eqa
  \lefteqn{\sum_{l\geq 0}\sum_{l'\ge0}  p_{X_0}(l,l',U)  \sum_{w \in \Z^d}f(w-l\ej -l'\eii)q^U_{l,l',X_0}(w)}\non\\ 
   &&   \Bigl\{ r_{11,X_0}(l,l',U) \rhr^U_{1,1,-\ej-\eii ;l,l',X_0}(w) - r_{10,X_0}(l,l',U) \rhr^U_{1,0,-\ej;l,l',X_0}(w)
     \phantom{XXX}
          \non \\
  &&\qquad\qquad\quad\hskip1in\mbox{}
   - r_{01,X_0}(l,l',U) \rhr^U_{0,1,-\eii ;l,l',X_0}(w) + 1 \Bigr\}. \label{ADB-RN-expression}
\ena

We now use the decomposition
\[
     \adbb{r}\rhr \Eq (r-1)(\rhr-1) + (r-1) + (\rhr - 1) + 1
\]
in each term of~\Ref{ADB-RN-expression}.  The sum corresponding to taking~$1$ yields nothing. 
Then, for the sum corresponding to taking~$(r-1)$ alone, summing over~$w$ first and using $\|f\|_\infty \le 1$, 
we have
\eqa
   \lefteqn{\sum_{l,l'\geq 0}  p_{X_0}(l,l',U) \sum_{w \in \Z^d} |f(w-l\ej -l'\eii )|\, q^U_{l,l',X_0}(w) }\non\\
   &&\qquad\qquad\quad |r_{11,X_0}(l,l',U) - r_{10,X_0}(l,l',U) - r_{01,X_0}(l,l',U)  + 1 |
             \label{ADB-RN-expression-r}\\
    &\le&  \sum_{l\geq 0}\sum_{l'\ge0} 
       p_{X_0}(l,l',U)|r_{11,X_0}(l,l',U) ) - r_{10,X_0}(l,l',U) - r_{01,X_0}(l,l',U)  + 1 |.
                 \non
\ena
As for \Ref{ADB-Poisson-part-2} and~\Ref{ADB-prob-este-part-1}, the processes $(N\nud,(N')\nud)$ can be
coupled to independent Poisson processes with rates $ng\uj$ and~$ng\uii$ respectively on the
interval $[0,U]$, with failure probability at most $\pr_{X_0}[\hht\nud < U]$.
Hence, using $\p\uj$ to denote~$\Po(nUg\uj)$, \Ref{ADB-RN-expression-r} gives a contribution 
to~\Ref{ADB-RN-expression} of at most
\eqa
 \lefteqn{\slo \sldo |\p\uj\{l\} - \p\uj\{l-1\}|\,
           |\p\uii\{l'\} - \p\uii\{l'-1\}|  + 4\pr_{X_0}[\hht\nud < U]}\non\\
  && \Eq 4\dtv\bigl(\p\uj,\p\uj*\ptmass \bigr)\,\dtv\bigl(\p\uii,\p\uii*\ptmass\bigr)
                + 4\pr_{X_0}[\hht\nud < U] \non\\
     && \Le \frac4{\sqrt{g\uj g\uii}} \,\frac1{n U} + 8n^{-4}, \phantom{XXXXXXXXXXXXXXXXXXXXX}
        \label{ADB-second-diff-b}
\ena
for $n \ge \max\{n_{\Ref{Oct-n2.7-def}},\smh^{-1}(\d)\}$, uniformly in $\nS{X_0-nc} \le n\d/4$.

We separate the sum corresponding to $(r-1)(\rhr-1)$ in~\Ref{ADB-RN-expression} into three pieces,
corresponding to the subscripts $(1,1)$, $(1,0)$ and~$(0,1)$, and use $\|f\|_\infty \le 1$.
We then use an argument similar to that leading to~\Ref{ADB-single-rho-result-S}; we sketch
it for the $(1,1)$ case.
First, by conditioning on the paths of $N\nud$ and~$(N')\nud$ and using~\Ref{ADB-M(1)-bound}
below, it follows, much as for~\Ref{ADB-single-rho-result-S} and for~\Ref{ADB-abs-bnd-S}, that, 
for each $l,l' \ge 0$,
\eqa
   \leqn{ \sum_{w \in \Z^d} q^U_{l,l',X_0}(w)|1 - \rhr^U_{1,1,-\ej-\eii ;l,l',X_0}(w)| } \non\\
     &\le& \min\{2, 2n^{-1/2}\sqrt{\adbr{K_{\Ref{ADB-QV-S}}(G\uii + G\uj)} U} 
                                    + 4n^{-1}\adbr{K_{\Ref{ADB-QV-S}}(G\uii + G\uj)} U \non\\ 
      &&\qquad\quad\mbox{} + 2\pr_{X_0}[\hht\nud < U \giv N\nud(u)=l, (N')\nud(u) = l']\}  \label{ADB-rR-1}\\
     &\le&  4n^{-1/2}\sqrt{\adbr{K_{\Ref{ADB-QV-S}}(G\uii + G\uj)} U} 
                                 +  2\pr_{X_0}[\hht\nud < U \giv N\nud(u)=l, (N')\nud(u) = l'].
           \non
\ena
Then, as in treating~\Ref{ADB-RN-expression-r}, and using Lemma~\ref{ADB-lema3}, we have 
\eqa
    \leqn{ \sum_{l,l'\geq 0}  p_{X_0}(l,l',U)|r_{11,X_0}(l,l',U) - 1|} \non\\
   &\le&  2\bigl\{\dtv\bigl(\p\uj,\p\uj*\ptmass \bigr) 
           + \dtv\bigl(\p\uii,\p\uii*\ptmass\bigr)\bigr\}  + 4\pr_{X_0}[\hht\nud < U] \non\\           
   &\le& \frac2{\sqrt{nUg\uj}} + \frac2{\sqrt{nUg\uii}} + 8n^{-4} \Le \frac4{\sqrt{nUg_{ij}^-}} + 8n^{-4}, 
     \label{ADB-rR-2}
\ena
for $n \ge \max\{n_{\Ref{Oct-n2.7-def}},\smh^{-1}(\d)\}$, uniformly in $\nS{X_0-nc} \le n\d/4$.

Combining the first part of~\Ref{ADB-rR-1} with~\Ref{ADB-rR-2} 
gives a contribution to~\Ref{ADB-RN-expression} bounded by
\eq\label{ADB-second-diff-a}
    K n^{-1}d^{1/2} \bigl((G\uii + G\uj)/g_{ij}^-\bigr)^{1/2} + 12n^{-2},
\en
uniformly for \adbr{$U \le U\uij$} and $\nS{X_0-nc} \le n\d/4$, for 
$K := 4\sqrt{\adbr{K_{\Ref{ADB-QV-S}}}} \in \KKA$.
Taking the second part of~\Ref{ADB-rR-1} with~\Ref{ADB-rR-2}, it is immediate that
\eqs
   \leqn{ 2\sum_{l,l'\geq 0}  p_{X_0}(l,l',U)|r_{11,X_0}(l,l',U) - 1|
               \pr_{X_0}[\hht\nud < U \giv N\nud(u)=l, (N')\nud(u) = l']} \\[-1ex]
   &&\mbox{}\hskip3.4in \bone\{r_{11,X_0}(l,l',U) \le n^2\} \\
   &&\Le 2n^2 \pr_{X_0}[\hht\nud < U] \Le 4n^{-2},\phantom{XXXXXXXXXXXXXXXXXX}
\ens 
by Lemma~\ref{ADB-lema3}, since $n \ge \max\{n_{\Ref{Oct-n2.7-def}},\smh^{-1}(\d)\}$.  
For the remainder, we have at most
\eqa
   \leqn{2 \sum_{l,l'\geq 0}  p_{X_0}(l,l',U)\bone\{r_{11,X_0}(l+1,l'+1,U) > n^2\}} \label{ADB-new-small-bit}\\
  &&\Le 2\pr_{X_0}[\hht\nud < U] + 2\sum_{l,l'\geq 0} \p\uj\{l\} \p\uii\{l'\}\bone\{r_{11,X_0}(l+1,l'+1,U) > n^2\}.
          \non
\ena
Now
\[
    |p_{X_0}(l,l',U) - \p\uj\{l\} \p\uii\{l'\}| \Le \pr_{X_0}[\hht\nud < U] \Le 2n^{-4}.
\]
\adbr{This implies that, if 
\eq\label{ADB-Poisson-prob-condition}
     \min(\p\uj\{l\}, \p\uii\{l'\}, \p\uj\{l+1\}, \p\uii\{l'+1\})\ \ge\ 2n^{-2},
\en 
then $r_{11,X_0}(l+1,l'+1,U) \le n^2$, giving no contribution to the sum in~\Ref{ADB-new-small-bit}.}
\adbb{This is because}
\[
   r_{11,X_0}(l+1,l'+1,U) \Le 3 \frac{\p\uj\{l\} \p\uii\{l'\}}{\p\uj\{l+1\} \p\uii\{l'+1\}}
                          \Le 3\frac{(l+1)(l'+1)}{n^2U^2 g_{ij}^- g_{ij}^+}\,;
\]
by Proposition A.2.3~(i) of \BHJ~(1992), if~\Ref{ADB-Poisson-prob-condition} holds,
\[
   3\frac{(l+1)(l'+1)}{n^2U^2 g_{ij}^- g_{ij}^+} \Le 100(\log n)^2 \ <\ n^2,
\]
for all $n \ge 40$. In proving the first inequality,
we assume that $nUg_{ij}^- \ge 1$, since the inequality in the 
statement of the theorem is immediate for smaller~$nU$.
\adbr{
This leaves only a contribution to the sum in~\Ref{ADB-new-small-bit} from $l,l'$ for 
which~\Ref{ADB-Poisson-prob-condition} does not hold,
and this is at most 
\[
    2 \slo \bigl\{\p\uj\{l\}\bone\{\p\uj\{l\} \le 2n^{-2}\} 
             + \p\uii\{l\}\bone\{\p\uii\{l\} \le 2n^{-2}\} \bigr\} \Le 8n^{-3/2},
\]
by Proposition A.2.3~(ii), (iii) and~(iv) of \BHJ~(1992), 
if $n\ge 10$, because we also have $nUg_{ij}^+ \le n$ in $U\le U\uij$. 
}

The trickiest sum is that corresponding to~$(\rhr-1)$ alone.  Using $\|f\|_\infty \le 1$, we need first
to examine the quantity
\eq\label{ADB-q-sum}
  \sum_{w \in \Z^d} q^U_{l,l',X_0}(w)  
  \Bigl|\rhr^U_{1,1,-\ej-\eii ;l,l',X_0}(w) - \rhr^U_{1,0,-\ej;l,l',X_0}(w) 
                - \rhr^U_{0,1,-\eii ;l,l',X_0}(w) + 1\Bigr|.
\en
We treat it, after conditioning on realizations of the underlying Poisson processes
$N\nud$ and~$(N')\nud$, as the expectation of the absolute value at time~$U$ of an 
\adbb{$\FF^{X\nud}$-martingale $M\ut(W\nud)$, defined in~\Ref{ADB-MG-def-2} below}. 
Let $W^u := (W(t),\,0\le t\le u)$ denote the restriction of a function~$W$ on $\re_+$ to $[0,u]$.
Write $\bs_l := (s_1,\ldots,s_l)$, $\bs'_{l'} := (s'_1,\ldots,s'_{l'})$.
If realizations of $N\nud$ and~$(N')\nud$, having $l$ and~$l'$ points respectively in $[0,U]$,  are denoted by
$\n_l(\cdot;\bs_l)$ and $\n'_{l'}(\cdot;\bs'_{l'})$, as in~\Ref{ADB-nu-def-S}, we then denote conditional
probability and expectation,
given $(N\nud)^U = \n_l(\cdot;\bs_l)$, $((N')\nud)^U = \n'_{l'}(\cdot;\bs'_{l'})$ and $X\nud(0) = X$, by
$\bP^U_{\bs_l, \bs'_{l'}, X}$ and $\bE^U_{\bs_l, \bs'_{l'}, X}$, and we denote the corresponding conditional
density of~$(W\nud)^u$ at the path segment~$W^u$, with respect to some suitable reference measure, by
\[
   q^U(u,W^u;\,\bs_l, \bs'_{l'}, X).
\]
We then define the Radon--Nikodym derivatives 
\eqs
    \rhr^U_{11}(u,W^u; (\bs_{l-1},s_*),(\bs'_{l'-1},s'_*),X_0)
     &:=& \frac{q^U(u,W^u;\,\bs_{l-1}, \bs'_{l'-1}, X_0 - \ej - \eii )}
            {q^U(u,W^u; (\bs_{l-1},s_*),(\bs'_{l'-1},s'_*),X_0)};\\
   \rhr^U_{10}(u,W^u; (\bs_{l-1},s_*),(\bs'_{l'-1},s'_*),X_0)
     &:=& \frac{q^U(u,W^u;\,\bs_{l-1}, (\bs'_{l'-1},s'_*), X_0 - \ej )}
            {q^U(u,W^u; (\bs_{l-1},s_*),(\bs'_{l'-1},s'_*),X_0)};\\
   \rhr^U_{01}(u,W^u; (\bs_{l-1},s_*),(\bs'_{l'-1},s'_*),X_0)
     &:=& \frac{q^U(u,W^u;\,(\bs_{l-1},s_*), \bs'_{l'-1}, X_0 - \eii )}
            {q^U(u,W^u; (\bs_{l-1},s_*),(\bs'_{l'-1},s'_*),X_0)};
\ens
these have explicit formulae analogous to~\Ref{ADB-rho-hat-star}ß.
We use them to formulate the analogue of the argument used in the proof of Theorem~\ref{ADB-initial-displacement}.
For example, we can write
\eqs
   \lefteqn{\sum_{w \in \Z^d} q^U_{l,l',X_0}(w) \rhr^U_{1,1,-\ej-\eii ;l,l',X_0}(w) } \\
  &=& \frac1{U^{l+l'}}\int_{[0,U]^{l+l'}} ds_1\ldots ds_{l-1}ds_*\,ds'_1\ldots ds'_{l'-1}ds'_* 
   \\  && \qquad\qquad\qquad
  \sum_{w \in \Z^d} 
      \bP^U_{\bs_{l-1},\bs'_{l'-1},X_0 - \ej - \eii }[W(U) = w] \\
  &=& \frac1{U^{l+l'}}\int_{[0,U]^{l+l'}} ds_1\ldots ds_{l-1}ds_*\,ds'_1\ldots ds'_{l'-1}ds'_* \\
  && \ \sum_{w \in \Z^d}  \bE^U_{(\bs_{l-1},s_*),(\bs'_{l'-1},s'_*),X_0}
        \{\rhr_{11}^U(U,W^U; (\bs_{l-1},s_*),(\bs'_{l'-1},s'_*),X_0)I[W(U) = w]\}.
\ens
The mean zero martingale~$\adbb{M\ut(W\nud)}$ of main interest to us can then be expressed as
\eqa
  \lefteqn{M\ut(W\nud)(u) \Def \rhr^U_{11}(u,(W\nud)^u; (\bs_{l-1},s_*),(\bs'_{l'-1},s'_*),X_0) }  
             \non\\
   &&\qquad\qquad\quad\ \mbox{} 
               - \rhr^U_{10}(u,(W\nud)^u; (\bs_{l-1},s_*),(\bs'_{l'-1},s'_*),X_0) \label{ADB-MG-def-2}\\
   &&\qquad\qquad\qquad\ \mbox{} 
       - \rhr^U_{01}(u,(W\nud)^u; (\bs_{l-1},s_*),(\bs'_{l'-1},s'_*),X_0) + 1,\non
\ena
with~$(W\nud)^U$ a random element with distribution~$\bP^U_{(\bs_{l-1},s_*),(\bs'_{l'-1},s'_*),X_0}$.
We also define the \adbb{$\FF^{X\nud}$}-martingale
\[
    M\ui(W\nud)(u) \Def \rhr^U_{11}(u,(W\nud)^u; (\bs_{l-1},s_*),(\bs'_{l'-1},s'_*),X_0)  - 1,
\]
for use in the proof below, as well as for the proof of the estimate of the $(1,1)$ term 
in~\Ref{ADB-rR-1} above.

We now set $x\nud(u) := n^{-1}(W\nud(u) + X_0 - \ej \n_{l-1}(u;\bs_{l-1}) - \eii \n'_{l'-1}(u;\bs'_{l'-1}))$ for
\hbox{$u < \min\{s_*,s'_*\}$.}
If, for $u < \min\{s_*,s'_*\}$ and $\nS{x\nud(u) - c} \le \d - 3n^{-1}\JmaxS$,
there is a jump of~$\eur$ in~$W\nud$ at time~$u$, for some
$1\le r\le d$, this gives rise to a jump in the martingale~$M\ut(W\nud)$ at~$u$ of
\eqs
    \lefteqn{\rhr^U_{11}(u-,(W\nud)^{u-}; (\bs_{l-1},s_*),(\bs'_{l'-1},s'_*),X_0)
          \Bigl(\frac{\hg^{\eur}(x\nud(u-) - n^{-1}(\ej+\eii ))}{\hg^{\eur}(x\nud(u-))}-1 \Bigr)}\\
    &-& \rhr^U_{10}(u-,(W\nud)^{u-}; (\bs_{l-1},s_*),(\bs'_{l'-1},s'_*),X_0)
           \Bigl(\frac{\hg^{\eur}(x\nud(u-) - n^{-1}\ej)}{\hg^{\eur}(x\nud(u-))}-1 \Bigr) \\
   &{}~~-& \rhr^U_{01}(u-,(W\nud)^{u-}; (\bs_{l-1},s_*),(\bs'_{l'-1},s'_*),X_0)
             \Bigl(\frac{\hg^{\eur}(x\nud(u-) - n^{-1}\eii )}{\hg^{\eur}(x\nud(u-))}-1 \Bigr).\phantom{X}
\ens
If $s_* < u < s'_*$, the elements $-n^{-1}\ej$ are removed from the arguments of $\hg^{\eur}$, simplifying
the considerations, but then $x\nud(u)$ is replaced by $x\nud(u)-n^{-1}\ej$; the elements $-n^{-1}\eii$ are 
removed  
if $s'_* < u < s_*$, and then $x\nud(u)$ is replaced by $x\nud(u)-n^{-1}\eii$;  if 
$u > \max\{s_*,s'_*\}$, both elements $-n^{-1}\ej$ and $-n^{-1}\eii$ are removed, and so there is no jump.
Now, because the transition rate $g^{\eur}(x)$ is linear in~$x$,
\eqs
   \lefteqn{ \Bigl(\frac{\hg^{\eur}(x\nud(u) - n^{-1}(\ej+\eii ))}{\hg^{\eur}(x\nud(u))}-1 \Bigr)
           - \Bigl(\frac{\hg^{\eur}(x\nud(u) - n^{-1}\ej)}{\hg^{\eur}(x\nud(u))}-1 \Bigr)}\\
    &&   \qquad \mbox{}\phantom{XXXXXXXXXXXX}   
                  - \Bigl(\frac{\hg^{\eur}(x\nud(u) - n^{-1}\eii )}{\hg^{\eur}(x\nud(u))}-1 \Bigr) 
   \Eq 0 , 
\ens
and so $R^U$ can be replaced by $|R^U-1|$ when bounding the sizes of the jumps, irrespective of
the relative positions of $s_*$, $s'_*$ and~$u$.
Since also, from~\Ref{ADB-L-defs} and Assumption~S4,
\eq\label{ADB-Y-bnd}
   \Blm \frac{\hg^{\eur}(x\nud(u) + n^{-1}Y)}{\hg^{\eur}(x\nud(u))}-1 \Brm \Le 2n^{-1}|Y|L_1,
\en
the remaining contributions to the jump in~$M\ut\adbb{(W\nud)}$ are at most
\eqa
   \lefteqn{\frac{4 L_1}{n} \{|\rhr^U_{11}(u,(W\nud)^u; (\bs_{l-1},s_*),(\bs'_{l'-1},s'_*),X_0) - 1|} \non\\ 
     &&\qquad \mbox{} + |\rhr^U_{10}(u,(W\nud)^u; (\bs_{l-1},s_*),(\bs'_{l'-1},s'_*),X_0) - 1| \non\\
    &&\qquad\qquad\mbox{}+ |\rhr^U_{01}(u,(W\nud)^u; (\bs_{l-1},s_*),(\bs'_{l'-1},s'_*),X_0) - 1|\}.
        \phantom{X}\label{ADB-main-jump}
\ena
We can now bound the quadratic variation arising from each of the three terms individually,
by the argument leading to~\Ref{ADB-mean-square-MG-S}.  Defining
$$ 
    \tf_n \Def \inf\bigl\{u\ge0\colon\, \tm(u) \ge 2 \bigr\}, 
$$
where
\eqs
    \tm(u) &:=&
     \max\bigl\{\rhr^U_{11}(u,(W\nud)^u; (\bs_{l-1},s_*),(\bs'_{l'-1},s'_*),X_0), \\
    &&\qquad\qquad        \rhr^U_{10}(u,(W\nud)^u; (\bs_{l-1},s_*),(\bs'_{l'-1},s'_*),X_0),\\
    &&\qquad\qquad\qquad         \rhr^U_{01}(u,(W\nud)^u; (\bs_{l-1},s_*),(\bs'_{l'-1},s'_*),X_0)\bigr\} ,
\ens
we use the martingale~$M\ui\adbb{(W\nud)}$ \adbr{and~\Ref{ADB-Y-bnd} with the argument leading to~\Ref{ADB-QV-S}} to give
\eqa\lefteqn{
    \bE^U_{(\bs_{l-1},s_*),(\bs'_{l'-1},s'_*),X_0}
     \{[\rhr^U_{11}(u\wedge \hht_n^\d \wedge \tf_n,(W\nud)^{u\wedge \hht_n^\d \wedge \tf_n}; 
         (\bs_{l-1},s_*),(\bs'_{l'-1},s'_*),X_0) - 1]^2 \} } \non\\
   && \Le n^{-1} \adbr{4K_{\Ref{ADB-QV-S}}(G\uii + G\uj) u;}\phantom{XXXXXXXXXXXXXXXXXXXXXXX}
           \label{ADB-M(1)-bound}
\ena
the same bound \adbr{holds for $\rhr^U_{10}$ and~$\rhr^U_{01}$ also, but with $4(G\uii + G\uj)$ 
replaced by $G\uj$ and~$G\uii$ respectively.}  
Hence the expected quadratic variation of the martingale~$M\ut(W\nud)$ stopped at 
$u\wedge \hht_n^\d \wedge \tf_n$ is at most 
\eqs
    \leqn{ \adbr{n(G\uii + G\uj) \int_0^u 
              \Bl\frac{12 L_1}{n}\Br^2  \Bl \frac{4K_{\Ref{ADB-QV-S}}(G\uii + G\uj) v}n \Br\,dv} } \\
    && \qquad \Le \adbr{2n^{-2}\bigl((G\uii + G\uj) u\bigr)^2  
                 (12 L_1)^2  \,K_{\Ref{ADB-QV-S}}  \Le n^{-2}K_8\bigl((G\uii + G\uj) u\bigr)^2}, 
\ens
uniformly in~$\nS{X_0-nc} \le n\d$, and in $l$, $l'$, $\bs_{l-1}$, $\bs'_{l'-1}$, $s_*$ and~$s'_*$,
for $K_8 := 2(12 L_1)^2  \,\adbr{K_{\Ref{ADB-QV-S}}} \in \KKA$.  This gives
a contribution of at most \adbr{$n^{-1}\sqrt{K_8}\,(G\uii + G\uj) U$} to~\Ref{ADB-q-sum}, and 
hence to~\Ref{ADB-RN-expression},
from the expectation of $|M\ut - 1|$, stopped at $U\wedge \hht_n^\d \wedge \tf_n$.

Because the martingale~$M\ut(W\nud)$ is not uniformly bounded from below, we can no longer use
an argument as for~\Ref{ADB-abs-bnd-S} to bound the contributions to~\Ref{ADB-RN-expression} from 
the events $\hht_n^\d < U$
and $\tf_n < U$.  Instead, we consider their contributions for each element of~$M\ut(W\nud)$
separately.  For example, writing 
\eqs
    \trhr_*^U &:=& \rhr^U_{11}((W\nud)^U; (\bs_{l-1},s_*),(\bs'_{l'-1},s'_*),X_0);\\
    \Estar &:=&  \bE^U_{(\bs_{l-1},s_*),(\bs'_{l'-1},s'_*),X_0}; \quad
    \Pstar \Def  \bP^U_{(\bs_{l-1},s_*),(\bs'_{l'-1},s'_*),X_0},
\ens
we have
\eqs
 \lefteqn{\sum_{w \in \Z^d} f(w-l\ej-l'\eii ) \Estar 
      \{\trhr_*^U I[W\nud(U)=w,\hht_n^\d < U] \} }\\
  &\le& \Estar  \Bigl\{\trhr_*^U  I\Bigl[\sup_{0\le u\le U}|X_0 + W\nud(u) - \ej \n(u;\bs_{l-1},s_*) \\
  &&\quad\qquad\qquad\qquad\qquad\qquad\mbox{}  
                - \eii \n'(u;\bs'_{l'-1},s'_*) - nc| \ge n\d - 3\JmaxS\Bigr] \Bigr\} \\
  &=& \bE^U_{\bs_{l-1},\bs'_{l'-1},X_0-\ej-\eii }
           \Bigl\{I\Bigl[\sup_{0\le u\le U}|X_0 + W\nud(u) - \ej \n(u;\bs_{l-1}) \\
  &&\quad\mbox{} - \eii \n'(u;\bs'_{l'-1})  - \ej \charI_{[s_*,U]}(u) - \eii \charI_{[s'_*,U]}(u)- nc| 
              \ge n\d - 3\JmaxS \Bigr] \Bigr\} \\
  &\le& \bP^U_{\bs_{l-1},\bs'_{l'-1},X_0-\ej-\eii } [\t\nud(\d - 3n^{-1}\JmaxS - 2n^{-1}/\sqrt{\lmin(\S)}) < U].
\ens
Now both the inequalities
\[
   \nS{X_0 -\ej-\eii - nc} \Le n\d/2
\] 
and
\[
    \t\nud\bigl(\d - 3n^{-1}\JmaxS - 2n^{-1}/\sqrt{\lmin(\S)}\bigr)\ \ge\ \t\nud(3\d/4)
\]
are satisfied if $n^{3/4}\d > 20d^{-1}\JmaxS$ and $\nS{X_0-nc} \le n\d/4$.
Taking expectations over the realizations of $(N\nud)^U$ and~${((N')\nud)}^U$ 
and invoking Lemma~\ref{ADB-lema3} thus gives a contribution
to~\Ref{ADB-RN-expression} of at most $2n^{-4}$, uniformly in $\nS{X_0-nc} \le n\d/4$,
for~$n \ge \max\bigl\{n_{\ref{ADB-lema3}}(1/g_*),(20d^{-1}\JmaxS/\d)^{4/3},\smh^{-1}(\d)\bigr\}$; 
this inequality is satisfied if $n \ge \max\{\adbg{n_{\Ref{Oct-n2.7-def}}},\smh^{-1}(\d)\}$.  Then
\eqs
   \lefteqn{\sum_{w \in \Z^d} f(w-l\ej-l'\eii) \Estar 
      \Bigl\{\trhr_*^U  
   I[W\nud(U)=w,\tf_n < \min\{U,\hht\nud\}] \Bigr\} }\\
    &\le& \Estar 
      \{\trhr_*^U I[\tf_n < \min\{U,\hht\nud\}] \} \\
    &\le& 2 \Pstar[\tf_n < \min\{U,\hht\nud\}] 
   + \Estar \{\trhr_*^U I[\tf_n^{11} < \min\{U,\hht\nud\}] \},
\ens
where 
\[
   \tf^{11}_n \Def \inf\bigl\{u\ge0\colon\, \rhr^U_{11}(u,(W\nud)^u; (\bs_{l-1},s_*),(\bs'_{l'-1},s'_*),X_0) \ge 2\}.
\]
The first of these terms is at most \adbr{$5n^{-1}K_{\Ref{ADB-QV-S}}(G\uii + G\uj) U$}, using~\Ref{ADB-M(1)-bound}
and its analogues for the quantities $R_{11}^U$, $R_{10}^U$ and~$R_{01}^U$ appearing
in the definition of~$\tm(U)$, and then applying Kolmogorov's inequality; the argument is
much as for~\Ref{ADB-phi-bound-S}. The second is no larger than
\eq\label{ADB-last-hope}
    2\bP^U_{\bs_{l-1},\bs'_{l'-1},X_0  - \ej - \eii } [\tf_n^{11}  < \min\{U,\hht\nud\}].
\en
However, under $\bP^U_{\bs_{l-1},\bs'_{l'-1},X_0  - \ej - \eii }$,
the process 
$$
          M' \Def  \{1/\rhr^U_{11}(u,(W\nud)^u; (\bs_{l-1},s_*),(\bs'_{l'-1},s'_*),X_0),\,u\ge0\}
$$
is an \adbb{$\FF^{X\nud}$-}martingale with mean~$1$.  Arguing much as for~\Ref{ADB-QV-S},
its expected quadratic variation up to the time $\min\{U,\hht\nud,(\tf'_n)^{11}\}$
can be shown to be at most $4n^{-1}\adbr{K_{\Ref{ADB-QV-S}}(G\uii + G\uj)} U$; 
here, $(\tf'_n)^{11} := \inf\{u\ge0\colon\, M'(u) \ge 2\}$.  Using an argument much as that
for~\Ref{ADB-phi-bound-S},
Kolmogorov's inequality now shows that the quantity in~\Ref{ADB-last-hope} is itself at most 
$8n^{-1}\adbr{K_{\Ref{ADB-QV-S}}(G\uii + G\uj)} U$, 
giving a contribution to~\Ref{ADB-RN-expression} of order~$O(n^{-1}\adbr{(G\uii + G\uj)} U)$.
Combining these considerations with \Ref{ADB-second-diff-b} and~\Ref{ADB-second-diff-a}, the inequality
of the theorem follows for \adbb{$n \ge \max\{\adbg{n_{\Ref{Oct-n2.7-def}}},\smh^{-1}(\d)\}$}.
\end{proof}

\subsection{Coupling copies of $X\nud$}\label{ML-couplings}
In this section, we show 
that copies of~$X\nud$ with different initial states can be defined
on the same probability space, in such a way that they 
coincide rather quickly. As a consequence, the total variation distance between
their distributions becomes small as time increases. Our arguments are
reminiscient of those in Roberts \& Rosenthal~(1996).

The basic coupling that we use relies mainly on the drift
towards~$nc$ to achieve this.  We define the process $(X\nudci(t),X\nudct(t))$ on
$\tB_{n,\d}(c) \times \tB_{n,\d}(c)$
to have the transition rates
$$
\begin{array}{ll}
   (X_1,X_2) \ \to\ (X_1+J,X_2+J) \ &\mbox{at rate}\quad n\{g^J_\d(n^{-1}X_1)\wedge g^J_\d(n^{-1}X_2)\};\\
   (X_1,X_2) \ \to\ (X_1,X_2+J) \quad &\mbox{at rate}\quad n\{g^J_\d(n^{-1}X_2) - g^J_\d(n^{-1}X_1)\}_+;\\
   (X_1,X_2) \ \to\ (X_1+J,X_2) \quad &\mbox{at rate}\quad n\{g^J_\d(n^{-1}X_1) - g^J_\d(n^{-1}X_2)\}_+,
\end{array}
$$
for each $J\in\JJ$.  Let its generator be denoted by~$\tAA\nud$.
Our coupling argument begins with a drift inequality.

\begin{lemma}\label{ML-drift-lemma-2}
Let~$X_n$ be a sequence of elementary processes.
Define $h_1(X_1,X_2) := \nS{X_1-X_2}^2$;
let $\lla_1$ be as in~\Ref{ADB-Lambda-def} and~$\d_1 := \d_0/\sqrt{\lmax(\S)}$.
Then, for $\d \le \d_1/3$, there exists $K_{\ref{ML-drift-lemma-2}} \in \KKA$,
defined in~\Ref{ML-K3-def}, such that, for all $(X_1,X_2)$ with 
$\max\{\nS{X_1-nc},\nS{X_2-nc}\} \le n\d - \JmaxS$ and $\nS{X_1-X_2} \ge dK_{\ref{ML-drift-lemma-2}}$, we have
\[
   \tAA\nud h_2(X_1,X_2) \Le -\half\lla_1 h_2(X_1,X_2),
\]
where $h_2(X_1,X_2) := h_1(X_1,X_2) + d^2K_{\ref{ML-drift-lemma-2}}^2$.
\end{lemma} 

\begin{proof}
For any $\ccc > 0$, write $h^{(\ccc)}(X_1,X_2) := h_1(X_1,X_2) + \ccc$.
By the definition of~$h_1$, the transitions where both components of $(X_1,X_2)$ make the
same jump make no contribution to $\tAA\nud h^{(\ccc)}(X_1,X_2)$.  Hence, for $(X_1,X_2)$ with
$\max\{\nS{X_1-nc},\nS{X_2-nc}\} \le n\d - \JmaxS$,
and writing $x_i := n^{-1}X_i$, $i=1,2$, we have
\eqs
 \leqn{\tAA\nud h^{(\ccc)}(X_1,X_2) \Eq n\sJJ \Bigl\{(g^J(x_1) - g^J(x_2))_+\{2J^T \S^{-1}(X_1-X_2) + J^T\S^{-1}J\} } \\
    &&\mbox{}\qquad\quad\quad\qquad +  (g^J(x_2) - g^J(x_1))_+ \{-2J^T \S^{-1}(X_1-X_2) + J^T\S^{-1}J\} \Bigr\} \\
    && = 2n(F(x_1) - F(x_2))^T\S^{-1}(X_1-X_2) + n\sJJ|g^J(x_1) - g^J(x_2)|\,\nS{J}^2,\phantom{XX}
\ens
since, for such~$(X_1,X_2)$, $g^J_\d(x_1) = g^J(x_1)$ and $g^J_\d(x_2) = g^J(x_2)$.
Now, 
since the transition rates~$g^J(x)$ are all linear in~$x$, we have
\eqs
  \leqn{2n(F(x_1) - F(x_2))^T\S^{-1}(X_1-X_2) \Eq -(X_1-X_2)^T\S^{-1}\s^2\S^{-1}(X_1-X_2)} \\
          &&\Le -\lmin(\s^2_\S) \nS{X_1-X_2}^2 \Eq -2\lla_1 h_1(X_1,X_2). \phantom{XXXXXXXXXXX}
\ens
Then
\eqs
  n\sJJ|g^J(x_1) - g^J(x_2)|\,\nS{J}^2 &\le& L_1 (\L/\lmin(\S)) \nS{X_1-X_2} \sqrt{\lmax(\S)} \\
     &\le& \lla_1\nS{X_1-X_2}^2 \Eq \lla_1 h_1(X_1,X_2), 
\ens
if \adbr{$\nS{X_1-X_2} \ge dK_{\ref{ML-drift-lemma-2}}$, where 
\eq\label{ML-K3-def}
    K_{\ref{ML-drift-lemma-2}} \Def \max\{1, L_1 (\Lbar/\lla_1)\sqrt{\Rh(\S)/\lmin(\S)}\}\ \in\ \KKA .
\en 
}
From this, it follows that
$$
   \tAA\nud h^{(\ccc)}(X_1,X_2) \Le -\lla_1 h_1(X_1,X_2) \Le -\half\lla_1(h_1(X_1,X_2) + d^2K_{\ref{ML-drift-lemma-2}}^2),
$$  
for $\d \le \d_1/3$ and for $\nS{X_1-X_2} \ge dK_{\ref{ML-drift-lemma-2}}$.  
Taking $\ccc = d^2K_{\ref{ML-drift-lemma-2}}^2$ proves the lemma.
\end{proof}

We now convert the drift inequality into a bound on the distribution of the coupling time
\[
   \t_C \Def \inf\{t \ge 0\colon X\nudci(t) = X\nudct(t)\},
\]
for arbitrary values of $(X\nudci(0),X\nudct(0))$. Our broad strategy is as follows.  
If $\max\{\nS{X\nudci(0) - nc},\nS{X\nudct(0)-nc}\} > 3n\d/8$, we run both
processes {\it independently\/} for a fixed time interval~$t_1$, chosen in such a way that we have 
$\max\{\nS{X\nudci(t_1) - nc},\nS{X\nudct(t_1)-nc}\} \le 3n\d/8$, with
probability at least $1/16$.
If not, we continue to repeat the procedure, over intervals of length~$t_1$, until both 
$X\nudci$ and~$X\nudct$ are within~$3n\d/8$ of~$nc$ \adbg{in the $\nS{\cdot}$-norm}.
We then couple the processes $X\nudci$ and $X\nudct$
as for Lemma~\ref{ML-drift-lemma-2}, and run them until the minimum $(\t_3\wedge\t_0)$ of the time~$\t_3$, 
at which $\nS{X\nudci(t) - X\nudct(t)}$ first falls below the value~$dK_{\ref{ML-drift-lemma-2}}$, and 
the time~$\t_0$, at which first $\max\{\nS{X\nudci(t) - nc},\nS{X\nudct(t)-nc}\} > n\d/2$. 
We call these two stages together a `drift phase'.

If~$\t_0$ is the first to occur, we begin another drift phase.
If not, we enter a `trial phase', of length $t_3 = 1/\lla_1$. 
By Theorem~\ref{ADB-initial-displacement} and Assumption~S2, and because $\nS{X\nudci(t) - X\nudct(t)}
\le dK_{\ref{ML-drift-lemma-2}}$ implies that 
$|X\nudci(t) - X\nudct(t)|_1 \le d^{3/2}K_{\ref{ML-drift-lemma-2}}\sqrt{\lmax(\S)}$, we have
\eq\label{ADB-TV-X1-X2}
   \dtv\bigl(\law(X\nudci(\t_3 + t_3)),\law(X\nudct(\t_3 + t_3))\bigr) 
              \Le d^{3/2} C_* n^{-1/2}\Bl \frac{d\Lbar}{g_*} \Br^{1/4},
\en
for $n \ge n_{\Ref{Oct-n2.7-def}}$, where \adbr{we use $G\uj \le d\Lbar$}, and where
\eq\label{ADB-C_*-def}
    C_* \Def K_{\ref{ML-drift-lemma-2}}\sqrt{\lmax(\S)}\max_{1\le j\le d} \Blb K^j_{\ref{ADB-initial-displacement}}
           \max\Bl 1, \frac{\sqrt{\lla_1}}{(dg\uj\Lbar)^{1/4}}\Br \Brb \ \in\ \KKA.
\en
Hence the two processes can be coupled
in such a way that $X\nudci(\t_3 + t_3) = X\nudct(\t_3 + t_3)$, except on an event of probability at most
$d^{3/2} C_* n^{-1/2}(\L/g_*)^{1/4}$, and this is the coupling that we use.  On the event that
the values of the two processes are equal at time $\t_3 + t_3$, the coupling is said to have been
successful.  If not, a new drift phase begins, (or another trial phase, if
$\nS{X\nudci(\t_3 + t_3) - X\nudct(\t_3 + t_3)} \le dK_{\ref{ML-drift-lemma-2}}$ and 
$\max\{\nS{X\nudci(\t_3 + t_3) - nc},\nS{X\nudct(\t_3 + t_3)-nc} \} \le n\d/2$).  
This sequence of steps is repeated until coupling is achieved.

\begin{theorem}\label{ML-TV}
Let~$X_n$ \adbr{be a sequence of elementary processes.}
Let $\lla_1$ be as in~\Ref{ADB-Lambda-def}, and let $\d_1 := \min\{\adbr{3},\d_0/\sqrt{\lmax(\S)}\}$. 
Then, for any $\d \le \d_1/3$,
there is a constant 
$n_{\ref{ML-TV}}$ in~$\KKA$ such that, 
whatever the values of  $X_1,X_2 \in \tB_{n,\d}(c) := \{X \in \Z^d\colon \nS{X-nc} \le n\d\}$ and $t\ge0$,
\[
    \dtv\bigl(\law(X\nud(t) \giv X\nud(0) = X_1),\law(X\nud(t) \giv X\nud(0) = X_2)\bigr)
               \Le 9\, (2n)^{1/16} e^{-\lla_2 t},
\]
for all $n \ge \max\{d^4, n_{\ref{ML-TV}},\smh^{-1}(\d/2)\}$, where $\lla_2 :=  \lla_1/128$.
The quantity $n_{\ref{ML-TV}}$ is defined in~\Ref{Dec-n3.5-def}. 
\end{theorem}

\ignore{
\begin{remark}\label{Aug-K3}
{\rm
We prove a simple bound in Theorem~\ref{ML-TV} that only involves the quantities in~$\KKA$
through~$\lla_1$.  For this, the radius~$n\d$ of the state space of~$X\nud$ needs to be restricted, and
this explains the upper bound on~$\d$ involving~$K_{\ref{ML-drift-lemma-2}}$, imposed through 
that on~$\d_1$. 
Since we think of~$\d$ being chosen to be small, such upper bounds are of no great importance.}
\end{remark}
}

\begin{proof}
Recalling the definition~\Ref{Aug-tilde-B-def} of~$\tB_{n,\d}(c)$, we begin by writing
\eqs
    B_0 &:=& \tB_{n,\d}(c)\times \tB_{n,\d}(c);\\
    B_1 &:=& \{(X_1,X_2) \in B_0\colon \max_{i=1,2}\nS{X_i - nc} \le n\d/2\}; \\
    B_2 &:=& \{(X_1,X_2) \in B_0\colon \max_{i=1,2}\nS{X_i - nc} \le 3n\d/8\};\\
    B_3 &:=& \{(X_1,X_2) \in B_1\colon  \nS{X_1-X_2} \le dK_{\ref{ML-drift-lemma-2}}\} . 
\ens
Clearly, $B_3 \subset B_1 \subset B_0$ and $B_2 \subset B_1$.
Let 
\eqs
    \t_0 &:=& \inf\{t \ge 0\colon (X\nudci(t),X\nudct(t)) \notin B_1\};\\
    \t_2 &:=& \inf\{t \ge 0\colon (X\nudci(t),X\nudct(t)) \in B_2\};\\
    \t_3 &:=& \inf\{t \ge 0\colon (X\nudci(t),X\nudct(t)) \in B_3\}.
\ens 
Then, for any $s,\grrr > 0$, and for $0\le i\le 4$, we define 
\eq\label{ML-phi_i-def}
     \f_i(\grrr,s) \Def \max_{(X_1,X_2)\in B_i} 
             \ex\uxit\{e^{\grrr (\t_{C}\wedge s)} \},
\en
where the coupling time~$\t_{C}$ is defined for copies of the processes $X\nudci$ and~$X\nudct$ specified using
drift and trial phases as above, and where $\ex\uxit$ and~$\pr\uxit$ refer to the distribution 
conditional on $(X\nudci(0),X\nudct(0)) = (X_1,X_2)$.  We shall establish that, for~$n$ large enough,
\[
    \pr\uxit[\t_C > t] \Le 9(2n)^{1/16} e^{-\lla_1 t/128}.
\]

Fix~$t_1'$ such that $e^{t_1'} = 128$, and write $t_1 := \lla_1^{-1}t_1'$.  Then, 
for any~$(X_1,X_2) \in B_0$, it follows from Lemma~\ref{ML-time-to-root-n} that 
$$
    \pr\uxit[\inf\{t > 0\colon \nS{X\nudci(t)} \le n\d/4\} >  t_1 ] \Le 1/2,
$$
if also $n\d \ge 8\JmaxS$; this is true, from~\Ref{Dec-n2.7-implies}, for 
$n \ge \max\{n_{\Ref{Oct-n2.7-def}},\smh^{-1}(\d/2)\}$. 
But now, taking $\h = 3\d/8$ and $\nS{X_0-nc} \le \d/4$ in Lemma~\ref{ADB-lema3}, 
it follows that 
$$
   \pr\uxit[\nS{X\nudci(t_1)} \le 3\d/8]  \ \ge\ 1/4,
$$
provided that~$n \ge \max\{n_{\ref{ML-drift-lemma-1}}, (t_1'K_{\ref{ADB-lema3}}(\Lbar/\lla_1))^{1/2}\}$, 
and that $2n^{-4} \le 1/2$.
Hence, for any choice of~$(X_1,X_2) \notin B_2$ and for \adbb{$n \ge \max\{n_1,\smh^{-1}(\d/2)\}$}, where
\[
    n_{1} \Def \adbb{\max\{d^4,\adbg{n_{\Ref{Oct-n2.7-def}}}\}} \ \in\ \KKA,
\]
it follows that
$\pr\uxit[(X\nudci(t_1),X\nudct(t_1)) \in B_2] \ge 1/16$, if $X\nudci$ and~$X\nudct$ are run
independently over the interval $[0,t_1]$.
Thus, defining
\[
     \f(\grrr,s)\ :=\ \max_{(X_1,X_2) \in B_0}\ex\uxit\bigl\{e^{\grrr (\t_2\wedge s)}\bigr\} ,
\]
the Markov property yields
\[
   \f(\grrr,s) \Le e^{\grrr t_1} \{1 + 15\f(\grrr,s)\}/16.
\]
Choosing $u_0 = 1/32$, so that, in particular, $15e^{u_0t_1'}/16  = 15(128)^{u_0}/16 < 31/32$, and 
then~$\grrr = u_0\lla_1$,  it follows for any $s>0$ that
\[
    \f(\grrr,s) \Le 2 (128)^{u_0} \Le 31/15.
\]
Considering the possibilities if $\t_C \le \t_2$ or if $\t_C > \t_2$, it now follows by
the strong Markov property that
\eq\label{ADB-phi_0-inequality}
    \f_0(u_0\lla_1,s) \Le 31 \f_2(u_0\lla_1,s)/15.
\en

We next consider what happens if the process starts in~$B_2$.
If $dK_{\ref{ML-drift-lemma-2}} \ge 3n\d/4$, then $B_2 \subset B_3$, and so $\f_2(u_0\lla_1,s) \le \f_3(u_0\lla_1,s)$.
If not, note that
\eqa
  \leqn{ \ex\uxit \{e^{\grrr(\t_C\wedge s)}\} } \non\\
  &\le&  \ex\uxit \bigl\{e^{\grrr(\t_C\wedge s)}\bone\{\t_3 \le (\t_0 \wedge s)\} \bigr\}
         + \ex\uxit \bigl\{e^{\grrr(\t_C\wedge s)}\bone\{\t_0 \le (\t_3 \wedge s)\} \bigr\} \non\\
  &&\qquad\mbox{} + \ex\uxit \bigl\{e^{\grrr(\t_C\wedge s)}\bone\{(\t_0 \wedge \t_3) > s\} \bigr\}\non\\
  &\le& \ex\uxit \bigl\{e^{\grrr\t_3}\bone\{\t_3 \le (\t_0 \wedge s)\} \bigr\}\f_3(\grrr,s) \non\\
  &&\ \mbox{}  + e^{\grrr s}\{\pr\uxit [\t_0 \le (\t_3 \wedge s)] 
       + \pr\uxit[(\t_0 \wedge \t_3) > s]\} \f_0(\grrr,s), \label{ADB-Markov-split}
\ena
with the last inequality following from the strong Markov property.
Now, in view of Lemma~\ref{ML-drift-lemma-2}, if we define
\eq\label{ML-M1-def}
     M_1(t) \Def e^{\lla_1 t/2} h_2(X\nudci(t),X\nudct(t)),
\en
then $M_1(t\wedge \t_3 \wedge \t_0)$ is \adbb{an $\FF^{X\nudci,X\nudct}$-}supermartingale.  This implies that
\eq\label{ML-tau_3-bnd-0}
    d^2K_{\ref{ML-drift-lemma-2}}^2\, \ex\uxit \bigl\{e^{\lla_1\t_3/2} \bone\{\t_3 \le (\t_0 \wedge s)\} \bigr\}
          \Le h_2(X_1,X_2),
\en
and also that
\eq\label{ML-prob-bnd}
    d^2 K_{\ref{ML-drift-lemma-2}}^2\,  e^{\lla_1 s/2} \pr\uxit[(\t_3 \wedge \t_0) > s] 
          \Le h_2(X_1,X_2).
\en
Thence, by Jensen's inequality and~\Ref{ML-tau_3-bnd-0}, for any $0 \le u\le 1$, we also have
\eq\label{ML-tau3-bnd}
    \ex\uxit \bigl\{e^{u\lla_1\t_3/2} \bone\{\t_3 \le (\t_0 \wedge s)\} \bigr\}
              \Le \{h_2(X_1,X_2)/d^2 K_{\ref{ML-drift-lemma-2}}^2\}^u.
\en
Finally, from Lemma~\ref{ADB-lema3} with $\h = \d/2$, 
for $(X_1,X_2) \in B_2$ and for~$\th_1$ as in Lemma~\ref{ML-drift-lemma-1}, 
\eq\label{ADB-t0-prob}
    \pr\uxit[\t_0 \le (\t_3 \wedge s)] \Le \pr\uxit[\t_0 \le s]
            \Le 2 \exp\{-7n\th_1\d^2/64\},
\en
if $s  \le n/\lla_1$ and $n \ge \max\{n_2,\smh^{-1}(\d/2)\}$, where
$$
  n_2 \Def \max\Blb \adbg{n_{\Ref{Oct-n2.7-def}}}, \frac{K_{\ref{ADB-lema3}} \Lbar}{\lla_1} \Brb
        \ \in\ \KKA;
$$
this follows because $n_{\ref{ML-drift-lemma-1}} \le \adbg{n_{\Ref{Oct-n2.7-def}}}$, and because, for $n$ and~$s$
chosen in this way, $\exp\{n^{-1}\th_1(3n\d/8)^2\} \ge n^9 \ge nK_{\ref{ADB-lema3}}\L s$, because $\d \ge 2\smh(n)$ 
and $n \ge d^4$.

Hence, taking $\grrr = u_0\lla_1/2$, with~$u_0=1/32$ as above, substituting  \Ref{ML-prob-bnd},
\Ref{ML-tau3-bnd} and \Ref{ADB-t0-prob} into~\Ref{ADB-Markov-split}, and then using~\Ref{ADB-phi_0-inequality}, 
we have
\eqs
 \leqn{ \f_2(u_0\lla_1/2,s) 
    \Le (2n\d/dK_{\ref{ML-drift-lemma-2}})^{2u_0} \f_3(u_0\lla_1/2,s) + P(u_0,s,n)\f_0(u_0\lla_1/2,s) }\non\\
     &&\Le  (2n\d/dK_{\ref{ML-drift-lemma-2}})^{2u_0} \f_3(u_0\lla_1/2,s) + 31P(u_0,s,n)\f_2(u_0\lla_1/2,s)/15, \phantom{XX}
\ens
where

\eqs
   P(u,s,n) &:=& e^{u\lla_1 s/2}\{\pr\uxit[\t_0 \le (\t_3 \wedge s)] + \pr\uxit[(\t_3 \wedge \t_0) > s]\}  \\
   &\le& e^{u\lla_1 s/2}\{2 \exp\{-7n^{3/4} d\th_1\d^2/64\} + (2n\d/dK_{\ref{ML-drift-lemma-2}})^{2u_0} e^{-\lla_1 s/2}\},
\ens
recalling $n \ge d^4$ for the final inequality.
It thus follows that $P(u_0,s_n,n) \le 15/62$ for all $n \ge \max\{n_{2},d^4,\smh^{-1}(\d/2)\}$ such that
\eq\label{n3-cond}
    \Bl \frac{2n\d}{dK_{\ref{ML-drift-lemma-2}}} \Br^{2u_0} e^{-\lla_1(1-u_0)s_n/2} \Le \frac7{62}
   \quad \mbox{and}\quad e^{-(7n^{3/4} d\th_1\d^2/64 - u_0\lla_1 s_n/2)} \Le \frac2{31}.
\en
Picking $s = s_n := 64\log n/\lla_1$, and recalling that
 $u_0 = 1/32$ and $2\smh(n) \le \d \le \d_1/3 \adbr{\le 1 \le K_{\ref{ML-drift-lemma-2}}}$, it is enough
that $n \ge n_3$, where 
\eqs 
    n_3  &:=& \max\Blb \left[ \adbr{2^{1/16}} \Bl\frac{62}7\Br \right]^{1/30}  ,
         \Bl\frac{31}2 \Br^{1/6} \Brb 
        \ \in\ \KKA.
\ens
\ignore{
the last two inequalities ensure that
\[
    7n^{3/4}d\th_1\d^2/128 \ \ge\ n^{1/4}/64 \ \ge\ \log\{31/2\}, 
\]
satisfying the second of the conditions in~\Ref{n3-cond}, and the first two ensure that
\[
    (2n\d/dK_{\ref{ML-drift-lemma-2}})^{1/16} \Le 7n^{1/4}/62 \Le 7e^{31n^{1/4}/64}/62,
\]
satisfying the first of them.
}
 Hence, for $n \ge \max\{n_{2},n_{3},\smh^{-1}(\d/2)\}$,
\eq\label{ADB-phi_2-inequality}
    \f_2(u_0\lla_1/2,s_n) \Le 2(2n\d/dK_{\ref{ML-drift-lemma-2}})^{2u_0} \f_3(u_0\lla_1/2,s_n).
\en

For $(X_1,X_2) \in B_3$, we take $t_3 = 1/\lla_1$, and use~\Ref{ADB-TV-X1-X2} to
conclude that
\eq\label{ML-f3-bnd}
  \f_3(u\lla_1/2,s_n) \Le e^{u/2}\{1 + d^{3/2}C_* n^{-1/2}(\L/g_*)^{1/4}\,\f_0(u\lla_1/2,s_n)\} ,
\en
for~$C_*$ as defined in~\Ref{ADB-C_*-def}, if $n \ge n_{\Ref{Oct-n2.7-def}}$.  
From \Ref{ADB-phi_0-inequality} and~\Ref{ADB-phi_2-inequality}, we have
\eq\label{ADB-phi_0-phi_3-inequality}
   \f_0(u_0\lla_1/2,s_n) \Le 62(2n\d/dK_{\ref{ML-drift-lemma-2}})^{2u_0} \f_3(u_0\lla_1/2,s_n)/15,
\en
for all $n \ge \max\{d^4,n_{2},n_{3},\smh^{-1}(\d/2)\}$. Taking $u = u_0 = 1/32$ in~\Ref{ML-f3-bnd} 
and using~\Ref{ADB-phi_0-phi_3-inequality}, the coefficient of $\f_3(u_0\lla_1/2,s_n)$ on the
right hand side of~\Ref{ML-f3-bnd} is at most
\[
  e^{1/64} d^{3/2}  C_* n^{-1/2}(\L/g_*)^{1/4} 62(2n\d/dK_{\ref{ML-drift-lemma-2}})^{1/16}/15 \Le 1/2
\]
if, using $n \ge d^4$ and $\d \le \d_1/3 \adbr{\le 1 \le K_{\ref{ML-drift-lemma-2}}}$, 
\[
   n\ \ge\ n_{4} \Def  e\,\adbr{2^4} \Bl \frac{124C_*}{15} \Br^{64} \Bl \frac{\Lbar}{g_*} \Br^{16}   
                 \ \in\ \KKA.
\]
Hence, from \Ref{ML-f3-bnd} and~\Ref{ADB-phi_0-phi_3-inequality}, for 
$n \ge  \max\{\max_{2\le l\le 4} n_{l},\smh^{-1}(\d/2)\}$, we have 
\[
   \f_3(u_0\lla_1/2,s_n) \Le 2e^{1/64}.
\]
Combining this with~\Ref{ADB-phi_0-phi_3-inequality} and the definition~\Ref{ML-phi_i-def} of~$\f_0$, 
it follows that, for all $(X_1,X_2) \in B_0$, we have
\[
     \pr\uxit[\t_C > t] \Le 9 (2n)^{2u_0} e^{-u_0\lla_1 t/2}, \quad 0\le t\le s_n, 
\]
for 
$n \ge \max\{\adbg{n_{\Ref{Oct-n2.7-def}}},\max_{1\le l\le 4} n_{l},\smh^{-1}(\d/2)\}$.
In particular, 
$$
    \pr\uxit[\t_C > s_n] \Le 9(2n)^{2u_0} e^{-u_0 \lla_1 s_n/2}.
$$

However, by the strong Markov property, for $t > s_n$,
\[
   \pr\uxit[\t_C > t] \Le 
    \pr\uxit[\t_C > s_n] \max_{(X_1,X_2) \in B_0} \pr\uxit[\t_C > t-s_n].
\]
Arguing inductively, it follows that, for $r \in \Z_+$ and $0 \le v < s_n$,
\[
    \pr\uxit[\t_C > rs_n+v] \Le 
    \bigl(9(2n)^{2u_0}  e^{-u_0 \lla_1 s_n/2}\bigr)^r 9(2n)^{2u_0} e^{-u_0\lla_1 v/2}.
\]
Take~$n$ so large that $9(2n)^{2u_0} e^{-u_0 \lla_1 s_n/4} \le 1$; this is true, for $u_0 = 1/32$
and $s_n = 64\log n/\lla_1$, if 
$n \ge n_5 := 9^{4}2^{1/4}$.
Then, for all $(X_1,X_2) \in B_0$, we have
\[
   \pr\uxit[\t_C > rs_n+v] \Le 9(2n)^{2u_0} e^{-u_0 \lla_1(rs_n+v)/4}.
\]
Since $u_0 = 1/32$, the inequality in the theorem is thus proved for  
 $n \ge \max\{n_{\ref{ML-TV}},\smh^{-1}(\d/2)\}$, where 
\eq\label{Dec-n3.5-def}
   n_{\ref{ML-TV}} \Def \max_{1\le l\le 5} n_{l} \in \KKA,
\en
with 
$\lla_1/128$ for~$\lla_2$.
\end{proof}

\section{Stein's method based on~$X\nud$}\label{ADB-Steins-method}
\setcounter{equation}{0}

\subsection{Bounding the solutions of the Stein equation}\label{ADB-Stein-bounds}
We now use the results of the previous section to bound the first and second
differences of the solutions $h_B := h_{B,n}^\d$ of the Stein equation corresponding to
the generator~$\AA\nud$ defined in~\Ref{ADB-generator-eqn}, for the elementary
processes satisfying Assumptions G0, G1 and S2--S4.  
We recall from the
introduction the definitions
\eq\label{ADB-delta(X)-norms}
   \|\D f(X)\|_\infty \Def \max_{1\le j\le d}|\D_j f(X)|;\quad
   \|\D^2 f(X)\|_\infty \Def \max_{1\le j,k \le d}|\D_{jk} f(X)|,
\en
where $\D_j$ and~$\D_{jk}$, as defined in~\Ref{ADB-h-diffs-def}, denote the components of the 
first and second difference operators $\D$ and~$\D^2$, respectively.

\begin{theorem}\label{ADB-h-bounds}
Let~$X_n$ be \adbr{an elementary process, 
satisfying Assumptions G0, G1 and~S2--S4 for some $\d_0 > 0$.}  
Let~$\d_1 := \min\{\adbr{3},\d_0/\sqrt{\lmax(\S)}\}$. 
Then there are constants $\k_0,\k_1,\k_2 \in \KKA$ such that,
for any $B \subset \Z^d$ and any $\d \le \d_1/3$,
the solution~$h_B := h_{B,n}^\d$ of the Stein equation 
$$
    \AA\nud h_B(X) \Eq \charI_B(X) - \Pi\nud\{B\}
$$
satisfies
\eqs
   &&|h_B(X)| \Le \lla_1^{-1}\k_0\log n;\quad \|\D h_B(X)\|_\infty \Le \lla_1^{-1}\k_1 d^{1/4} n^{-1/2}\log n;\\ 
         &&    \|\D^2 h_B(X)\|_\infty \Le \lla_1^{-1} \k_2 d^{1/2} n^{-1}\log n,
\ens
for all $\nS{X - nc} \le n\d/4$ and $n \ge  \max\{n_{\ref{ML-TV}},\smh^{-1}(\d/2)\}$.
The constants $\k_0,\k_1,\k_2$ are given in \Ref{Dec-kappa0}, \Ref{Dec-kappa1} and~\Ref{Dec-kappa2},
respectively, \adbr{and~$\lla_1$ is as in~\Ref{ADB-Lambda-def} .}
\end{theorem}

\begin{proof}
The argument starts from
the explicit representation of~$h_B$ given in~\Ref{ADB-Poisson-eq}, which immediately yields
\eqa
 h_B(X) &=& - \int_0^\infty (\pr[X\nud(t) \in B \giv X\nud(0) = X]  - \Pi\nud\{B\})\,dt 
            \non \\
     &=& - \sum_{Y\in \tB_{n,\d}(c)} \int_0^\infty (\pr[X\nud(t) \in B \giv X\nud(0) = X] \non\\
  &&\qquad\qquad\qquad\mbox{}   -  \pr[X\nud(t) \in B \giv X\nud(0) = Y])\Pi\nud(Y)\,dt. 
            \phantom{XXX}\label{ADB-h_B-def}
\ena 
Because $\d \le \d_1/3$ and $9\cdot2^{1/16} \le 10$, we can use Theorem~\ref{ML-TV} for 
$n \ge \max\{n_{\ref{ML-TV}},\smh^{-1}(\d/2)\}$ to give
\eqs
    |h_B(X)| &\le& 2\lla_2^{-1}\log n 
           + \int_{2\lla_2^{-1}\log n}^\infty 10 n^{1/16}  e^{-\lla_2 t}\,dt \\
      &\le& 2\lla_2^{-1}\log n + 10 \lla_2^{-1} n^{-1},
\ens
\adbr{with $\lla_2 = \lla_1/128$,} proving the bound on~$|h_B(X)|$, with 
\eq\label{Dec-kappa0}
   \k_0 \Def 1536.
\en 

Next, from~\Ref{ADB-h_B-def}, we have
\eqa
   \lefteqn{h_B(X-\ej) - h_B(X)}   \label{ADB-h-diff-def}\\
  &&= \int_0^\infty (\pr[X\nud(t) \in B \giv X\nud(0) = X] - \pr[X\nud(t) \in B \giv X\nud(0) = X-\ej])\,dt. 
            \non
\ena
Taking $f(X) := \charI_B(X)$, Theorem~\ref{ADB-initial-displacement}, \adbr{with $G\uj$ bounded by~$\L$}, implies that
\eqs
   \lefteqn{\int_0^{2\lla_2^{-1}\log n} |\pr[X\nud(t) \in B \giv X\nud(0) = X] 
                   - \pr[X\nud(t) \in B \giv X\nud(0) = X-\ej]|\,dt }\\
    &&\Le   K^j_{\ref{ADB-initial-displacement}} \lla_2^{-1} n^{-1/2}
                 \{2\log n + 2\lla_2(\L g\uj)^{-1/2}\}\Bl \frac{\L}{g\uj} \Br^{1/4},
     \phantom{XXXXXXX}
\ens
if $\nS{X - nc} \le n\d/4$ and $n \ge \max\{n_{\Ref{Oct-n2.7-def}},\smh^{-1}(\d/2)\}$;
in performing the integration, the range is split at $t = (\L g\uj)^{-1/2}$.
The remainder of the integral is bounded by $10 \lla_2^{-1} n^{-1}$, as above,
if also $n \ge \max\{n_{\ref{ML-TV}},\smh^{-1}(\d/2)\}$, since $n_{\ref{ML-TV}} \ge n_{\Ref{Oct-n2.7-def}}$, 
completing the bound on $|h(X-\ej) - h(X)|$, and thence on~$\|\D h(X)\|_\infty$, with 
\eq\label{Dec-kappa1}
   \k_1 \Def  128 
                \Blb 10 + K_*\ui \Bl 2 + \frac{2\lla_2}{\sqrt{\L g_*}}\Br \Brb 
                      \Bl \frac{\Lbar}{g_*} \Br^{1/4}  \ \in\ \KKA,
\en
where $K_*\ui := \max_{1\le j\le d}K^j_{\ref{ADB-initial-displacement}}$.

For the second differences, the argument is entirely similar, using
Theorem~\ref{ADB-initial-displacement-2} for the bulk of the estimate, and bounding
the integrand by~$2$ for $0 \le t \le n^{-1}\bigl(\L g_{ij}^+\bigr)^{-1/2}$.  This gives the bound on~$\|\D^2 h(X)\|_\infty$,
with
\eq\label{Dec-kappa2}
    \k_2 \Def128\Blb 10 + \frac{2\lla_2}{\sqrt{\L g_*}} + K_*\ut \Bl\frac{\Lbar}{g_*}\Br^{1/2}
        \Bl 2 + \frac{\lla_2}{\sqrt{\L g_*}}\Br \Brb \ \in\ \KKA,
\en
where $K_*\ut := \max_{1\le i,j \le d} K^{ji}_{\ref{ADB-initial-displacement-2}}$.
\end{proof}

\begin{remark}\label{ADB-Poisson-case-4}
{\rm
As \adbb{in} Remark~\ref{ADB-Poisson-case}, \adbr{the dependence on~$d$, appearing through the
factors $d^{1/4}$ and~$d^{1/2}$
in the bounds on $\|\D h_B(X)\|_\infty$ and~$\|\D^2 h_B(X)\|_\infty$, respectively, is not needed
\adbb{if~$g_*^{-1}\max_{1\le j\le d}\{G\uj+g\uj\}$} remains bounded as $\nti$.}
\adbb{This is the case  if~$A = -\l I$, for some~$\l > 0$, and if $(X_n)$ is as in Theorem~\ref{ADB-MPP-A-too}.}
}
\end{remark}

\subsection{Reducing the generator}\label{ADB-reduction}
In this section, we show that \adbb{we can work with the simpler operator~$\ABA_n$, defined in~\Ref{ADB-approx-gen},
in place of the generator~$\AA\nud$.}   As a first step in the
reduction, we use two technical lemmas to bound the expectation
of the Newton remainder
\eq\label{ADB-Newton-remainder}
   e_2(W,J,h) \Def h(W+J) - h(W) - \D h(W)^T J - \half J^T \D^2h(W) J,
\en
for~$W$ a random vector.
For $\h > 0$ and $f\colon \Z^d \to \re$, we use the notation
\eq\label{ADB-norm-notation}
   \|f\|^\S_{n\h,\infty} \Def \max_{\nS{X - nc} \le n\h} |f(X)|,
\en
analogous to that of~\Ref{ADB-norm-E-defs}, but using $\nS{\cdot}$-balls,
with~$nc$ implicit.  Similarly, we write $\|\D h\|^\S_{n\h,\infty}$ and
$\|\D^2 h\|^\S_{n\h,\infty}$ for $\|f\|^\S_{n\h,\infty}$, when $f(X) = \|\D h(X)\|_\infty$
and $f(X) = \|\D^2 h(X)\|_\infty$, respectively, so that the conclusion of Theorem~\ref{ADB-h-bounds}
can be expressed as
\eqa
  && \|h_B\|^\S_{n\d/4,\infty} \Le \lla_1^{-1}\k_0\log n;\quad 
        \|\D h_B\|^\S_{n\d/4,\infty} \Le \lla_1^{-1}\k_1 d^{1/4} n^{-1/2}\log n;\non\\ 
  &&    \|\D^2 h_B\|^\S_{n\d/4,\infty} \Le \lla_1^{-1} \k_2 d^{1/2} n^{-1}\log n,  \label{ADB-bounds-on-h}
\ena
for $n \ge \max\{n_{\ref{ML-TV}},\smh^{-1}(\d/2)\}$ and $\d \le \d_1$.

Our control over the differences of the
functions~$h_B$ is only such that we can bound their $\|\cdot\|^\S_{n\h,\infty}$ norms
for suitable~$\h$, so these are the quantities that we need in our estimates.
For instance, if~$W$ is a random vector in~$\Z^d$ such that $\dtv(\law(W),\law(W+\ej)) \Le \e$,
we immediately have the bound
\[
    |\ex\{f(W+\ej) - f(W)\}| \Le 2\e \|f\|_\infty.
\]
Because we often only have control within certain (large) balls, we are led instead to
bounding a truncated quantity 
$$
   |\ex\{(f(W+\ej) - f(W))I[\nS{W-nc} \le n\h_1]\}|
$$ 
in terms of $\|f\|^\S_{n\h_2,\infty}$, for suitable choices of $\h_1$ and~$\h_2$.
The following lemma, proved in Section~\ref{L4.2}, provides what we need.

\begin{lemma}\label{ADB-Newton-truncation}
Suppose that~$W$ is a random vector in~$\Z^d$ such that
\[
   \dtv(\law(W),\law(W+\ej)) \Le \e_1, \ 1\le j\le d;\quad \pr[\nS{W-nc}> n\d/4] \Le \e_2. 
\]
Then, for $f\colon \Z^d\to\re$ and $1\le j\le d$, and for any $U,X \in \Z^d$ such that
$\max\{\nS{X},\nS{X+U}\} \le n\d/6$ and $n\d \ge 12\nS{U}$,
\eqs
    \leqn{ |\ex\{(f(W+X+U) - f(W+X))I[\nS{W-nc} \le n\d/3]\}| }\\
     && \quad \Le (\e_1\nl{U} + \e_2)\|f\|^\S_{n\d/2,\infty}.\phantom{XXXXXXXXXXX}
\ens
\end{lemma}

We use Lemma~\ref{ADB-Newton-truncation} to bound the Newton remainder defined in~\Ref{ADB-Newton-remainder}.  
Instead of bounding 
$e_2(W,J,h)$ directly, we bound a perturbation of it, 
\eq\label{Sept-E2-def}
   E_2(W,J,h) \Def e_2(W,J,h) + \frac12 \sum_{j=1}^d  J_j \D_{jj}h(W),
\en
which can be represented as a sum of third differences of~$h$.
The result, proved in Section~\ref{L4.3}, is as follows.
 
\begin{lemma}\label{ADB-Newton}
 If $W$ is a random vector in~$\Z^d$ such that 
\[
   \dtv(\law(W),\law(W+\ej)) \Le \e_1, \ 1\le j\le d;\quad \pr[\nS{W-nc} \ge n\d/4] \Le \e_2, 
\]
then, for any $J \in \Z^d$ such that $\nS{J} \le n\d/12$,
\eqa
  (i)\ &&\Bigl|\ex\Bigl\{ E_2(W,J,h) I[\nS{W-nc} \le n\d/3] \Bigr\}  \Bigr| \non\\
       && \quad \Le  \|\D^2 h\|^\S_{n\d/2,\infty} (C^{(1)}_{\ref{ADB-Newton}}(J)\e_1 + C^{(2)}_{\ref{ADB-Newton}}(J)\e_2), 
   \phantom{XXXXXXXXXXX}
    \non
\ena
where 
\eq\label{Newton-C-defs}
   C^{(1)}_{\ref{ADB-Newton}}(J) \Def \sixth \nl{J}(\nl{J}+1)(\nl{J}+2);\quad
   C^{(2)}_{\ref{ADB-Newton}}(J) \Def \half \nl{J}(\nl{J}+1).
\en
If the conditions are replaced by
$\pr[|W-nc| \ge n\d/4] \Le \e_2^E$ and $|J| \le n\d/12$, then
\eqa
   (ii)\ &&\Bigl|\ex\Bigl\{ E_2(W,J,h) I[|W-nc| \le n\d/3] \Bigr\}  \Bigr| \non\\
      && \quad \Le  \|\D^2 h\|_{n\d/2,\infty} (C^{(1)}_{\ref{ADB-Newton}}(J)\e_1 + C^{(2)}_{\ref{ADB-Newton}}(J)\e^E_2), 
    \phantom{XXXXXXXXXX}
    \non
\ena
where $\|\D^2 h\|_{n\d/2,\infty}$ is as defined in~\Ref{ADB-norm-E-defs}.
\end{lemma}

\begin{remark}\label{ADB-C(J)}
{\rm
 Note that, because $0 \neq J \in \Z^d$, we have
\[
    C^{(1)}_{\ref{ADB-Newton}}(J) \Le \nl{J}^3 \Le d^{3/2}|J|^3;\quad C^{(2)}_{\ref{ADB-Newton}}(J) \Le 2d|J|^2.
\]
}
\end{remark}

\medskip
Lemma~\ref{ADB-Newton} allows us to prove the following reduction theorem,
useful for approximating the generator of any Markov jump process
satisfying our general assumptions.

\begin{theorem}\label{ADB-generator-match}
Suppose that $(g^J,\,J\in\JJ)$, $c$, $A$, $\s^2$, $\g$, $\d_0$, $\AA\nud$ and~$\ABA_n$ are as 
in Sections \ref{Introduction} and~\ref{ADB-assumptions},
and that Assumptions G0--G4 are satisfied.  Suppose that~$W$ is a random
vector in~$\Z^d$, such that, for some $\e,v > 0$, 
\eq\label{new-4.6}
  \begin{array}{rl}
   {\rm (i)}& \ex\nS{W -nc}^2 \Le dvn;\\[1ex]
   {\rm (ii)}& \dtv(\law(W),\law(W+\ej)) \Le \e, \mbox{ for each } 1\le j\le d. 
  \end{array} 
\en
Then, for any $\d \le \d_0/\sqrt{\lmax(\S)}$ and for any $0 < \d' \le \d/2$, and for
$n \ge \max\{n_{\Ref{Oct-n2.7-def}},\smh^{-1}(\d/2)\}$,
\eqa
  \leqn{|\ex\{(\AA\nud h(W) - \ABA_n h(W))I[\nS{W-nc} \le n\d'/3]\}|} \non\\
 &\le& d^{5/2}\Lbar\Bigl( \half L_2 \lmax(\S) v\|\D h\|^\S_{n\d/4,\infty} \label{ADB-new-4.6}\\
 &&\mbox{} + n\|\D^2 h\|^\S_{n\d/4,\infty}\bigl\{ L_1 \adbr{\sqrt{v\lmax(\S)}}n^{-1/2} 
              + d^{1/2}(\gbar/\Lbar)\e + 32d^{1/2} v/\{n(\d')^2\}\bigr\} \Bigr),\non
\ena 
where $\|\D h\|^\S_{n\h,\infty}$
and $\|\D^2 h\|^\S_{n\h,\infty}$ are bounded in~\Ref{ADB-bounds-on-h}. 
\end{theorem}
 
\begin{proof}
Consider 
\eqa
   \AA\nud h(X)
  &:=& n\sJJ g_\d^J(X/n)\{h(X+J) - h(X)\} \label{ADB-new-4.6a} \\
  &=& n\sJJ \{g^J(c) + Dg^J(c)^T n^{-1}(X - nc) + e_1(X,J,g_\d^J)\}  \non\\
            &&\qquad\times   \{\D h(X)^T J + \half J^T \D^2h(X) J + e_2(X,J,h)\}, \non
\ena
where
\[
   e_1(X,J,g_\d^J) \Def g_\d^J(X/n) - g^J(c) - n^{-1} Dg^J(c)^T (X-nc),
\]
and~$e_2$ is as in~\Ref{ADB-Newton-remainder}.
Observing that $\sJJ g^J(c) \D h(X)^T J = \D h(X)^T F(c) = 0$, because~$F(c) = 0$, and that
\eqs
   &&\sJJ g^J(c) J^T \D^2h(X) J \Eq \tr\{\s^2 \D^2h(X)\};\\
   &&\sJJ Dg^J(c)^T n^{-1}(X - nc) \D h(X)^T J \Eq n^{-1}\tr\{A (X-nc)\D h(X)^T\},
\ens
it follows, once again writing $I\nuh(W) := I[\nS{W-nc} \le n\h/3]$, that
\eqa
  \lefteqn{|\ex\{(\AA\nud h(W) - \ABA_n h(W))I\nudd(W)\}|} \non\\
  &\le& \!\! n\ex\Blb \sJJ |e_1(W,J,g_\d^J)| \,|h(W+J) - h(W)|I\nudd(W)\Brb \non\\
  && \mbox{}  
    \!\!\!\!\! + \ex\Blb\sJJ |Dg^J(c)^T (W - nc)|\,|h(W+J) - h(W) - \D h(W)^T J|I\nudd(W)\Brb \non\\
  && \mbox{}\!\!\!\!\! + n\Bigl|\ex\Bigl\{\sJJ g^J(c) e_2(W,J,h)I\nudd(W)\Bigr\}\Bigr|.  
      \label{ADB-Taylor-parts}
\ena

Now, from~\Ref{ADB-L-defs}, and recalling~\Ref{Aug-cal-X-def}, it follows that,
for $X \in \dnud(J)$,
\eqa
  |e_1(X,J,g_\d^J)| &=& |g_\d^J(X/n) - g^J(c) - Dg^J(c)^T n^{-1}(X - nc)| \non\\
   &\le& 
           \half n^{-2}|X-nc|^2 L_2 g^J(c), \label{Aug-e1-bnd}
\ena
provided that $\d \sqrt{\lmax(\S)} \le \d_0$.
Since, for $\nS{X-nc} \le n\d'/3 \le n\d/6$ and $n\d/12 > \JmaxS$, we have 
$$
   \nS{X+J-nc} \Le \JmaxS + \nS{X-nc} \Le n\d/4,
$$
it follows, for such~$X$ and~$n$, that $|h(X+J) - h(X)| \le \nl{J}\,\|\D h\|^\S_{n\d/4,\infty}$ and
that~\Ref{Aug-e1-bnd} is satisfied.  Hence, since $\ex\nS{W-nc}^2 \le dvn$, we have
\eqs
    \lefteqn{n\ex\Blb \sJJ |e_1(W,J,g_\d^J)| \,|h(W+J) - h(W)|\,I\nudd(W)\Brb}\\
   &\le& \half L_2 \sJJ g^J(c) \nl{J}\,\|\D h\|^\S_{n\d/4,\infty} n^{-1}\ex|W-nc|^2  \\
   &\le& \half L_2 \sJJ g^J(c) \nl{J}\,\|\D h\|^\S_{n\d/4,\infty} dv \lmax(\S) \\
    &\le& \half L_2 v d^{5/2}\Lbar \lmax(\S)\,\|\D h\|^\S_{n\d/4,\infty}, 
\ens
if $n\d/12 > \JmaxS $, and this condition is satisfied for
$n \ge \max\{n_{\Ref{Oct-n2.7-def}},\smh^{-1}(\d/2)\}$.

Then, from~\Ref{ADB-first-difference} and~\Ref{ADB-L-defs}, if $n\ge n_{\Ref{Oct-n2.7-def}}$
and $\nS{X-nc} \le n\d'/3$, 
\eqs
  \lefteqn{\sJJ |Dg^J(c)^T (X - nc)|\,|h(X+J) - h(X) - \D h(X)^T J|\,I\nudd(X)} \\
          &&\Le \half L_1 |X-nc|\,\|\D^2 h\|^\S_{n\d/4,\infty} \sJJ g^J(c) \nl{J}(\nl{J}+1)
                   \phantom{XXXXXXXX} \\
          &&\Le d^2\Lbar L_1 |X-nc|\,\|\D^2 h\|^\S_{n\d/4,\infty};
\ens
hence, since  $\ex\nS{W-nc} \le \sqrt{dv n}$,
\eqs
   \lefteqn{\ex\Blb \sJJ |Dg^J(c)^T (W - nc)|\,|h(W+J) - h(W) - \D h(W)^T J|\,I\nudd(W)\Brb}\\
      && \Le d^{5/2}\Lbar L_1 \adbr{\sqrt{v\lmax(\S)n}}\,\|\D^2 h\|^\S_{n\d/4,\infty}.\phantom{XXXXXXXXXXXXXXXXXXX}
\ens
  Finally, from Lemma~\ref{ADB-Newton} and Chebyshev's inequality, 
\eqs
   \lefteqn{n\left|\ex\Blb \sJJ g^J(c) \Bl e_2(W,J,h) 
               + \half \sum_{j=1}^d  J_j \D_{jj}h(W)\Br I\nudd(W) \Brb \right|}\\
     &&\Le \sJJ g^J(c)  n\|\D^2 h\|^\S_{n\d/4,\infty}(C^{(1)}_{\ref{ADB-Newton}}(J)\e 
                          + 16C^{(2)}_{\ref{ADB-Newton}}(J)dv/\{n(\d')^2\}) \phantom{XXX} \\
     &&\Le  n\|\D^2 h\|^\S_{n\d/4,\infty}(d^{3}\gbar\e + 32d^3\Lbar v/\{n(\d')^2\}),
\ens
and, for each~$j$, $\sJJ J_j g^J(c) = 0$, because~$F(c)=0$.
This completes the proof of the theorem.
\end{proof}

\begin{remark}\label{Almost-MPP}
\adbr{
{\rm
 Suppose that processes~$(X_n,\,n\ge0)$ are {\it almost\/} density dependent, in that
they have rates as in~\Ref{ADB-transition-rates}, but with $g^J(x)$ being replaced by~$g^J(x,n)$, where 
$\lim_{n\to\infty}n^{1/2}\e\nud = 0$ for all $\d > 0$ small enough, where 
$\e\nud := \sup_{|x-c| \le \d}\sJJ |g^J(x,n) - g^J(x)|$.  Then it is immediate from~\Ref{ADB-new-4.6a}
that, if Assumptions G0--G4 hold for the transition rates~$g^J(x)$, then the conclusion of 
Theorem~\ref{ADB-generator-match} continues to hold, with the
limiting definitions of $c$, $A$ and~$\s^2$, provided only that an asymptotically small term
\[
      d\Lbar \{n^{1/2}\e\nud\} \, n^{1/2}\|\D h\|^\S_{n\d/4,\infty}
\]
is added to the bound given in~\Ref{ADB-new-4.6}.  In particular, if~$n\e\nud$ is bounded as $\nti$,
the asymptotic order of the error bound is not increased.
}
}
\end{remark}

\subsection{Total variation approximation}\label{ADB-first-TV}
We are now in a position to prove Theorem~\ref{ADB-first-approx-thm}, which gives a measure of the error
in the approximation of the distribution of a random vector~$W$ in $\Z^d$
by the distribution~$\Pi\nud$, if the process~$X\nud$ satisfies the special
assumptions of Section~\ref{ADB-special}.   The statement of
Theorem~\ref{ADB-first-approx-thm} is to some extent complicated by the presence of
the indicator truncating the range of~$W$ in the main condition~(iii).  The
truncation is necessary, because Theorem~\ref{ADB-h-bounds} only enables one to bound the
differences of the functions~$h_B$ solving the Stein equation~\Ref{ADB-Stein-eqn} in
balls of radius~$n\d/4$, for any $\d \le \d_0/3\sqrt{\lmax(\S)}$.

\begin{theorem}\label{ADB-first-approx-thm}
\adbr{
Given any $c\in\re^d$ and $d\times d$ matrices $A$ and~$\s^2$, with~$A$ having
eigenvalues all with negative real parts, and~$\s^2$ being positive definite, 
there exists a sequence of elementary processes $(X_n,\,n\ge1)$, 
given in Theorem~\ref{ADB-MPP-A-too}, satisfying
Assumptions G0, G1 and S2--S4 for $\d_0 = \lmin(\s^2)/(8\|A\|) > 0$,  having $F(c)=0$, 
$DF(c)=A$ and $\s^2$ given by~\Ref{ADB-sigma2-def}.  Let~$\S$ be as in~\Ref{ADB-Sigma-eqn}.}
Define $\td_0 := \min\{\adbr{1},\d_0/3\sqrt{\lmax(\S)}\}$ and $\Lbar := \lbar(\s^2)$, and suppose that 
$\d' \le \td_0/2$.
Then, for any~$v > 0$, there exists a constant $C_{\ref{ADB-first-approx-thm}}(v,\d')$, which is a 
function of~$v$, $\d'$, $\|A\|/\Lbar$ and the elements 
of~$\Sp'(\s^2/\Lbar)$ and~$\Sp'(\S)$, with the following property:
if~$W$ is any random vector in $\Z^d$ such that, for some $n \ge \max\{n_{\ref{ML-TV}},\smh^{-1}(\d')\}$
and for some $\e_1,\e_{20},\e_{21},\e_{22} > 0$,
$$
\begin{array}{rl}
 {\rm (i)}& \ex\nS{W -nc}^2 \Le dvn;\\[1ex]
 {\rm (ii)}& \dtv(\law(W),\law(W+\ej)) \Le \e_1, \mbox{ for each } 1\le j\le d;\\[1ex]
 {\rm (iii)}&  |\ex\{\ABA_n h(W)I[\nS{W-nc} \le n\d'/3]\}| \\
          &\qquad\Le \lbar(\s^2)(\e_{20}\|h\|^\S_{n\td_0/4,\infty} + \e_{21}n^{1/2}\|\D h\|^\S_{n\td_0/4,\infty}
                + \e_{22}n\|\D^2 h\|^\S_{n\td_0/4,\infty}),
\end{array}
$$
where $\ABA_n$ is as defined in~\Ref{ADB-approx-gen},
then, for any~$\d$ such that $2\d' \le \d \le \td_0$, 
\[
    \dtv(\law(W),\Pi\nud) \Le C_{\ref{ADB-first-approx-thm}}(v,\d')(d^3n^{-1/2}+ \adbr{d^4}\e_1
       +\e_{20} + d^{1/4}\e_{21} +d^{1/2}\e_{22})\log n .
\]
\end{theorem}

\begin{proof}
From~\Ref{ADB-Stein-este}, we have
\eqs
  \leqn{\dtv(\law(W),\Pi\nud)} \\
 && \Le \sup_{B \subset \tB_{n,\d}(c)}  
            |\ex\{\AA\nud h_B(W) I[W \in \tB_{n,\d'/3}(c)]\}| + \pr[W \notin \tB_{n,\d'/3}(c)],
\ens
where $\tB_{n,\h}(c) := \Z^d \cap B_{n\h,\S}(nc)$ and $h_B := h_{B,n}^\d$ is as for~\Ref{ADB-Stein-eqn}.
The probability in the second term is at most $9dv/\{n(\d')^2\}$, by~(i) and Chebyshev's inequality.  
Then, for $2\d' \le \d \le \td_0$ and for $n \ge \max\{n_{\ref{ML-TV}},\smh^{-1}(\d')\}$, we can 
use~\Ref{ADB-bounds-on-h} in~(iii), giving
\eqs
    \leqn{ |\ex\{\ABA_n h_B(W)\}I[\nS{W-nc} \le n\d'/3]| }\\
   &\le& \lbar(\s^2)(\e_{20}\|h_B\|^\S_{n\td_0/4,\infty} + \e_{21}n^{1/2}\|\D h_B\|^\S_{n\td_0/4,\infty} 
              + \e_{22}n\|\D^2 h_B\|^\S_{n\td_0/4,\infty}) \\
   &\le& (\lbar(\s^2)/\lla_1) (\e_{20}\k_0 + \e_{21}\k_1 d^{1/4} + \e_{22}\k_2 d^{1/2})  \log n.
\ens
Finally, from Theorem~\ref{ADB-generator-match} and~\Ref{ADB-bounds-on-h} \adbr{and recalling
that $L_2=0$ for elementary processes,}
for $n \ge \max\{n_{\ref{ML-TV}},\smh^{-1}(\d')\}$ and $2\d' \le \d \le \td_0$, we have
\eqs
   \leqn{|\ex\{(\AA\nud h_B(W) - \ABA_n h_B(W)) I[W \in \tB_{n,\d'/3}(c)]\}|}\\
   &\le& d^{3}\log n\, \frac{\lbar(\s^2)}{\lla_1}\Blb \frac{\k_2 L_1 \adbr{\sqrt{v\lmax(\S)}}}{\sqrt n}
          + \frac{32\k_2 d^{1/2}v}{n(\d')^{2}} + \k_2 d^{1/2}(\gbar/\Lbar)\,\e_1 \Brb ,
\ens
\adbr{and noting that $\gbar \le d^{-1/2}\Lbar \Jmax$  completes} the proof of the theorem.
\end{proof}

\section{Application: approximating a Markov jump process}\label{Sect-MPP}
Suppose that~$(X_n,\,n\ge1)$ is a fixed sequence of Markov jump processes with 
$X_n(\cdot) \in n^{-1} \Z^d$ for some fixed~$d$, and with transition 
rates determined by the fixed collection of functions $(g^J\colon \re^d \to \re_+,\,J\in\JJ)$, 
satisfying Assumptions G0--G4 for some $c$ and~$\d_0$. Then, for large~$n$, $X_n$ has a quasi-equilibrium
behaviour near~$nc$, in the sense that the process, if started near~$nc$, remains
within any ball $\tB_{n,\d}(c)$ for a length of time whose expectation, for fixed $\d>0$,
grows exponentially with~$n$.  During this time, its behaviour is asymptotically extremely close
to that of~$X\nud$, the choice of $\d < \d_0$ having almost no effect:  see Barbour \&~Pollett~(2012,
Section~4).  Thus it has a quasi-equilibrium distribution that, as $n\to\infty$, is asymptotically extremely 
close to~$\Pi\nud$, for any $0 < \d < \d_0$.  \adbr{Theorem~\ref{ADB-MPP-approx-thm} below shows that~$\Pi\nud$, 
in turn, can be closely approximated by the equilibrium distribution~$\hPi\nud$ of an elementary
process.}

We begin by noting that the variance of~$\Pi\nud$ is of the correct order, satisfying Condition~(i)
of Theorem~\ref{ADB-first-approx-thm},  and that Condition~(ii) is also satisfied, with $\e_1 = O(n^{-1/2})$.
The proofs of these two results are in Section~\ref{MPP-appx}.
Since, for this application, all the data of the problem, apart from~$n$,
are fixed, we can simplify many statements to order expressions as $n\to\infty$.

\begin{lemma}\label{ADB-equilibrium-variance-bound}
Let~$X_n$ be a Markov jump process whose transition rates are given in~\Ref{ADB-transition-rates}, 
satisfying Assumptions G0--G4 for some $\d_0 > 0$, and 
let~$\d_{\ref{ML-drift-lemma-1}}$ and~$\d'_{\ref{ML-drift-lemma-1}}(d)$ be as in 
Lemma~\ref{ML-drift-lemma-1}. 
Then,  for any 
$0 < \d \le \min\{\d_{\ref{ML-drift-lemma-1}},\d'_{\ref{ML-drift-lemma-1}}(d)\}$, 
if $X\nud \sim \Pi\nud$, we have
\[
    \ex \nS{X\nud-nc}^2 \Eq O(n).
\]
\end{lemma}

\begin{proposition} \label{ADB-lema2} 
Under Assumptions G0--G4, if $X\nud \sim \Pi\nud$ for some fixed
$\d \le \min\{\d_{\ref{ML-drift-lemma-1}},\d'_{\ref{ML-drift-lemma-1}}(d)\}$,
then, for each $1\le j\le d$,
$$
   \dtv(\Pi\nud,\Pi\nud*\e_{\ej}) \Eq O(n^{-1/2}),
$$
where $\e_J$ denotes the point mass at~$J$ and~$*$ denotes convolution.
\end{proposition}

We now give the approximation theorem.

\begin{theorem}\label{ADB-MPP-approx-thm}
\adbr{
Under the above assumptions on~$X_n$, there is a sequence of elementary processes~$\hX_n$ such that,
for any fixed $0 < \d < \min\{\d_0/\sqrt{\lmax(\S)},\td_0\}$, $\hX\nud$ has equilibrium 
distribution~$\hPi\nud$ satisfying 
\[
    \dtv(\Pi\nud,\hPi\nud) \Eq O(n^{-1/2}\log n),
\]
as $n\to\infty$, where~$\d_0$ is as in Assumptions G0--G4 for~$X_n$, and~$\td_0$ is as in 
Theorem~\ref{ADB-first-approx-thm} for~$\hX_n$.}
\end{theorem}

\begin{proof}
\adbr{
We apply Theorem~\ref{ADB-first-approx-thm}, using the elementary process~$\hX_n$,
given in Theorem~\ref{ADB-MPP-A-too}, that shares the same
$c$, $A$ and~$\s^2$, and hence the same~$\ABA_n$, as~$X_n$.
We show that, for $W \sim \Pi\nud$,  Conditions (i)--(iii) of Theorem~\ref{ADB-first-approx-thm} are satisfied,
with suitable choices of $v$, $\e_1$ and~$\e_{2l}$, $0\le l\le 2$. 
} 

Condition~(i) follows immediately for some $v > 0$ from
Lemma~\ref{ADB-equilibrium-variance-bound}, and  Condition~(ii) is implied by
Proposition~\ref{ADB-lema2}, with $\e_1 = O(n^{-1/2})$.  It then follows
from Theorem~\ref{ADB-generator-match} with $\d' = \d/2$ that, for any function~$h$, we have
\eqa
  \leqn{|\ex\{(\AA\nud h(W) - \ABA_n h(W))I[\nS{W-nc} \le n\d/6]\}|} \non\\
       && \Eq O\bigl(n^{-1/2}(n^{1/2}\|\D h\|^\S_{n\d/4,\infty} + n\|\D^2 h\|^\S_{n\d/4,\infty})\bigr).\label{MPP-1}
\ena
However, since~$\Pi\nud$ is the equilibrium distribution of~$X\nud$, it follows that $\ex\{\AA\nud h(W)\} = 0$.
Then, since $\nS{W-nc} \le n\d$ implies that $|W-nc| \le \d_0$, because $\d  < \d_0/\sqrt{\lmax(\S)}$, 
it follows that
\eqa
    \leqn{|\ex\{\AA\nud h(W) I[\nS{W-nc} > n\d/6]\}|} \non\\ 
         &&\Le n\sJJ |g^J|_{\d_0} \ex\{|h(W+J) - h(W)| I[\nS{W-nc} > n\d/6]\} \non\\
         &&\Le 2nL_0\L \|h\|_{n\d_0,\infty} \pr[\nS{W-nc} > n\d/6], \label{MPP-2}
\ena
where $L_0$, $\L$ and~$|\cdot|_\d$ are as in Section~\ref{ADB-assumptions}.
From Lemma~\ref{ADB-lemma1}, with $r=2$, $\pr[\nS{W-nc}> n\d/6] = O(n^{-2})$ as $n\to\infty$.
Then, since the left hand side of~\Ref{MPP-1} is unchanged if we set $h(X)=0$ for $\nS{X-nc} > n\d/4$,
provided that $\JmaxS \le n\d/12$, we can replace $\|h\|_{n\d_0,\infty}$ by $\|h\|^\S_{n\d/4,\infty}$ in~\Ref{MPP-2} for
all~$n$ sufficiently large.  These two observations imply, with~\Ref{MPP-2}, that
\[
   |\ex\{\AA\nud h(W) I[\nS{W-nc} \le n\d/6]\}| \Eq O(n^{-1}\|h\|^\S_{n\d/4,\infty}).
\]
Combining this with~\Ref{MPP-1}, it follows that
\eqs
   \lefteqn{|\ex\{\ABA_n h(W)I[\nS{W-nc} \le n\d/6]\}|} \\
    && \Eq O\bigl(n^{-1/2}(\|h\|^\S_{n\d/4,\infty}
             +  n^{1/2}\|\D h\|^\S_{n\d/4,\infty} + n\|\D^2 h\|^\S_{n\d/4,\infty})\bigr),
\ens
which in turn implies that Condition~(iii) of Theorem~\ref{ADB-first-approx-thm} is satisfied,
with $\e_{20},\e_{21}$ and $\e_{22}$ all of order $O(n^{-1/2})$, proving the result. 
\end{proof}

\adbr{
It is shown in Part~II, Theorem~2.3, that the equilibrium distributions of elementary processes are at
distance $O(n^{-1/2}\log n)$ in total variation from discrete normal distributions. The
theorem above thus extends this to the quasi-equilibrium distributions of very general \MPP es.
However, other equally explicit approximations may be available.  For example, consider the
bivariate immigration--death process~$X_n$ with immigration rates $n\a_1$ for a single type~$1$ individual,
$n\a_2$ for a single type~$2$ individual, and~$n\a_{12}$ for a pair with one of each type.
Assume that individuals have independent exponentially distributed lifetimes, with mean
$1/\m_i$ for type~$i$.   This process has equilibrium distribution~$\Pi_n := \law(N_1+N_3,N_2+N_3)$,
where $N_1, N_2$ and~$N_3$ are independent Poisson random variables with means
\[
   \ex N_1 \Eq \frac n{\m_1}\Bl \a_1 + \frac{\a_{12}\m_2}{\m_1+\m_2} \Br;\quad
   \ex N_2 \Eq \frac n{\m_2}\Bl \a_2 + \frac{\a_{12}\m_1}{\m_1+\m_2} \Br;\quad
   \ex N_3 \Eq \frac{n\a_{12}}{\m_1+\m_2}.
\]
It is then easy to show that, for any $\d > 0$, $\dtv(\Pi_n,\Pi_n(\d)) = O(n^{-1})$, where~$\Pi_n(\d)$
denotes the equilibrium distribution of the restriction of~$X_n$ to an $n\d$-ball around its mean
$n\hc$, given by $\hc := \bigl(\m_1^{-1}(\a_1 + \a_{12}),\m_2^{-1}(\a_2 + \a_{12})\bigr)^T$.
Taking any $a := (a_1,a_2)^T > -\hc$, and then translating~$X_n$ by $(\lfloor na_1 \rfloor,\lfloor na_2 \rfloor)^T$
(the integer parts are needed, to stay in~$\Z^2$), we obtain an almost density dependent \MPP~$\tX_n$
satisfying Assumptions G0--G4, with $\JJ := \{(1,0),(0,1),(1,1),(-1,0),(0,-1)\}$ and
\eqs
   g^{(1,0)}(x,n) &=& \a_1;\quad g^{(0,1)}(x,n) \Eq \a_2;\quad g^{(1,1)}(x,n) \Eq \a_{12}; \\
   g^{(-1,0)}(x,n) &=& \m_1(x_1 - n^{-1}\lfloor na_1 \rfloor);\quad 
           g^{(0,-1)}(x,n) \Eq \m_2(x_2 - n^{-1}\lfloor na_2 \rfloor),
\ens
having equilibrium distribution $\tPi_n := \Pi_n * \e_{(\lfloor na_1 \rfloor,\lfloor na_2 \rfloor)}$.
}

\adbr{
By Remark~\ref{Almost-MPP} and Theorem~\ref{ADB-first-approx-thm}, $\dtv(\tPi\nud,\hPi\nud) = O(n^{-1/2}\log n)$,
where~$\hX_n$ is the elementary process from Theorem~\ref{ADB-MPP-A-too}, having
\[
   c \Def a + \hc,\quad A \Def \left(\begin{matrix}
                         - \m_1 & 0 \\ 0 & - \m_2
                    \end{matrix} \right) \quad \mbox{and} 
     \quad \s^2 \Def \left(\begin{matrix}
                         2(\a_1 + \a_{12}) & \a_{12} \\ \a_{12} &  2(\a_2 + \a_{12})
                    \end{matrix} \right).
\]
Using Theorem~\ref{ADB-first-approx-thm} again, it follows that any other \MPP~$X'_n$
satisfying Assumptions G0--G4 with the same $c$, $A$ and~$\s^2$, restricted to any $n\d$-ball
around~$nc$, has an equilibrium distribution at distance of order $O(n^{-1/2}\log n)$ from~$\tPi\nud$,
and hence also from~$\tPi_n$.  This covers a wide range of processes, but by no means all.
For instance, it is easy to check that $0 \le \corr(X_n^1,X_n^2) \le 1/2$ for all positive
choices of the birth and death rates.
}

\section{Technicalities}\label{appendix}
\subsection{Proof of Lemma~\ref{ADB-Newton-truncation}}\label{L4.2}
In order to bound $|\ex\{(f(W+X+U) - f(W+X))I[\nS{W-nc} \le n\d/3]\}|$, we
write
\eqs
    \lefteqn{ (f(W+X+U) - f(W+X))I[\nS{W-nc} \le n\d/3]}\\
  &=& \!\!\tff_X(W + U) - \tff_X(W)\\
    &&\mbox{} \!\!\!\!+ f(W+X+U)\{I[\nS{W-nc} \le n\d/3] - I[\nS{W+U-nc} \le n\d/3]\},
\ens
where $\tff_X(Y) := f(Y+X)I[\nS{Y-nc} \le n\d/3]$.  Since~$|\tff_X(Y)| \le \|f\|^\S_{n\d/2,\infty}$ 
if $\nS{X} \le n\d/6$, it is immediate that
\[
   |\ex\{\tff_X(W + U) - \tff_X(W)\}| \Le \e_1 \nl{U}\|f\|^\S_{n\d/2,\infty}.
\]
Then, on the set
$$
  \{\nS{Y-nc} \le n\d/3 < \nS{Y+U-nc}\}\cup\{\nS{Y+U-nc} \le n\d/3 < \nS{Y-nc}\},
$$ 
it follows that 
$\nS{Y-nc} > n\d/3 - \nS{U} \ge n\d/4$ if $n\d \ge 12\nS{U}$, and that 
$$
    \nS{Y+X+U-nc} \Le n\d/3 + \max\{\nS{X},\nS{X+U}\} \Le n\d/2,
$$
this last by assumption.  Hence
\eqs
    \lefteqn{ |\ex\{f(W+X+U)(I[\nS{W-nc} \le n\d/3] - I[\nS{W+U-nc} \le n\d/3])\}| }\\
    &&\Le \e_2 \|f\|^\S_{n\d/2,\infty}, \phantom{XXXXXXXXXXXXXXXXXXXXXXXXXX}
\ens
and Lemma~\ref{ADB-Newton-truncation} is proved.  \hfill$\qed$

\subsection{Proof of Lemma~\ref{ADB-Newton}}\label{L4.3}
We prove only part~(i), showing that, if~$E_2(W,J,h)$ is as in~\Ref{Sept-E2-def}, then
\eqs
  \leqn{\Bigl|\ex\Bigl\{ E_2(W,J,h) I[\nS{W-nc} \le n\d/3] \Bigr\}  \Bigr|} \\
  &&\Le \|\D^2 h\|^\S_{n\d/2,\infty} (C^{(1)}_{\ref{ADB-Newton}}(J)\e_1 + C^{(2)}_{\ref{ADB-Newton}}(J)\e_2),
\ens
for constants $C^{(1)}_{\ref{ADB-Newton}}(J), C^{(2)}_{\ref{ADB-Newton}}(J)$ given in~\Ref{Newton-C-defs}, whenever
\[
   \dtv(\law(W),\law(W+\ej)) \Le \e_1, \ 1\le j\le d;\quad \pr[\nS{W-nc} \ge n\d/4] \Le \e_2; 
\] 
the proof of part~(ii) is entirely similar.
We begin by taking any function~$f\colon\Z^d\to\re$ and any $k\in\Z$ and $1\le j\le d$.
First, we note that, for $X \in \Z^d$, 
\eq\label{ADB-f-first-diff}
    f(X+k\ej) - f(X) \Eq \Blb  \begin{array}{ll}
           \sum_{l=1}^{k} \D_j f(X + (l-1)\ej) &\quad \mbox{if}\ k \ge 1; \\
           -\sum_{l=1}^{|k|} \D_j f(X - l\ej) &\quad \mbox{if}\ k \le -1.
                         \end{array}  \right.
\en
Hence, by considering positive and negative~$k$ separately, we find that 
\[
   |f(X+k\ej) - f(X) - k \D_jf(X)| \Le \half |k|(|k|+1) \|\D^2f\|_\infty.
\] 
For more general increments $J \in \Z^d$, we define 
\[
   J\us \Def (J_1,J_2,\ldots,J_s,0,0,\ldots,0),\quad s\ge1;\qquad J^{(0)} \Def (0,\ldots,0).
\]
Then, from the inequalities above, we have
\[
   |f(X+J\us) - f(X+J\usi) - J_s \D_s f(X+J\usi)| \Le \half|J_s|(|J_s|+1) \|\D^2f\|_\infty,
\]
and
\[
   |\D_s f(X+J\usi) - \D_s f(X)| \Le \nl{J\usi} \|\D^2f\|_\infty.
\]
Hence it follows that
\eqs
   \lefteqn{|f(X+J\us) - f(X+J\usi) - J_s \D_s f(X)|}\\
   &&\quad \Le \Blb \half|J_s|(|J_s|+1) + |J_s|\nl{J\usi}\Brb \|\D^2f\|_\infty. \phantom{XXXXX}
\ens
Adding over $1\le s\le d$, this gives
$$
   |f(X+J) - f(X) - Df(X)^TJ| \Le \half \nl{J}(\nl{J}+1) \|\D^2f\|_\infty.
$$
The same argument also shows that
\eq\label{ADB-first-difference}
   |f(X+J) - f(X) - Df(X)^TJ| \Le \half \nl{J}(\nl{J}+1) \|\D^2f\|^\S_{n\d/2,\infty},
\en
if $\nS{X - nc} \le n\d/3$ and $\nS{J} \le n\d/6$.

\ignore{
Note that then, for any function~$f$, and for $s\ge1$,
\eq\label{ADB-first-difference}
  f(X+J\us) - f(X+J^{(s-1)}) - \D_{s} f(X+J^{(s-1)})^T J_s 
     \Eq  \sum_{l=0}^{J_s-1}\{\D_s f(X + J^{(s-1)} + l\ej) - \D_j f(X+J^{(s-1)})\}, 
\en
if $J_s \ge 1$, whereas, if $J_s \le -1$,
\eq\label{ADB-first-difference-neg}
  f(X+J\us) - f(X+J^{(s-1)}) - \D_{s} f(X+J^{(s-1)})^T J_s \Eq  \sum_{l=1}^{|J_s|}
     \{\D_j f(X+J^{(s-1)}) - \D_s f(X + J^{(s-1)} - l\ej)\};
\en 
for $J_s = 0$, the left hand side of~\Ref{ADB-first-difference} is zero.
Taking expectations, this gives, in any of the three cases,
\[
   |\ex\{f(X+J\us) - f(X+J^{(s-1)}) - \D_{s} f(X+J^{(s-1)})^T J_s\}| 
          \Le \half |J_s|^2 \|\D^2 f\|_\infty\e.
\]
\eq\label{ADB-first-difference-ex}
   |\ex\{f(X+J) - f(X) - \D f(X)^T J\}| \Le \half \nl{J}^2 \|\D^2 f\|_\infty\e.
\en
}

We now prove Part~(i) of the lemma by induction on the number~$r$ of non-zero components of~$J$.
We write $I\nuh(X)$ as shorthand for $I[\nS{X - nc} \le n\h/3]$, for any $\h > 0$.
Starting with $r=1$, we consider three cases.
For $J = k\ej$ and $k \ge 1$, we have
\eqa
   \lefteqn{  h(X+k\ej) - h(X) - k\D_j h(X) - \half k(k-1)\D_{jj} h(X) }\non\\
    &&\Eq \sum_{l=1}^{k-1} \{\D_j h(X+l\ej) - \D_j h(X)\} - \half k(k-1)\D_{jj} h(X) \non\\
    &&\Eq \sum_{l=1}^{k-1} \sum_{r=1}^{l-1}\{\D_{jj} h(X+r\ej) - \D_{jj} h(X) \}.
          \label{ADB-1D}
\ena
From Lemma~\ref{ADB-Newton-truncation}, with $X=0$ and $U = r\ej$, it follows that
\eq\label{ADB-1D-result}
   |\ex\{(\D_{jj} h(W+r\ej) - \D_{jj} h(W))\,I\nud(W) \}| \le (r\e_1 + \e_2)\|\D^2 h\|^\S_{n\d/2,\infty}
\en
for $r \le k-2$,
if $\nS{k\ej} = |J|_1\nS{\ej} \le n\d/12$. 
Multiplying~\Ref{ADB-1D} by $I\nud(X)$, replacing $X$ by~$W$, then
taking expectations, invoking \Ref{ADB-1D-result}, and adding,
this yields the claim for $J = k\ej$ and $k\ge1$, with the upper bounds 
$C^{(1)}_{\ref{ADB-Newton}}(k\ej) =\sixth (k-2)(k-1)k$ and $C^{(2)}_{\ref{ADB-Newton}}(k\ej) = \half (k-1)k$.
If $J = k\ej$ and $k=0$, there is nothing to prove.  For $J = -k\ej$ and $k\ge1$,
we have
\eqa
   \lefteqn{  h(X-k\ej) - h(X) - (-k)\D_j h(X) - \half (-k)(-k-1)\D_{jj} h(X) }\non\\
    &&\Eq \sum_{l=1}^{k} \{\D_j h(X) - \D_j h(X-l\ej)\} - \half k(k+1)\D_{jj} h(X) \non\\
    &&\Eq \sum_{l=1}^{k} \sum_{r=1}^{l}\{\D_{jj} h(X-r\ej) - \D_{jj} h(X) \}.
          \label{ADB-1D-neg}
\ena
Arguing as before yields the claim for $J = -k\ej$ and $k\ge1$,
with 
$$
  C^{(1)}_{\ref{ADB-Newton}}(-k\ej) \Eq \sixth k(k+1)(k+2);\qquad C^{(2)}_{\ref{ADB-Newton}}(-k\ej) \Eq \half k(k+1),
$$
again if $\nS{k\ej} = \nl{J}\nS{\ej} \le n\d/12$.  
This establishes that the inequality~(i) is true for $r=1$, when~$J$ has just one non-zero component.

Now, for any $2\le r\le d$, we
assume that~(i) is true for all~$J$ with at most $r-1$ non-zero components,
and show that this implies that~(i) is also true for all~$J$ with at most~$r$ non-zero components.
Without loss of generality, we consider any~$J$ with~$J_j=0$ for $r < j \le d$.  
First, we write
\[
   h(X+J) - h(X) \Eq \{h(X+J) - h(X+J^{(r-1)})\} + \{h(X+J^{(r-1)}) - h(X)\}.
\]
The induction hypothesis gives
\eqs
   \lefteqn{\Bigl|\ex\Bigl\{ \Bigl(e_2(W,J^{(r-1)},h)  + \half \sum_{j=1}^{r-1}  J_j \D_{jj}h(W)\Bigr)\, 
              I\nud(W) \Bigr\}  \Bigr|} \\
       && \Le  \|\D^2 h\|^\S_{n\d/2,\infty} (C^{(1)}_{\ref{ADB-Newton}}(J^{(r-1)})\e_1+C^{(2)}_{\ref{ADB-Newton}}(J^{(r-1)})\e_2).\phantom{XX}
\ens
Thus it remains only to consider the expectation of the quantity
\eqs
    \leqn{\Bigl(h(W+J) - h(W+J^{(r-1)}) - J_r \D_r h(W)} \\
    &&\mbox{} - \half J_r^2 \D_{rr} h(W)
      - J_r \sum_{j=1}^{r-1} J_j\D_{rj} h(W) + \half J_r \D_{rr} h(W)\Bigr)\,I\nud(W).\phantom{XX}
\ens
The one dimensional result gives 
\eqs
   \lefteqn{\bigl|\ex\bigl\{\bigl( h(W+J) - h(W+J^{(r-1)}) - J_r \D_r h(W+J^{(r-1)})} \\
   &&\mbox{}\qquad \qquad\qquad  
          - \half J_r(J_r-1) \D_{rr} h(W+J^{(r-1)})\bigr) \,I\nud(W)\bigr\}\bigr| \\
   && \Le \sixth |J_r|(|J_r| + 1) \|\D^2 h\|^\S_{n\d/2,\infty}((|J_r|+2)\e_1+3\e_2),\phantom{XXXXX}
\ens
leaving an expectation involving the expression
\eqa
   \lefteqn{ J_r\Blb \D_r h(X+J^{(r-1)}) - \D_r h(X) - \sum_{j=1}^{r-1} J_j \D_{rj} h(X)\Brb }\non\\
    &&\mbox{}\qquad + \half J_r(J_r-1) \{\D_{rr} h(X+J^{(r-1)}) - \D_{rr} h(X)\}.\phantom{XXX}
     \label{ADB-final-term}
\ena
The first line in~\Ref{ADB-final-term} can be expressed as
\eq\label{ADB-term-1-sum}
   J_r \sum_{s=1}^{r-1} \{\D_r h(X + J\us) - \D_r(X + J\usi) - J_s \D_{rs} h(X)\}.
\en 
From~\Ref{ADB-f-first-diff} with $f := \D_d h$, we have
\eqs
  \lefteqn{ \D_r h(X + J\us) - \D_r h(X + J\usi) - J_s \D_{rs} h(X) }\\[1ex]
  &&=\ \Blb  \begin{array}{ll}
          \sum_{l=1}^{J_s} \{\D_{rs} h(X+J\usi + (l-1)e^{(s)}) - \D_{rs} h(X)\},
                           &\quad {\rm if}\ J_s \ge 1;\\[1.5ex]
          0,                &\quad {\rm if}\ J_s = 0;\\[1ex]
          - \sum_{l=1}^{|J_s|} \{\D_{rs} h(X+J\usi - le^{(s)}) - \D_{rs} h(X)\},
          &\quad {\rm if}\ J_s \le -1.
              \end{array} \right.
\ens
Hence, multiplying by $I\nud(X)$, taking expectations with~$W$ in place of~$X$ and using 
Lemma~\ref{ADB-Newton-truncation}, it follows for the 
first line in~\Ref{ADB-final-term} that 
\eqs
   \lefteqn{|J_r|\left|\ex\Blb \sum_{s=1}^{r-1} \{\D_r h(W + J\us) - \D_r(W + J\usi) 
         - J_s \D_{rs} h(W)\}\,I\nud(W) \Brb \right| }\\
    &\le& |J_r| \sum_{s=1}^{r-1} |J_s| \|\D^2 h\|^\S_{n\d/2,\infty}\{(\nl{J\usi} + \half|J_s|)\e_1+\e_2\}
           \phantom{XXXXXXXX}\\
    &\le&   |J_r|\nl{J^{(r-1)}} \|\D^2 h\|^\S_{n\d/2,\infty}(\half\nl{J^{(r-1)}}\e_1+\e_2).
\ens
The second line in~\Ref{ADB-final-term} is directly bounded using 
Lemma~\ref{ADB-Newton-truncation}, giving
\eqs
   \lefteqn{\half |J_r|(|J_r|+1)\bigl|\ex\{(\D_{rr} h(W+J^{(r-1)}) - \D_{rr} h(W))\,I\nud(W)\}\bigr| }\\
    &&\Le \half |J_r|(|J_r|+1) \|\D^2 h\|^\S_{n\d/2,\infty}(\nl{J^{(r-1)}}\e_1+\e_2) .\phantom{XX}
\ens
This establishes the inequality~(i) for~$J$ with $J_j=0$ for $r < j \le d$, 
since it is easily checked that
\[
  C^{(1)}_{\ref{ADB-Newton}}(J) \ \ge\ C^{(1)}_{\ref{ADB-Newton}}(J^{(r-1)}) + \half |J_r|\,\nl{J^{(r-1)}}^2
           + \half |J_r| (|J_r|+1)\,\nl{J^{(r-1)}},
\]
and that
\[
  C^{(2)}_{\ref{ADB-Newton}}(J) \ \ge\ C^{(2)}_{\ref{ADB-Newton}}(J^{(r-1)}) +  |J_r|\,\nl{J^{(r-1)}}
           + \half |J_r| (|J_r|+1).
\]
The lemma now follows by induction. \hfill$\qed$

\subsection{Proofs of Lemma~\ref{ADB-equilibrium-variance-bound} and Proposition~\ref{ADB-lema2}}\label{MPP-appx}
To prove Lemma~\ref{ADB-equilibrium-variance-bound}, we need to show that, 
if $X\nud \sim \Pi\nud$ for some $\d \le \d_0/\sqrt{\lmax(\S)}$, then $\ex\{\nS{X\nud - nc}^2\} = O(n)$.
Now, by Dynkin's formula, we have $\ex\{\AA\nud h(X\nud)\} = 0$ for any choice of~$h$.
Take $h(X) = h_0(X) = \nS{X-nc}^2$ as in Lemma~\ref{ML-drift-lemma-1},
for which $\AA\nud h_0(X) \le -\lla_1 h_0(X)$ in 
$\nS{X-nc} \ge K_{\ref{ML-drift-lemma-1}}\sqrt{nd}$.  Then, from~\Ref{ML-first-drift-bound}, for all 
$\nS{X-nc} \le n\,\min\{\d_{\ref{ML-drift-lemma-1}},\d'_{\ref{ML-drift-lemma-1}}(d)\}$,
$\AA\nud h_0(X) \le n dK_1'$, where, from \Ref{ADB-trace-inequality}, 
and~\Ref{ADB-finally}, 
$$
    K_1' \Def L_0 d^{-1}\tr(\s^2_\S) \le L_0  \lmax(\s^2_\S) \Le L_0 \adbr{\lla_1}\Rh(\s^2)\Rh(\S).
$$
This implies that
\[
    0 \Eq \ex\{\AA\nud h_0(X\nud)\} \Le -\lla_1\ex\{\nS{X\nud - nc}^2 I[\nS{X-nc} \ge K_{\ref{ML-drift-lemma-1}}\sqrt{nd}]\}
          + ndK_1'.
\]
Since also, using~\Ref{ADB-K1-def}, 
\[
   \ex\{\nS{X\nud - nc}^2 I[\nS{X-nc} < K_{\ref{ML-drift-lemma-1}}\sqrt{nd}]\} 
                   \Le ndK_{\ref{ML-drift-lemma-1}}^2  \Le 4ndL_0 \Rh(\s^2)\Rh(\S),
\]
the claim is proved.   $\hfill\qed$

\medskip
To prove Proposition~\ref{ADB-lema2}, we start with a concentration bound,
used to handle the truncation.

\begin{lemma}\label{ADB-lemma1}
Define $K_\S := 2\Lbar K_{\ref{ADB-lema3}} / (d\th_1\lla_1) \in \KK$.
Under Assumptions G0--G4, for any 
$0 < \d \le \min\{\d_{\ref{ML-drift-lemma-1}},\d'_{\ref{ML-drift-lemma-1}}(d)\}$, 
$n \ge n_{\ref{ML-drift-lemma-1}}$ and $\h > K_{\ref{ML-drift-lemma-1}}\sqrt{d/n}$, 
$$
   \Pi\nud\{\nS{X-nc}> n\h \} \ \leq\   \h^{-2}d^2K_\S e^{-n\th_1\h^2}.
$$ 
In particular, for any fixed~$\h > 0$, $\Pi\nud\{\nS{X-nc}> n\h \} = O(n^{-r})$ as $n\to\infty$,
for any $r \ge 1$.
\end{lemma}

\begin{proof}
Again, for $\d \le \d_0/\sqrt{\lmax(\S)}$ and $X\nud \sim \Pi\nud$, we have $\ex\{\AA\nud h(X\nud)\} = 0$ 
for any choice of~$h$,
by Dynkin's formula. Take $h(X) = h_{\th_1}(X)$  as in Lemma~\ref{ML-drift-lemma-1},
for which, from~\Ref{ML-second-MG} and for $n \ge n_{\ref{ML-drift-lemma-1}}$, 
$$
   \AA\nud h_{\th_1}(X) \Le -\half n^{-1} \lla_1 \th_1 h_0(X)h_{\th_1}(X) \quad \mbox{for}\quad 
          \nS{X-nc} \ge K_{\ref{ML-drift-lemma-1}}\sqrt{nd}.
$$
This, with~\Ref{ADB-H-th-bound}, implies that 
\[
   \half n^{-1}\lla_1 \th_1 
   \ex\bigl\{h_0(X)h_{\th_1}(X) I\bigl[\nS{X-nc} \ge K_{\ref{ML-drift-lemma-1}}\sqrt{nd}\bigr]\bigr\}
       \Le n\L K_{\ref{ADB-lema3}}.
\]
Hence, for $\h > K_{\ref{ML-drift-lemma-1}}\sqrt{d/n}$, it follows that
\[
    \half n^{-1}\lla_1 \th_1\, (n\h)^2 e^{n^{-1}\th_1(n\h)^2} \Pi\nud\{\nS{X-nc}> n\h \} \Le n\L K_{\ref{ADB-lema3}},
\]
proving the first part.  The second is then immediate.
\end{proof}

The proof of the next lemma is rather close to that of Theorem~\ref{ADB-initial-displacement}, so
we only give a quick sketch.

\begin{lemma}\label{ADB-thdtvpi} 
Under Assumptions G0--G3, for any 
fixed $\d \le \min\{\d_{\ref{ML-drift-lemma-1}},\d'_{\ref{ML-drift-lemma-1}}(d)\}$, and 
for any $J\in\JJ$,
$$
   d_{TV}\{\Pi\nud,\Pi\nud*\e_J\} \Eq O(n^{-1/2})
$$ 
as $n\to\infty$,
where $\e_J$ denotes the point mass at~$J$ and~$*$ denotes convolution.
\end{lemma} 

\begin{remark}\label{Dec-displacements-rk}
{\rm
 Note that we cannot directly replace $J$ by~$\ej$ here, to obtain Proposition~\ref{ADB-lema2}, because~$\ej$ may not
belong to~$\JJ$.  Under Assumption~G4, we can do so: see~\adbr{\Ref{ADB-Sept-mixing}} below.
}
\end{remark}

\begin{proof} 
Fix any~$U>0$, and use the 
stationarity of~$\Pi\nud$ to give the inequality
\begin{eqnarray}
 d_{TV}\{\Pi\nud,\Pi\nud*\e_J\}
 &\le& \!\!\sum_{X \in \Z^d} \Pi\nud(X)\, d_{TV}\{\law_X(X\nud(U)),\law_X(X\nud(U)+J)\},\phantom{XXx} \label{ADB-disttv}
\end{eqnarray}
where $\law_X$ denotes distribution conditional on $\{X\nud(0)=X\}$.
By Lemma~\ref{ADB-equilibrium-variance-bound},  it then follows that, for any 
$\d \le \min\{\d_{\ref{ML-drift-lemma-1}},\d'_{\ref{ML-drift-lemma-1}}(d)\}$, 
\eq\label{ADB-bound-1}
  d_{TV}\{\Pi\nud,\Pi\nud*\e_J\} 
       \Eq D_{Jn}(\d/2) + O(n^{-1}),
\en
where
\eqs
D_{Jn}(\d') \ :=\ 
  \sum_{X\colon\nS{X-nc} \le n\d'}\Pi\nud(X)\, d_{TV}\{\law_X(X\nud(U)), 
     \law_X(X\nud(U)+J)\}.
\ens
This alters our problem to one of finding a bound of similar form, but now
involving the transition probabilities of the
process~$X\nud$ over a finite time~$U$, started in any state~$X$ which is reasonably 
close to~$nc$.  
\ignore{
\adbg{Our argument is carried through for} \adbb{$n\ \ge\ \max\{n_{\ref{ADB-thdtvpi}}(1/\m_0^J),\smh^{-1}(\d)\}$, where
\eq\label{ADB-n-limited}
   n_{\ref{ADB-thdtvpi}}(T) \Def \max\Blb (5(d^{-1}\JmaxS)\max\{1,\sqrt{d\th_1}\})^{8/3}, 
     n_{\ref{ADB-lema3}}(T)\Brb \ \in\ \KK(\Lbar T),
\en 
to be compared with~$n_{\Ref{Oct-n2.7-def}}$.} 
}

By Assumption~G3, $J$-jumps occur in~$X\nud$ with rate at least~$n\m_0^J$,
whenever it is in the set~$\XX\nud(J)$.
Thus, by analogy with~\Ref{ADB-bivariate-process-2}, we can realize the chain~$X\nud$ with $X\nud(0) = X_0$ 
in the form $X\nud(u) := X_0 + JN\nud(u) + W\nud(u)$, where 
the transition $(l,W) \rightarrow (l+1,W)$ occurs at rate $n\m_0^J$.
This leads to a decomposition
\eqa
    \leqn{d_{TV}\{\law_X(X\nud(U)), \law_X(X\nud(U)+J)\} } \non\\
   &\le&  \frac{1}{2}\sum_{l \geq 0}
       |\pr_{X_0}[N\nud(U)=l]-\pr_{X_0}[N\nud(U)=l-1]| \label{ADB-Poisson-part-2a}\\
  && \ \mbox{}+\frac{1}{2}\sum_{X \in \Z^d}\sum_{l \geq 1}\pr_{X_0}[N\nud(U)=l-1]
      |q^{U}_{l,X_0}(X-lJ)-q^U_{l-1,X_0}(X-lJ)|, \non
\ena
where $q^{U}_{l,X}(W)$ is as defined in~\Ref{ADB-dens1-S}.

Much as for~\Ref{ADB-first-part-bnd-S}, and using Lemma~\ref{ADB-lema3}, 
the first sum in~\Ref{ADB-Poisson-part-2a} is bounded by 
\eq\label{Dec-first-contn}
   \pr_{X_0}[\hht\nud \le U] + \{n\m_0^J U\}^{-1/2} \Eq O(n^{-1/2}), 
\en
if we choose $U = 1/\sqrt{\L\m_0^J}$, where~$\hht\nud$ is as defined in~\Ref{tau-hat-def-S}.  
For the second part of~\Ref{ADB-Poisson-part-2a}, we argue as
in the proof of Theorem~\ref{ADB-initial-displacement}, using the Radon--Nikodym
derivative $\rstar(u,w^u) := d\bPstar/d\bP^U_{\bs_{l-1},X}(w^u)$.  As long as
$\rstar(u,w^u) \le 2$ and $u \le \hht_\d$, the $\bP^U_{\bs_{l-1},X}$-martingale $\rstar(u,w^u)$ makes jumps of size
at most $2|J|L_1/(n\e_0)$, and this enables the quadratic variation of
the stopped martingale to be 
controlled by \adbr{$n^{-1}|J|^2 4L_0(L_1/\e_0)^2\L u$}, as in \Ref{ADB-QV-S}.  
Choosing $U = 1/\sqrt{\L\m_0^J}$ once more, and arguing as for \Ref{ADB-phi-bound-S} and~\Ref{ADB-single-rho-result-S},
the second part in~\Ref{ADB-Poisson-part-2a} is also shown to be of order $O(n^{-1/2})$. 
Combining these observations with~\Ref{ADB-bound-1},
the lemma follows.
\end{proof}

To deduce Proposition~\ref{ADB-lema2} from Lemma~\ref{ADB-thdtvpi}, take any $1\le j\le d$. 
Then it is immediate from the triangle inequality that, because
$\sum_{l=1}^{r(j)} J_l\uj = \ej$ for $J_1\uj,\ldots,J_{r(j)}\uj$ as given
in Assumption~G4, we also have
\eq\label{ADB-Sept-mixing}
   \dtv(\Pi\nud,\Pi\nud*\e_{\ej}) \Le \sum_{l=1}^{r(j)} \dtv\bigl(\Pi\nud,\Pi\nud*\e_{J_l\uj}\bigr)
               \Eq O(n^{-1/2}).  
\en

\adbr{
\begin{remark}\label{ADB-J-jumps-in-Th3.1}
{\rm
Replacing~$\ej$ by any $J\in\JJ$ in the statement of
Theorem~\ref{ADB-initial-displacement}, the corresponding conclusion can be established, by adapting the proof 
much as for Lemma~\ref{ADB-thdtvpi} above,
for sequences of Markov jump processes satisfying Assumptions G0--G4.
In the bounds, $K_{\ref{ADB-initial-displacement}}^J$ depends on~$J$ through a factor of~$|J|$, 
and $G\uj$ is replaced by~$\L$.
}
\ignore{
This result is used, for instance, in proving cut-off in the convergence to equilibrium of
Markov jump processes in Barbour, Brightwell \& Luczak~(2016).
}
\end{remark}
}

\end{document}